\documentclass[a4paper, 11pt, reqno]{amsart}

\numberwithin{equation}{section}

\usepackage{amssymb}
\usepackage{mathrsfs}
\usepackage{dsfont}
\usepackage{stmaryrd, MnSymbol}
\usepackage{enumerate, xspace}
\usepackage{changebar}
\usepackage{comment}
\usepackage{color}
\usepackage{mathtools}

\newtheorem{theorem}{Theorem}
\newtheorem{lemma}{Lemma}[section]

\newtheorem{proposition}{Proposition}[section]

\newtheorem{remark}{Remark}[section]

\numberwithin{equation}{section}

\renewcommand{\a}{\alpha}
\renewcommand{\b}{\beta}

\def\R{{\mathbb{R}}}
\def\N{{\mathbb{N}}}
\def\Z{{\mathbb{Z}}}
\def\T{{\mathbb{T}}}

\begin{document}

\title{A family of fractional diffusion equations derived from stochastic harmonic chains with long-range interactions}

\author[H.~Suda]{Hayate Suda}
\address{Graduate School of Mathematical Sciences, University of Tokyo, 3-8-1, Komaba, Meguro-ku, Tokyo, 153--8914, Japan}
\email{hayates@ms.u-tokyo.ac.jp}

\begin{abstract}
We consider one-dimensional infinite chains of harmonic oscillators with stochastic perturbations and long-range interactions which have polynomial decay rate $|x|^{-\theta}, x \to \infty, \theta > 1$, where $x \in \Z$ is the interaction range. We prove that if $2< \theta \le 3$, then the time evolution of the macroscopic thermal energy distribution is superdiffusive and governed by a fractional diffusion equation with exponent $\frac{3}{7-\theta}$, while if $\theta > 3$, then the exponent is $\frac{3}{4}$. The threshold is $\theta = 3$ because the derivative of the dispersion relation diverges as $k \to 0$ when $\theta \le 3$.
\end{abstract}

\maketitle

\section{Introduction}

\subsection{Background : Exponential decay model}

Chains of oscillators are classical microscopic systems which are used widely to understand the macroscopic behavior of energy. In late 90's, anomalous heat transport in one-dimensional unpinned FPU-Chains is numerically observed \cite{LLPb}, and then many groups started to investigate this phenomenon in various ways, see the reviews \cite{D,Ls,LLP}. Since the mathematical analysis of such nonlinear deterministic systems is out of reach of current techniques, the problem has been studied in models where the nonlinealities are replaced by random exchange of momenta which conserve total energy and total momentum, see the review \cite[Chapter 5]{Ls}. The model is defined as follows. Denote by $(p_{x}(t),q_{x}(t)) \in \R \times \R$ the momentum and the position of the particle labeled by $x \in \Z$ at time $t \ge 0$. Their stochastic dynamics is given by the following stochastic differential equation : 
\begin{align*}
\begin{cases}
d q_{x}(t) = p_{x}(t) dt \\
d p_{x}(t) = \{ - (\a * q)_{x}(t) - \frac{\gamma}{2} (\b * p)_{x}(t) \} dt + \sqrt{\gamma} \sum_{z = -1, 0, 1} ( Y_{x+z}p_{x}(t) ) dw_{x+z},
\end{cases}
\end{align*}
where $\a : \Z \to \R$ is the interaction potential which satisfies

$(a.1) ~ \a_{x} \le 0 $ for all $x \in \Z \setminus \{ 0 \}$, $\a_{x} \neq 0$ for some $x \in \Z$.

$(a.2) ~ \a_{x} = \a_{-x} $ for all $ x \in \Z.$

$(a.3) $ There exists some positive constant $C>0$ such that $|\a_{x}| \le Ce^{-\frac{|x|}{C}} $ for all $x \in \Z$.

$(a.4) ~ \widehat{\a}(k) >0 $ for all $k \neq 0$ , $\widehat{\a}(0) = 0 , \widehat{\a}''(0) > 0$.

\noindent Here $\widehat{\a}$ is the discrete Fourier transform defined as
\begin{align*}
\widehat{\a}(k) := \sum_{x \in \Z} \a_{x} e^{-2 \pi \sqrt{-1}x k} , \quad k \in \T, 
\end{align*}
and $\T$ is the one-dimensional torus. $\gamma > 0$ is the strength of the noise and $Y_{x}, x \in \Z$ are vector fields defined as
\begin{align*}
Y_{x} := (p_{x} - p_{x+1})\partial_{p_{x-1}} + (p_{x+1} - p_{x-1})\partial_{p_{x}} + (p_{x-1} - p_{x})\partial_{p_{x+1}},
\end{align*} 
$\{ w_{x}(t) ; x \in \Z, t \ge 0 \} $ are i.i.d. one-dimensional standard Brownian motions, and $\b_{x}$ is defined as
\begin{align*}
\b_{x} := \begin{cases} 6 , &x = 0, \\
-2 , &x = \pm 1, \\ 
-1 , &x = \pm 2, \\ 
0 , &\text{otherwise}. \end{cases}
\end{align*}
Denote by $r_{x}(t), e_{x}(t)$ the inter-particle distance and the energy of particle labeled by $x \in \Z$ defined as
\begin{align*}
r_{x}(t) &:= q_{x}(t) - q_{x-1}(t), \\
e_{x}(t) &:= \frac{1}{2} |p_{x}|^{2} - \frac{1}{4} \sum_{x' \in \Z, x' \neq x} \a_{x-x^{'}} |q_{x} - q_{x'}|^{2}. 
\end{align*}
The divergence of the thermal conductivity for this model was proved in \cite{BBO}. In \cite{KO2}, the authors showed that the stochastic perturbation decouple the phononic (mechanical) energy from the thermal energy. The former converges at ballistic scaling. Actually, the following convergence of the scaled empirical measure of $(p_{x}(t), r_{x}(t), e_{x}(t) )$ holds:
\begin{align*}
\lim_{\epsilon \to 0} \epsilon \sum_{x \in \Z} \begin{pmatrix} p_{x}(\frac{t}{\epsilon}) \\ r_{x}(\frac{t}{\epsilon}) \\ e_{x}(\frac{t}{\epsilon})  \end{pmatrix} J(\epsilon x) = \int_{\R} dy ~ \begin{pmatrix} \bar{p}(y,t) \\ \bar{r}(y,t) \\ \bar{e}_{ph}(y,t) \end{pmatrix} J(y)
\end{align*}
for any test function $J$ where $(\bar{p},\bar{r},\bar{e}_{ph})$ is the solution of the linear wave equation:
\begin{align*}
\begin{cases} \partial_{t} \bar{r}(y,t) = \partial_{y} \bar{p}(y,t) \\ \partial_{t} \bar{p}(y,t) = \frac{\widehat{\a}^{''}(0)}{8\pi^{2}} \partial_{y} \bar{r}(y,t) \\ \partial_{t} \bar{e}_{ph}(y,t) = \frac{\widehat{\a}^{''}(0)}{8\pi^{2}} \partial_{y}(\bar{p}\bar{r})(y,t). \end{cases}
\end{align*}
Notice that the stochastic perturbation does not explicitly affect the time evolution of the phononic energy. On the other hand, the Boltzmann-type equation \textcolor{black}{is obtained} as the time evolution law of the microscopic thermal energy distribution:
\begin{align*}
&\partial_{t} W_{\epsilon}(y,k,t) + \frac{\epsilon \omega^{'}(k)}{2\pi} \partial_{y} W_{\epsilon}(y,k,t) = \gamma [\mathcal{L} W_{\epsilon}(y,\cdot,t)](k) + o_{\epsilon}(1), \\
& \quad (\mathcal{L} f)(k) = \int_{\T} dk ~ R(k,k^{'}) (f(k^{'}) - f(k)),
\end{align*}
where local spectral density of energy $W_{\epsilon}$ depends on the position $y \in \R$ along the chain, the wave number $k \in \T = [-\frac{1}{2}, \frac{1}{2})$ and time $t \ge 0$. $\omega(k) = \sqrt{\widehat{\a}(k)}$ is the dispersion relation and $\mathcal{L}$ is a scattering operator on $\T$. In \cite{BOS,KJO}, under the weak noise assumption $\gamma = \epsilon \gamma_{0}$, the authors showed that the scaled solution of the Boltzmann equation converges to a soluiton of a fractional diffusion equation with the exponent $\frac{3}{4}$ by a two-step procedure. \textcolor{black}{As a first step, by taking a kinetic limit with time scale $\frac{t}{\epsilon}$, the so-called Boltzmann equation is derived:}
\begin{align*}
\partial_{t} W(y,k,t) + \frac{\omega^{'}(k)}{2\pi} \partial_{y} W(y,k,t) = \gamma_{0} [\mathcal{L} W(y,\cdot,t)](k).
\end{align*}
\textcolor{black}{As a second step}, they consider a limit of the rescaled energy distribution $\{ W_{N}(y,k,Nt) \}_{N \in \N}$ defined as the solution of 
\begin{align*}
\partial_{t} W_{N}(y,k,t) + \frac{\omega^{'}(k)}{2N^{\frac{2}{3}} \pi} \partial_{y} W_{N}(y,k,t) = \gamma_{0} [\mathcal{L} W_{N}(y,\cdot,t)](k).
\end{align*}
Thanks to the scattering effect $\mathcal{L}$, the resulting limit $u(y,t) := \lim_{N \to \infty} W_{N}(y,k,Nt)$ is homogenized on $\T$ and satisfies $\partial_{t} u(y,t) = - c_{\a,\gamma_{0}} (-\Delta)^{\frac{3}{4}} u(y,t)$. More recently, in \cite{JKO} the authors proved a direct limit to the fractional diffusion from the microscopic model with the stronger noise $\gamma = \epsilon^{s} \gamma_{0} , 0 \le s < 1$, and the time evolution of the direct limit $\{ W(y,t) ; y \in \R, t \ge 0 \}$ is governed by $\partial_{t} W(y,t) = - C_{\a,\gamma_{0}} (-\Delta)^{\frac{3}{4}} W(y,t)$. Note that they \textcolor{black}{did} not prove whether $c_{\a,\gamma_{0}} = C_{\a,\gamma_{0}}$ or not. In general, the two-step limit for an anharmonic chain does not coincide with the direct limit of that \cite{MA,S}. 

Since the above scaling limit results agree with numerical simulations and theoretical prediction by H. Spohn \cite{S} about FPU-Chains, the stochastic harmonic chain is considered to be a good approximation of some nonlinear chain. In this way, though the exponential-decay interaction potential may have infinite range, the macroscopic behavior of energy is essentially same with the nearest neighbor model ($\a_{0} = 2, \a_{\pm 1} = -1, \a_{z} = 0 , |z| \ge 2$) and the effect of microscopic long-range interaction on the time evolution of macroscopic energy distribution remains unclear. Hence a natural generalization is to study a model which has slower decay rate. 

\subsection{Polynomial decay model}

In the present study, we consider \textcolor{black}{stochastically perturbed harmonic chains} which have polynomial-decay rate interaction potentials, that is,
\begin{align*}
\a_{x} := - |x|^{- \theta} , \quad x \in \Z \setminus \{ 0 \} \quad \a_{0} := 2 \sum_{x \in \N} |x|^{- \theta} \quad \theta > 1.
\end{align*}
Notice that our interaction potentials do not satisfy the condition $(a.3)$. When $\theta \le 3$, $(a.4)$ is not satisfied because \textcolor{black}{$\widehat{\a}^{''}(k)$ is not a continuous function on $\T$.} Our stochastic perturbation is same with exponential decay model. Following the idea of \cite{BOS,KJO} and \cite{JKO}, we show the direct limit as Theorem \ref{main0} \textcolor{black}{and also} the two-step limit as Theorem \ref{main1} and \ref{main2}. The time evolution of the macroscopic thermal energy is governed by a fractional diffusion equation and the exponent of the fractional diffusion changes according to the value of $\theta$:
\begin{align}
\partial_{t} W(y,t) &= \begin{cases} - (2\pi)^{- \frac{6}{7 - \theta}} C_{\theta,\gamma_{0}} (-\Delta)^{\frac{3}{7-\theta}} W(y,t) \quad &2 < \theta \le 3, \\ - (2 \pi)^{- \frac{3}{2}} C_{\theta,\gamma_{0}} (-\Delta)^{\frac{3}{4}} W(y,t) \quad &\theta > 3, \end{cases} \tag{\text{Thm \ref{main0}}} \\
\partial_{t} u(y,t) &= \begin{cases} - (2\pi)^{- \frac{6}{7 - \theta}} c_{\theta,\gamma_{0}} (-\Delta)^{\frac{3}{7-\theta}} u(y,t) \quad &2 < \theta \le 3, \\ - (2 \pi)^{- \frac{3}{2}} c_{\theta,\gamma_{0}} (-\Delta)^{\frac{3}{4}} u(y,t) \quad &\theta > 3. \end{cases} \tag{\text{Thm \ref{main1}, \ref{main2}}}
\end{align}
We also show that the superdiffusion coefficient obtained by the direct limit coincides with that obtained by the two-step limit, $c_{\theta,\gamma_{0}} = C_{\theta,\gamma_{0}}$. Applying our calculation to the exponential-decay model, one can obtain $c_{\a,\gamma_{0}} = C_{\a,\gamma_{0}}$.

In particular, if $\theta > 3$ then the exponent is same with that of the exponential-decay model. The threshold is $\theta = 3$ because the derivative of the dispersion relation diverges as $k \to 0$ when $\theta \le 3$ : 
\begin{align*}
\omega^{'}(k) := \frac{\widehat{\a}^{'}(k)}{\sqrt{\widehat{\a}(k)}} \sim \begin{cases} |k|^{-\frac{3 - \theta}{2}} \quad &2 < \theta < 3, \\ \sqrt{ - \log k} \quad &\theta = 3, \\ 1 \quad &\theta > 3, \end{cases} \quad k \to 0.
\end{align*}
Roughly speaking, if $\omega^{'}(k) \sim |k|^{a}$ and the mean value of the scattering kernel $R(k) = \int_{\T} dk^{'} R(k,k^{'}) \sim k^{b}$ as $k \to 0$, then one will get $\frac{b+1}{2(b - a)}$ - fractional diffusion by the scaling limit of the thermal energy distribution when $\frac{b+1}{2(b - a)} \in (0,1) $ and normal diffusion when $\frac{b+1}{2(b - a)} \ge 1$ or \textcolor{black}{$a > b \ge 0$}. There are several choices of dispersion relation and stochastic perturbation, and an exponent of fractional diffusion depends on one's choice. Here is a table about asymptotic exponents of feature values of some well-studied models, and one can see that our formal discussion agrees with prior researches, see \cite{BO,JKO,KO,KOS}. Superdiffusion of energy only occurs in the unpinned model of \cite{JKO} and our models. 

\begin{table}[h]
\begin{tabular}{|c||c|c|c|c|c|} \hline
Model & Potential & Noise & $a$ & $b$ & $(b+1)/2(b-a)$ \\ \hline \hline
\cite{BO} & exp.UP & NLMCN & $0$ & $0$ & $0$ \\ \hline
\cite{JKO} & exp.UP & LMCN & $0$ & $2$ & $3/4$ \\ \hline
\cite{JKO} & exp.P & LMCN & $1$ & $2$ & $3/2$ \\ \hline
\cite{KO} & NA & LMCN & $1$ & $2$ & $3/2$ \\ \hline
\cite{KOS} & exp.UP & NLMCN & $0$ & $0$ & $\infty$ \\ \hline
& $2 < \theta < 3$ & LMCN & $-(3 - \theta)/2$ & $2$ & $3/(7-\theta)$ \\ \cline{2-6} 
Our Model & $\theta = 3$ & LMCN & $$0$ \text{ with log-corr.}$ & $2$ & 3/4 \\ \cline{2-6} 
& $\theta > 3$ &LMCN& $0$ & $2$ & 3/4 \\ \hline
\end{tabular}
\end{table}

\begin{center} exp.UP = exp.decay-unpinned, ~ exp.P = exp.decay-pinned, ~ NA = non-acoustic, \\ LMCN = locally momentum conservative noise, \quad NLMCN = non-LMCN. \end{center}

\noindent The effect of the logarithmic correction at $\theta = 3$ does not appear in the exponent of the fractional diffusion, but the space-time scaling is different between the case $\theta > 3$ and $\theta = 3$, see $(\ref{timescaling})$. The threshold of parameters $a,b$ is obtained by convergence or divergence of the following integral: 
\begin{align*}
\lim_{\epsilon \to 0} \int_{\T} dk ~ \frac{R(k) (\omega^{'}(k))^{2}}{(R(k))^{2} + \epsilon^{2}(\omega^{'}(k))^{2}} \begin{cases} = \infty \quad &\frac{b+1}{2(b - a)} \in (0,1], \\
< \infty \quad & \frac{b+1}{2(b - a)} > 1 ~ \text{or} ~ a > b. \end{cases}
\end{align*}
This integral is expected to be proportional to thermal conductivity and diffusive coefficient and appeared in the final part of the proof of Theorem \ref{main0}, see (\ref{mainterm2}). If $\frac{b+1}{2(b - a)} = 1$, then the time
scaling should be diffusive with log-correction, and the expected macroscopic behavior is normal diffusion. Note that if there is no stochastic noise and $\theta > 2$, then the behavior of the thermal energy is purely ballistic.

\textcolor{black}{In \cite{TS}, the authors consider finite model with periodic boundary condition and they obtained the relationship between the decay rate of the polynomial interaction $\delta > 0$ and the decay speed of the current correlation function $C(t) \sim t^{-\beta(\delta)}, t \to \infty$. For the unpinned model, they showed that if $2< \delta < 3$ then $\beta(\delta) = \frac{\theta - 2}{2}$ and if $\delta > 3$ then $\beta(\delta) = \frac{1}{2}$, which is same as short-range model, so their result agrees with ours. }

\textcolor{black}{In a forthcoming article \cite{HS}, we also show the convergence of the scaled empirical measure of $(p_{x}(t), l_{x}(t), e_{x}(t))$ to the limit $(\bar{p}(y,t), \bar{l}(y,t), \bar{e}(y,t))$ at superballistic scaling, where $l_{x}$ is the generalized tension at $x \in \Z$. Especially, $\bar{p}(y,t)$ satisfies the superballistic wave equation:
\begin{align*}
\partial_{t}^{2} \bar{p}(y,t) = \begin{cases} - (\frac{C(\theta)}{2\pi})^{2} (-\Delta)^{\frac{\theta - 1}{2}} \bar{p}(y,t) &1 < \theta \le 3, \\
 (\frac{C(\theta)}{2\pi})^{2} \Delta \bar{p}(y,t) &\theta > 3. \end{cases}
\end{align*}}

Note that we study infinite systems with long-range interactions and non-equilibrium initial distribution of finite total energy, and thus it is appropriate for us to define the dynamics through wave functions $\{ \widehat{\psi}(k,t) ; k \in \T, t \ge 0 \}$. If we start from the wave function, we need some argument to define the energy at $x$ because the inter-particle distance is ``not'' a macroscopic variable, see Section $3$ and \cite{HS}.  In this paper, we assume thermal-type condition $(\ref{initialbound2})$ to derive the scaling limit of the thermal energy. On the other hand, in \cite{HS} we assume the so-called phononic-type condition to derive the hydrodynamic limit for $(p_{x}(t), l_{x}(t), e_{x}(t))$. The above initial conditions guarantee that our systems are in $\mathbb{L}^{2}$ at any time in some sense, and such $\mathbb{L}^{2}$ bound enable us to derive the scaling limit. 

There are numerical results about anharmonic chains with long-range interactions \cite{B},\cite{B2},\cite{ICLLC} and they exhibit anomalous heat conduction. In \cite{B2}, it was observed that the thermal conductivity has non-monotonic dependence with $\theta$ taking a maximum $\theta = 2$. However, it is outside of the scope of the current study to consider the thermal energy behavior in the case $\theta \le 2$ because \textcolor{black}{$\widehat{\a}^{'}(k)$ is not a continuous function and our proof relies on the asymptotic behavior of $\widehat{\a}^{'}(k)$ as $k \to 0$}. Therefore it remains an important open problem. 

Our paper is organized as follows: In Section 2 we prepare some notations. In Section 3 we introduce our model. In Section 4 we state our main results, Theorem $\ref{main0}$, $\ref{main1}$, and $\ref{main2}$. Proofs of Theorem $\ref{main0}$, $\ref{main1}$ and $\ref{main2}$ are given in Section 5, 6, 7 respectively.
\section{Notations}

Let $\T$ be the one-dimensional torus and $\T_{0} := \T \setminus \{ 0 \} $. We often identify $\T \cong [-\frac{1}{2},\frac{1}{2})$. 

For $f, g : \Z \to \R, h \in \ell^{2}(\Z)$ we define $f * g: \Z \to \R$ and $\widehat{h} \in \mathbb{L}^{2}(\T) $ as
\begin{align*}
(f*g)_{x} &:= \sum_{x^{'} \in \Z} f_{x-x^{'}}g_{x^{'}} ,\\
\widehat{h}(k) &:= \sum_{x \in \Z} e^{-2 \pi \sqrt{-1} k x} h_{x}.
\end{align*}

For $J : \R \to \mathbb{C}$ such that $J(y)$ is rapidly decreasing in $y \in \R$ we define $\widetilde{J} : \R \to \mathbb{C}$ as
\begin{align*}
\widetilde{J}(p) := \int_{\R} dy ~ e^{-2 \pi \sqrt{-1} p y} J(y).
\end{align*}

Denote by $\mathbb{S}(\R \times \T)$ the space of smooth functions $J : \R \times \T \to \mathbb{C}$ satisfying
\begin{align*}
\sup_{y \in \R , k \in \T} (1 + y^{2})^{n_{1}}|\partial_{y}^{n_{2}} \partial_{k}^{n_{3}} J(y, k)| < \infty 
\end{align*}
for any $n_{i} \in \Z_{\ge 0} , i = 1,2,3$. Let $\mathbb{S}(\R \times \T_{0})$ be a set of functions $J \in \mathbb{S}(\R \times \T)$ satisfying 
\begin{align*}
\sup_{y \in \R, k \in \T} (1+y^{2})^{n_{1}} |k|^{-n_{2}} |\partial_{y}^{n_{3}} \partial_{k}^{n_{4}} J(y,k) | < \infty,
\end{align*}
for any $n_{i} \in \Z_{\ge 0}, i = 1,2,3,4$. We introduce a norm $||\cdot||$ on $\mathbb{S}(\R \times \T)$ defined as
\begin{align*}
||J|| := \int_{\R} dp \sup_{k} |\widetilde{J}(p,k)| 
\end{align*}
for all $J \in \mathbb{S}(\R \times \T)$. The topology of $\mathbb{S}(\R \times \T)$ is defined via this norm. We regard $\mathbb{S}(\R)$, the Schwartz space on $\R$, as a subspace of $\mathbb{S}(\R \times \T)$: 
\begin{align*}
\mathbb{S}(\R):= \{ J \in \mathbb{S}(\R \times \T) ; J(y,k) = J(y) \}.
\end{align*}
By $\mathbb{S}(\R \times \T)^{'}, \mathbb{S}(\R)^{'}$ we denote the dual spaces of $\mathbb{S}(\R \times \T), \mathbb{S}(\R)$ respectively.

For two functions $f,g$ defined on common domain $A$, we write $f \lesssim g$ if there exists some positive constant $C > 0$ such that $f(a) \le Cg(a)$ for any $a \in A$.

\section{The dynamics}

In this section we define harmonic chains with noise and long-range interactions. Since we analyze the system with finite total energy, it is appropriate for us to define the dynamics through the wave functions $\{ \widehat{\psi}(k,t) ; k \in \T, t \ge 0 \}$ as $\mathbb{L}^{2}(\T)$ solution of the stochastic differential equation \eqref{defofpsi}. Then we can reconstruct the classical variables $\{ p_{x}(t), q_{x}(t) ; x \in \Z , t \ge 0 \}$ from $\{ \widehat{\psi}(k,t) ; k \in \T, t \ge 0 \}$ and define the energy $\{ e_{x}(t) ; x \in \Z, t \ge 0\}$. However, it may be difficult to understand the physical meaning of the important functions such as $\widehat{a}(k)$ and $R(k)$ from \eqref{defofpsi}. To clarify the meaning of the feature values, we first give a formal description of the dynamics in terms of $\{ p_{x}(t), q_{x}(t) ; x \in \Z , t \ge 0 \}$ in Section \ref{subseq:deterministic} and \ref{subseq:stochastic}, and introduce the wave function $\{ \widehat{\psi}(k,t) ; k \in \T, t \ge 0 \}$ in Section \ref{subseq:wavefunction}. Since we do not specify the initial condition $\{ p_{x}(0), q_{x}(0) ; t \ge 0 \}$ until the last half of Section \ref{subseq:stochastic}, the above construction of the dynamics is formal.

\subsection{Deterministic Dynamics}\label{subseq:deterministic}

First we consider harmonic chains with long-range interaction without noise. The configuration space is $(\R \times \R)^{\Z}$ and a configuration is denoted by $\{ (p_{x},q_{x}) \in \R \times \R ; x \in \Z \}$. The formal Hamiltonian of the system is given by 
\begin{align*}
H(p,q) &:= \frac{1}{2} \sum_{x \in \Z} |p_{x}|^{2} + \frac{1}{4} \sum_{x,x' \in \Z, x \neq x'} \frac{1}{|x - x'|^{\theta}} |q_{x} - q_{x'}|^{2} \\
&= \frac{1}{2} \sum_{x \in \Z} |p_{x}|^{2} + \frac{1}{2} \sum_{x,x' \in \Z, x \neq x'} \a_{x - x'} q_{x}q_{x'} \\
&= \sum_{x \in \Z} e_{x}(t) 
\end{align*}
where $e_{x}(t) := \frac{1}{2} |p_{x}|^{2} + \frac{1}{4} \sum_{x' \in \Z, x' \neq x} \frac{1}{|x - x'|^{\theta}} |q_{x} - q_{x'}|^{2} $ is called the energy at $x$, $\a : \Z \to \R$ is defined as
\begin{align*}
\begin{cases}
\a_{x} := - |x|^{- \theta} ,  \quad x \neq 0,  \\
\a_{0} := 2 \sum_{x \in \N} |x|^{ - \theta}
\end{cases}
\end{align*}
and $\theta > 1$ is a positive constant. The time evolution law of $\{ p_{x}(t), q_{x}(t) ; x \in \Z , t \ge 0 \}$ is given by the following differential equations
\begin{align}\label{deterministic}
\begin{cases}
d q_{x}(t) &= \frac{\partial H}{\partial p_{x}} (p(t),q(t)) dt = p_{x}(t) dt \\
d p_{x}(t) &= - \frac{\partial H}{\partial q_{x}} (p(t),q(t)) dt = - (\a * q)_{x} (t) dt .
\end{cases}
\end{align}

We introduce an operator $A$ defined as
\begin{align*}
A := \sum_{x \in \Z} p_{x}\partial_{q_{x}} + \sum_{x^{'} \in \Z} a_{x-x^{'}} q_{x^{'}} \partial_{p_{x}}.
\end{align*}
Then our deterministic dynamics satisfies $\frac{d}{dt} f(p,q) = (A f)(p,q)$ for any smooth local function $f$.


\subsection{Wave function}\label{subseq:wavefunction}

In this subsection we define the wave function of the deterministic Hamiltonian dynamics. From $(\ref{deterministic})$, we have the time evolution of $\{ \widehat{p}(k,t), \widehat{q}(k,t) ; k \in \T, t \ge 0 \}$
\begin{align}\label{deterministicont}
\frac{d}{dt} \begin{pmatrix} \widehat{q}(k,t) \\ \widehat{p}(k,t) \end{pmatrix} = \begin{pmatrix} 0 & 1 \\ -\widehat{a}(k) & 0 \end{pmatrix} \begin{pmatrix} \widehat{q}(k,t) \\ \widehat{p}(k,t) \end{pmatrix}.
\end{align}
The eigenvalues of the matrix appeared in the right hand side of $(\ref{deterministicont})$ are $\pm \sqrt{-1} \omega(k),$  $\omega(k) := \sqrt{\widehat{\a}(k)}$ and corresponding eigenvectors $\widehat{\psi}(k,t) , \widehat{\psi}^{*}(k,t)$ can be written as 
\begin{align}\label{wavefunction}
\widehat{\psi}(k,t) &= \omega(k) \widehat{q}(k,t) + \sqrt{-1} \widehat{p}(k,t) , \notag \\
\widehat{\psi}^{*}(k,t) &= \omega(k) \widehat{q}(-k,t) - \sqrt{-1} \widehat{p}(-k,t) ,
\end{align}
and the time evolution of $\widehat{\psi}(k,t), \widehat{\psi}^{*}(k,t)$ are given by 
\begin{align*}
d \widehat{\psi}(k,t) &= - \sqrt{-1} \omega(k) \widehat{\psi}(k,t) dt , \\
d \widehat{\psi}^{*}(k,t) &= \sqrt{-1} \omega(k) \widehat{\psi}^{*}(k,t) dt. 
\end{align*}
Here we normalize eigenvectors to satisfy $\int_{\T} dk ~ |\widehat{\psi}(k,t)|^{2} = 2 H(p,q)$. We call $ \{ \widehat{\psi}(k,t) ; k \in \T , t \ge 0 \}$ the wave function of the (deterministic) dynamics. 

\begin{remark}\label{mentioned}
From \eqref{wavefunction}, we can represent $\{ \widehat{p}(k,t), \widehat{q}(k,t) ; k \in \T, t \ge 0 \}$ by using the wave function:
\begin{align}\label{defofpq}
\begin{cases}
\widehat{q}(k,t) = \frac{1}{2 \omega(k)} (\widehat{\psi}(k,t) + \widehat{\psi}^{*}(-k,t)) , \\
\widehat{p}(k,t) = - \frac{\sqrt{1}}{2} (\widehat{\psi}(k,t) - \widehat{\psi}^{*}(-k,t)). \end{cases}
\end{align}
Actually, as we will see in Section \ref{subseq:stochastic}, we define $ \{ \widehat{\psi}(k,t) ; k \in \T , t \ge 0 \}$ \textcolor{black}{first} as the dynamics on $\mathbb{L}^{2}(\T)$ and then we define $\{ \widehat{p}(k,t), \widehat{q}(k,t) ; k \in \T, t \ge 0 \}$ by the equations $(\ref{defofpq})$. \textcolor{black}{Then,} $\omega(k) \sim k^{\frac{\theta - 1}{2}} , k \to 0$, $\widehat{q}(k,t)$ \textcolor{black}{is not necessarily} defined as an element of $\mathbb{L}^{2}(\T)$ \textcolor{black}{in general}. 
\end{remark}

\subsection{Stochastic noise and rigorous definition of the dynamics}\label{subseq:stochastic}

Next we add to this deterministic dynamics $(\ref{deterministic})$ a stochastic noise which conserves $p_{x-1} + p_{x} + p_{x+1}$ and $p_{x-1}^{2} + p_{x}^{2} + p_{x+1}^{2}, x \in \Z$. The corresponding stochastic differential equations are written as 
\begin{align}\label{stodiffeq:pq}
\begin{cases}
d q_{x}(t) = p_{x}(t) dt \\
d p_{x}(t) = \{ - (\a * q)_{x}(t) - \frac{\gamma}{2} (\b*p)_{x}(t) \} dt + \sqrt{\gamma} \sum_{z = -1, 0, 1} ( Y_{x+z}p_{x}(t) ) dw_{x+z},
\end{cases}
\end{align}
where $\gamma > 0$ is the strength of the noise and $Y_{x}, x \in \Z$ are vector fields defined as
\begin{align*}
Y_{x} := (p_{x} - p_{x+1})\partial_{p_{x-1}} + (p_{x+1} - p_{x-1})\partial_{p_{x}} + (p_{x-1} - p_{x})\partial_{p_{x+1}},
\end{align*} 
$\{ w_{x}(t) ; x \in \Z, t \ge 0 \} $ are i.i.d. one-dimensional standard Brownian motions, and 
\begin{align*}
\b_{x} := \begin{cases} 6 , &x = 0, \\
-2 , &x = \pm 1, \\ 
-1 , &x = \pm 2, \\ 
0 , &\textcolor{black}{\text{otherwise}}. \end{cases}
\end{align*}
In other words, $\{ p_{x}(t), q_{x}(t) ; t \ge 0 \}$ is a Markov process generated by $A + \gamma S$, where $S := \sum_{x \in \Z} (Y_{x})^{2}$. Note that $A + \gamma S$ conserves the total energy and the total momentum. 

From \eqref{stodiffeq:pq}, the time evolution of $\{ \widehat{p}(k,t), \widehat{q}(k,t) ; k \in \T, t \ge 0 \}$ is given by
\begin{align}\label{stochasticont}
\begin{cases}
d \widehat{q}(k,t) &= \widehat{p}(k,t) dt \\
d \widehat{p}(k,t) &= [ - \widehat{\a}(k) \widehat{q}(k,t) - 2\gamma R(k) \widehat{p}(k,t) ] dt \\
& \quad + 2 \sqrt{-1} \int_{\T} r(k,k') \widehat{p}(k-k',t) B(dk',dt), 
\end{cases}  
\end{align}
where 
\begin{align}\label{sqrtscatt}
r(k,k') &:= 2\sin{\pi k}^{2}\sin{2\pi(k-k')} + \sin{2\pi k}\sin{\pi(k-k')}^{2}, \\
R(k) &:= \frac{\widehat{\b}(k)}{4} = 2 \sin^{4}{\pi k} + \frac{3}{2} \sin^{2}{2 \pi k}, \label{defofmeanscat} \\
B(dk,dt) &:= \sum_{x \in \Z} e^{2 \pi k x} dk ~ w_{x}(dt) \notag.
\end{align}
Then we also define $ \{ \widehat{\psi}(k,t) ; k \in \T , t \ge 0 \}$ by $(\ref{wavefunction})$ for this stochastic system. From $(\ref{stochasticont})$, the time evolution of $ \{ \widehat{\psi}(k,t) ; k \in \T , t \ge 0 \}$ can be written as 
\begin{align}\label{defofpsi}
\begin{cases}
d \widehat{\psi}(k,t) &= [ - \sqrt{-1} \omega(k) \widehat{\psi}(k,t) - \gamma R(k) \{ \widehat{\psi}(k,t) - \widehat{\psi}^{*}(-k,t) \} ] dt \\
& ~ + \sqrt{-1} \sqrt{\gamma} \int r(k,k') [\widehat{\psi}(k-k',t) - \widehat{\psi}^{*}(k'-k,t)] dB(dk',dt) .
\end{cases}
\end{align}

\textcolor{black}{\begin{remark}
Another well-studied stochastic noise, which is local and conserves the total energy and the total momentum, is the jump-type momentum exchange noise defined as follows. We introduce an operator $S^{'}$ defined as
\begin{align*}
(S^{'}f)(p) := \sum_{x \in \Z} (f(p^{x,x+1}) - f(p))
\end{align*} 
for any bounded local function $f$ where $p^{x,x+1}$ is defined as
\begin{align*}
p^{x,x+1}_{z} := \begin{cases} p_{x+1} \quad z = x, \\
p_{x} \quad z = x+1, \\
p_{z} \quad \text{otherwise}. \end{cases}
\end{align*}
Then one can consider the Markov process $\{ p_{x}(t), q_{x}(t) ; t \ge 0 \}$ generated by $A + \gamma S^{'}$ or the corresponding stochastic differential equation for $ \{ \widehat{\psi}(k,t) ; k \in \T , t \ge 0 \}$, and then one can obtain the same scaling limit as our model because the asymptotic behavior of the mean value of the scattering kernel for the jump-type noise $R_{S^{'}}(k)$ is also $k^{2}, k \to 0$. 
\end{remark}}

Now we present the rigorous definition of the dynamics.  As we have mentioned at the beginning of this section, we define $ \{ \widehat{\psi}(k,t) \in \mathbb{L}^{2}(\T) ; k \in \T , t \ge 0 \}$ as the solution of $(\ref{defofpsi})$ with initial distribution $\mu_{0}$ where $\mu_{0}$ is an arbitrary probability measure on $\mathbb{L}^{2}(\T)$. \textcolor{black}{For $\theta > 1$, we can show the existence and uniqueness of the solution by using a classical fixed point theorem, see Appendix \ref{app:exuni} for the sketch of the proof. Once we define the dynamics $ \{ \widehat{\psi}(k,t) \in \mathbb{L}^{2}(\T) \}$, then we can also define $\{ \widehat{p}(k,t), \widehat{q}(k,t) ; k \in \T, t \ge 0 \}$ by $(\ref{defofpq})$. } Since $\widehat{q}(k)$ may not be in $\mathbb{L}^{2}(\T)$, we \textcolor{black}{can not} define $q_{x}$ as the Fourier coefficient of $\widehat{q}(k)$. But still it is sufficient to define $\sum_{x^{'}} \a_{x-x^{'}}(q_{x} - q_{x^{'}})^{2}$ and so $e_{x}(t)$ by the following argument: \textcolor{black}{suppose that} $\{ q_{x}, x \in \Z \}$ \textcolor{black}{is an} $l^{2}(\Z)$ element, then the Fourier transform of $\sum_{x^{'}} \a_{x-x^{'}}|q_{x} - q_{x^{'}}|^{2}$ is equal to 
\begin{align*}
&\frac{1}{4} \int_{\T} dk dk^{'} F(k-k^{'},k^{'}) (\widehat{\psi}(k^{'}) + \widehat{\psi}(-k^{'})^{*}) (\widehat{\psi}(k-k^{'}) + \widehat{\psi}(k^{'}-k)^{*}), \\
& \quad F(k,k^{'}) := \frac{\widehat{\a}(k+k^{'}) - \widehat{\a}(k) - \widehat{\a}(k^{'})}{\omega(k) \omega(k^{'})}.
\end{align*}
Therefore \textcolor{black}{when} we start from the wave functions to define the dynamics, we define $\sum_{x^{'}} \a_{x-x^{'}}|q_{x} - q_{x^{'}}|^{2}$ as the Fourier coefficient of the above integration:
\begin{align}\label{defofpotential}
\sum_{x^{'}} \a_{x-x^{'}}|q_{x} - q_{x^{'}}|^{2} := \frac{1}{4} \int_{\T^{2}} dk dk^{'} & e^{2 \pi \sqrt{-1} (k+k^{'}) x} F(k,k^{'}) \notag \\ 
& \times (\widehat{\psi}(k) + \widehat{\psi}(-k)^{*}) (\widehat{\psi}(k^{'}) + \widehat{\psi}(-k^{'})^{*}).
\end{align}
Then we can define the energy of particle $x$ in \textcolor{black}{the} usual way:
\begin{align*}
e_{x}(t) := \frac{1}{2} |p_{x}|^{2} - \frac{1}{4} \sum_{x' \in \Z, x' \neq x} \a_{x-x^{'}} |q_{x} - q_{x'}|^{2}. 
\end{align*}

\section{Main Result : Superdiffusive behavior of the energy}

In this section we state our main results about the scaling limit of the energy. First we introduce the Wigner distribution, which is a good substitute of the empirical measure of the energy. We show that space-time-noise-scaled Wigner distribution converges to a solution of the fractional diffusion equation, see Theorem \ref{main0}, Theorem \ref{main1} and Theorem \ref{main2}. Since the strength of the noise means the strength of some nonlinear effect, if the noise is weaker, then the scaling of the time \textcolor{black}{should be slower} (cf. Theorem \ref{main0}). Critical scaling is space : time : noise $= \epsilon:\epsilon:\epsilon$, $\epsilon \to 0$. (cf. Theorem \ref{main1}, Theorem \ref{main2}) 

\subsection{Wigner distribution(local spectral density)}\label{subseq:wigner}

Let $0 < \epsilon < 1$ be a scale parameter and $\{ \mu_{\epsilon} \}_{0< \epsilon < 1}$ be a family of probability measures on $\mathbb{L}^{2}(\T)$. We define $ \{ \widehat{\psi}(k,t) = \widehat{\psi}_{\epsilon}(k,t) \in \mathbb{L}^{2}(\T) ; k \in \T , t \ge 0 \}_{0 < \epsilon < 1}$ as the solution of $(\ref{defofpsi})$ with initial condition $\mu_{\epsilon}$ and $\gamma := \epsilon^{s} \gamma_{0}, ~ \gamma_{0} > 0 , 0 \le s \le 1$. The exponent $0 \le s \le 1$ represents the strength of the noise. If $s = 1$ (resp. $0 \le s < 1$) , then we say that the noise is weak (resp. strong). We assume the following energy bound condition, called thermal type condition (cf.\cite{KO2}):
\begin{align}\label{initialbound2}
\sup_{0 < \epsilon < 1} \int_{\T} dk ~ \epsilon^{2} | E_{\mu_{\epsilon}}[ | \widehat{\psi}(k) |^{2} ] |^{2} \le K_{1}
\end{align}
where $E_{\mu_{\epsilon}}$ is the expectation with respect to $\mu_{\epsilon}$ and $K_{1}$ is a positive constant. Note that this assumption and the energy conservation law 
\begin{align}\label{eq:energyconservation}
\mathbb{E}_{\epsilon}[ | \widehat{\psi}(k,t) |^{2} ] = E_{\mu_{\epsilon}}[ | \widehat{\psi}(k) |^{2} ] 
\end{align}
imply 
\begin{align}
&\sup_{0 < \epsilon < 1} \int_{\T} dk ~ \epsilon  \mathbb{E}_{\epsilon}[ | \widehat{\psi}(k,t) |^{2} ] = \sup_{0 < \epsilon < 1} \int_{\T} dk ~ \epsilon  E_{\mu_{\epsilon}}[ | \widehat{\psi}(k) |^{2} ] \le \sqrt{K_{1}}, \label{initialbound} 
\end{align}
for any $J \in \mathbb{S}(\R), t \ge 0$, where $\mathbb{E}_{\epsilon}$ is the expectation of the dynamics with initial condition $\mu_{\epsilon}$. We can easily show \eqref{eq:energyconservation} by substituting $p = 0$ for both sides of \eqref{startpoint} and by taking the integral with respect to $k \in \T$.

Now we define the Wigenr distribution $W_{\epsilon,+}(t) \in \mathbb{S}(\R \times \T)^{'}$, which is the local spectral density of the energy as we will see later in Remark $\ref{localspecden}$, as follows:
\begin{align}
&<W_{\epsilon,+}(t),J> \notag \\
&:= \frac{\epsilon}{2} \sum_{x,x' \in \Z} \mathbb{E}_{\epsilon} [\psi_{x^{'}}(\frac{t}{f_{\theta,s}(\epsilon)})^{*} \psi_{x}(\frac{t}{f_{\theta,s}(\epsilon)})] \int_{\T} dk ~ e^{2 \pi \sqrt{-1} (x' - x) k} J(\frac{\epsilon(x+x')}{2},k)^{*} \label{Wigner+Z} \\
&= \frac{\epsilon}{2} \int_{\R \times \T} dpdk ~ \mathbb{E}_{\epsilon} [\widehat{\psi}(k - \frac{\epsilon p}{2},\frac{t}{f_{\theta,s}(\epsilon)})^{*} \widehat{\psi}(k + \frac{\epsilon p}{2},\frac{t}{f_{\theta,s}(\epsilon)})] \widetilde{J}(p,k)^{*} \label{Wigner+T} ,
\end{align}
for $t \ge 0$, $J \in \mathbb{S}(\R \times \T)$. The ratio of space-time scaling $f_{\theta,s}(\epsilon)$ is given by 
\begin{align}\label{timescaling}
f_{\theta,s}(\epsilon) := \begin{cases} \epsilon^{\frac{6 - s(\theta - 1)}{7 - \theta}}  \quad &1 < \theta < 3 , ~ 0 \le s \le 1, \\ \epsilon^{s} |h_{s}(\epsilon)|^{3}  \quad &\theta = 3, ~ 0 \le s < 1, \\ \epsilon  \quad &\theta = 3, ~ s = 1, \\ \epsilon^{\frac{3 - s}{2}}  \quad &\theta > 3, ~ 0 \le s \le 1, \end{cases}
\end{align}
where $h_{s}(\cdot)$ is the inverse function of $y \mapsto (\frac{y^{4}}{- \log{y}})^{\frac{1}{2(1-s)}}$ on $[0,1)$. 

\begin{remark}\label{localspecden}
Actually, $W_{\epsilon,+}(t), t \ge 0$ is well-defined on a wider class of test functions than $\mathbb{S}(\R \times \T)$. If $J(y,k) = J(k), (y,k) \in \R \times \T$ and $J(k), k \in \T$ is bounded , then we can define $<W_{\epsilon,+}(t),J>$ by $(\ref{Wigner+Z})$ and we have
\begin{align}\label{energydisonT}
<W_{\epsilon,+}(t),J> = \frac{\epsilon}{2} \int_{\T} dk ~ \mathbb{E}_{\epsilon}[| \widehat{\psi}(k,\frac{t}{f_{\theta,s}(\epsilon)}) |^{2}] J(k).
\end{align}
From $(\ref{energydisonT})$, we see that the Wigner distribution has the information of the spectral density of the energy. In addition, if $J \in \mathbb{S}(\R \times \T)$ and $J(y,k) = J(y),~ (y,k) \in \R \times \T$, then we have
\begin{align}\label{energydisonR}
<W_{\epsilon,+}(t),J> = \frac{\epsilon}{2} \sum_{x \in \Z} \mathbb{E}_{\epsilon}[|\psi_{x}(\frac{t}{f_{\theta,s}(\epsilon)})|^{2}] J(\epsilon x).
\end{align}
From $(\ref{energydisonR})$, we see that the Wigner distribution is the empirical measure of $\{ |\psi_{\epsilon}(x)|^{2} ; x \in \Z \}$. According to $(\ref{energydisonT})$ and $(\ref{energydisonR})$, we can think the Wigner distribution as the local spectral density of the energy.
\end{remark}

Next proposition ensures that the limit of the Wigner distribution is the macroscopic distribuiton of the energy:
\begin{proposition}\label{energywave}
Assume $(\ref{initialbound2})$. Then for any $t \ge 0, J \in S(\R)$, we have
\begin{align*}
\varlimsup_{\epsilon \to 0} |<W_{\epsilon,+}(t),J> - \epsilon \sum_{x \in \Z} \mathbb{E}_{\epsilon}[e_{x}(\frac{t}{f_{\theta,s}(\epsilon)})] J(\epsilon x)| = 0.
\end{align*}
\end{proposition}
We postpone the proof of this proposition to Appendix \ref{app:energywave}.

\subsection{Superdiffusive behavior of the energy : direct limit}

Now we state one of our main results.

\begin{theorem}\label{main0}
Suppose that $\theta > 2$, $0 \le s < 1$, $(\ref{initialbound2})$ and there exists some $W_{0} \in \mathbb{L}^{1}(\R)$ such that
\begin{align}\label{initialconvergence}
\lim_{\epsilon \to 0} <W_{\epsilon,+}(0),J> = \int_{\R} dy ~ W_{0}(y) J(y), 
\end{align}
for $J(y,k) \equiv J(y) \in \mathbb{S}(\R \times \T)$. Then for $J \in C_{0}^{\infty}(\R \times \R_{\ge 0})$, we have 
\begin{align*}
\lim_{\epsilon \to 0} \int_{0}^{\infty} dt ~ <W_{\epsilon,+}(t),J(\cdot,t)> = \int_{\R} dy \int_{0}^{\infty} dt ~ W(y,t) J(y,t) ,
\end{align*}
where $W(y,t)$ is given by
\begin{align*}
\widetilde{W}(p,t) = \begin{cases} \exp{ \{ - C_{\theta,\gamma_{0}} |p|^{\frac{6}{7 - \theta}} t \} } \widetilde{W}_{0}(p)  \quad &2 < \theta \le 3, \\ \exp{ \{ - C_{\theta,\gamma_{0}} |p|^{\frac{3}{2}} t \} } \widetilde{W}_{0}(p) \quad &\theta > 3 \end{cases}
\end{align*}
and 
\begin{align}\label{diffcoeffi}
C_{\theta,\gamma_{0}} &:= \begin{cases} \frac{24 \pi^{3} \csc{\frac{(4-\theta)\pi}{7 - \theta}}}{7 -\theta} (\frac{(\theta - 1)}{24 \pi^{2}})^{\frac{6}{7-\theta}} (\gamma_{0})^{-\frac{\theta - 1}{7 - \theta}} C(\theta)^{\frac{3}{7-\theta}}  \quad &2 < \theta \le 3, \\ \frac{\sqrt{6}}{12} \gamma_{0}^{-\frac{1}{2}} C(\theta)^{\frac{3}{4}} \quad &\theta > 3, \end{cases} \\
C(\theta) &:= \begin{cases} 4 \pi^{\theta - 1} \int_{0}^{\infty} dy ~ \frac{\sin^{2} y}{|y|^{\theta}} \quad &1 < \theta < 3, \\
4 \pi^{2} \quad &\theta = 3, \\
4 \pi^{2} \sum_{x \ge 1} |x|^{2 - \theta} \quad &\theta > 3. \end{cases} \label{diffcoeffi2}
\end{align}
In other words, $W(y,t)$ satisfies a fractional diffusion equation:
\begin{align*}
\partial_{t} W(y,t) = \begin{cases} - (2\pi)^{- \frac{6}{7 - \theta}} C_{\theta,\gamma_{0}} (-\Delta)^{\frac{3}{7-\theta}} W(y,t) \quad &2 < \theta \le 3, \\ - (2 \pi)^{- \frac{3}{2}} C_{\theta,\gamma_{0}} (-\Delta)^{\frac{3}{4}} W(y,t) \quad &\theta > 3. \end{cases}
\end{align*}

\end{theorem}

\begin{remark}\label{re:seqcompact}

From $(\ref{initialbound})$, we have the uniform boundness of the Wigner distribution:
\begin{align}\label{sequentialcompact}
\frac{|<W_{\epsilon,+}(t),J>|}{\| J \|} \le \frac{\sqrt{K_{1}}}{2} \quad 0 < \epsilon < 1, t \ge 0, J \in \mathbb{S}(\R \times \T).
\end{align}
Thanks to $(\ref{sequentialcompact})$, we obtain the $\star$ - weakly sequential compactness of the Wigner distribution in $\mathbb{L}^{\infty}([0,T] \times \mathbb{S})$ for any $T > 0$ , that is, there exists some subsequence $\{ <W_{\epsilon(n),+}(t),\cdot> \}_{n}$ and an element $<W(t),\cdot> \in \mathbb{S}^{'}(\R \times \T)$ such that
\begin{align*}
\lim_{n \to \infty} | \int_{0}^{T} dt ~ <W_{\epsilon(n),+}(t),J> f(t) - \int_{0}^{T} dt ~ <W(t),J> f(t) | = 0
\end{align*}
for any $J \in \mathbb{S}(\R \times \T)$, $f \in \mathbb{L}^{1}([0,T])$. Therefore if we verify the uniqueness of limits of convergent subsequences, then we can conclude the proof of Theorem $\ref{main0}$. \textcolor{black}{Note that any limit of a convergent subsequence $<W(t),\cdot>$ is positive a.e. $t$ and we can extend $<W(t),\cdot>$ to a finite positive measure $\mu(t)(dy)$ on $\R$, see Section $\ref{subsubsq:extension}$. }
\end{remark}

From Theorem $\ref{main0}$ and Proposition $\ref{energywave}$ we conclude that when $2 < \theta \le 3$, the time evolution of the macroscopic energy distribution is governed by $\frac{3}{7 - \theta}$-fractional diffusion equation.

\begin{remark}
\textcolor{black}{We can also consider pinned models and for these models the interaction potential $\a$ is defined as follows:
\begin{align*}
\a_{x} := - |x|^{- \theta} , \quad x \in \Z \setminus \{ 0 \} \quad \a_{0} := \nu + 2 \sum_{x \in \N} |x|^{- \theta} \quad \nu > 0, \theta > 1.
\end{align*}
The definition of the wave function and the Wigner distribution are same as unpinned chains, but the time scaling $f_{\theta,s}(\epsilon) = f_{\theta,\nu,s}(\epsilon)$ is changed as
\begin{align*}
f_{\theta,\nu,s}(\epsilon) := \begin{cases} \epsilon^{\frac{3 - s(\theta - 1)}{4 - \theta}}  \quad &2 < \theta < \frac{5}{2} , ~ 0 \le s < 1, \\ \epsilon^{2-s} \log{\epsilon^{-1}}  \quad &\theta = \frac{5}{2}, ~ 0 \le s < 1, \\ \epsilon^{2 - s}  \quad &\theta > \frac{5}{2}, ~ 0 \le s < 1. \end{cases}
\end{align*}}
 
\textcolor{black}{Assume that $(\theta,s) \in (2,\frac{5}{2}] \times [0,1) \cup (\frac{5}{2},\infty) \times (0,1)$. Then by using the same strategy of Section \ref{sec:proof0}, we have the direct limit for the Wigner distribution as follows:
\begin{align*}
\lim_{\epsilon \to 0} \int_{0}^{\infty} dt ~ <W_{\epsilon,+}(t),J(\cdot,t)> = \int_{\R} dy \int_{0}^{\infty} dt ~ W(y,t) J(y,t) ,
\end{align*}
where $W(y,t)$ is given by
\begin{align*}
\widetilde{W}(p,t) = \begin{cases} \exp{ \{ - C_{\theta,\nu,\gamma_{0}} |p|^{\frac{3}{4 - \theta}} t \} } \widetilde{W}_{0}(p)  \quad &2 < \theta \le \frac{5}{2}, \\ \exp{ \{ - C_{\theta,\nu,\gamma_{0}} |p|^{2} t \} } \widetilde{W}_{0}(p) \quad &\theta > \frac{5}{2} \end{cases}
\end{align*}
and 
\begin{align*}
C_{\theta,\nu,\gamma_{0}} &:= \begin{cases} \frac{12 \pi^{3} \csc{\frac{3(\theta - 2)\pi}{4(4 - \theta)} + \frac{\pi}{4} }}{4 - \theta} (\frac{C(\theta)}{24 \sqrt{\nu} \pi})^{\frac{3}{4 - \theta}} \gamma^{- \frac{\theta - 1}{4 - \theta}}  \quad &2 < \theta \le \frac{5}{2}, \\ \gamma_{0}^{-1} \int_{\T} dk ~ \frac{(\omega^{'}(k))^{2}}{2R(k)} \quad &\theta > \frac{5}{2}. \end{cases}  
\end{align*}}

\textcolor{black}{Although the chain is pinned, we see that if $2 < \theta \le \frac{5}{2}$ then the macroscopic energy behavior is anomalous, and  this result also agrees with the result of \cite{TS}. Since we want to avoid complicated division into cases, in this paper we do not give the proof for pinned chains. Note that we can also consider the cases $(\theta,s) \in (\frac{5}{2},\infty) \times \{ 0 \}$, and with some additional computations we can prove that the macroscopic energy behavior is normal. Since in this cases $f_{\theta,\nu,0}(\epsilon) = \epsilon ^{2}$, we have to expand the functions $R(k,k^{'},\epsilon p), R(k \pm \frac{\epsilon p}{2})$ appeared in $\eqref{startpoint}$ with respect to $\epsilon p$ up to the third order. }
\end{remark}

\begin{remark}
It is impossible for us to consider the case $1 < \theta \le 2$ by using the strategy of \cite{JKO} because their proof relies only on the asymptotic behavior of the sound speed at $k = 0$. In the case $1 < \theta \le 2$, \textcolor{black}{$\widehat{\a}^{'}(k)$ and $\omega^{'}(k)$ are not continuous functions on $\T \setminus \{ 0 \}$.} 
\end{remark}

\subsection{Derivation of the Boltzmann equation : micro to meso}

If $s=1$ we need the two-step scaling limit, one is microscopic scale to mesoscopic scale, then mesoscopic scale to macroscopic scale, to derive the fractional diffusion. 
\begin{theorem}\label{main1}
Suppose that $\theta > 2$, $s=1$ and there exists a finite measure $\mu_{0}$ on $\R \times \T$ such that
\begin{align}\label{ass:initialconvergence2}
\lim_{\epsilon \to 0} <W_{\epsilon,+}(0),J> = \int_{\R \times \T} d\mu_{0} J(y,k)
\end{align}
for any $J \in \mathbb{S}(\R \times \T)$. 
Then we have the following results:

\noindent $(1)$ For any $t \ge 0$, there exists $W(t) \in \mathbb{S}(\R \times \T)^{'}$ such that 
\begin{align*}
\lim_{\epsilon \to 0} |<W(t),J> - <W_{\epsilon,+},J>| = 0,
\end{align*}
for any $J \in \mathbb{S}(\R \times \T)$.

\noindent $(2)$ $W(t)$ can be extended to a finite measure $\mu(t)$ on $\R \times \T$.

\noindent $(3)$ $\mu(t)$ is the unique solution of the following measure-valued Boltzmann equation:
\begin{align}\label{boltzmann}
\partial_{t} \int_{\R \times \T} d \mu(t) J(y,k) &= \frac{1}{2 \pi} \int_{\R \times \T} d \mu(t) \omega^{'}(k) \partial_{y} J(y,k) + \gamma_{0} \int_{\R \times \T} d \mu(t) (\mathcal{L}J)(y,k) , \notag \\
\int_{\R \times \T} d \mu(0) J(y,k) &= \int_{\R \times \T} d \mu_{0} J(y,k) \quad \quad \mu(t)(dy,\{ 0 \}) = 0,
\end{align}
for any $J \in \mathbb{S}(\R \times \T_{0})$, where $\mathcal{L}$ is a scattering operator defined as
\begin{align}
(\mathcal{L}J)(y,k) &:= 2 \int_{\T} dk' ~ R(k,k') (J(y,k') - J(y,k)) , \label{jumpop} \\
R(k,k') &:= \frac{1}{2} [ r(k,k+k')^{2} + r(k,k-k')^{2} ] \label{jumpkernel}.
\end{align}
\end{theorem}

We usually expect that the limit of the local spectral density homogenize on $\T$ due to the scattering effect. However, $\mu(t)$ does not homogenize on $\T$ in the critical case $s=1$. This implies that there should exist a longer time scaling on which the homogenization on $\T$ occurs and the time evolution of the energy on $\R$ is non-trivial, and leads us to the problem of the scaling limit for the rescaled solution of \eqref{boltzmann}.

\subsection{Derivation of fractional diffusion equation : meso to macro}

In this subsection we construct a solution of $(\ref{boltzmann})$ probabilistically. Let $\{ K_{n} \in \T ; n \in \Z_{\ge 0} \}$ be a Markov chain whose transition probability $P(k,dk')$ is given by
\begin{align*}
P(k,dk') := \frac{R(k,k')}{R(k)} dk' .
\end{align*}
The invariant probability measure $\pi$ for $\{ K_{n} \in \T ; n \in \N \}$ is written as
\begin{align*}
\pi(dk) := \frac{2R(k)}{3} dk.
\end{align*}
Suppose that $\{ \tau_{n} ; n \in \N \}$ be an I.I.D. sequence of random variables such that $\tau_{1}$ is exponentially distributed with intensity 1, and $\{ K_{n} \in \T ; n \in \Z_{\ge 0} \}$ and $\{ \tau_{n} ; n \in \N \}$ are independent. Set $t_{n} := \sum_{m = 1}^{n} \frac{1}{2 \gamma_{0} R(K_{m-1})} \tau_{m} , n \ge 1 $ , $t_{0} := 0$. Now we define a continuous-time Markov process $\{ K(t) \in \T ; t \ge 0 \}$ as $K(t) := K_{n}$ if $ t_{n} \le t < t_{n+1} $ for some $n \in \Z_{\ge 0}$. By simple computations we see that $\{ K(t) \in \T ; t \ge 0 \}$ is the continuous random walk generated by $2\gamma_{0} \mathcal{L}$. Now we can construct a soluton of $(\ref{boltzmann})$. Suppose that $u_{0}:\R \times \T \to \R_{\ge 0}$ is a function such that $u(\cdot,k) \in C^{1}_{0}(\R)$ for all $k \in \T $ and $u(y,\cdot) \in C(\T)$ for all $y \in \R$. Then we define a function $u : \R \times \T \to \R_{\ge 0}$ as 
\begin{align*}
u(y,k,t) := \mathbb{E}_{k} [u_{0}(y + Z(t),K(t))],
\end{align*}
where
\begin{align*}
Z(t) := \int_{0}^{t} ds ~ \omega^{'}(K(s)) .
\end{align*}
$u(y,k,t)dydk$ is the unique solution of $(\ref{boltzmann})$ with initial condition $u(y,k,0) dydk = u_{0}(y,k) dydk$. By applying \cite[Theorem 2.8]{KJO}, we obtain the scaling limit of $u(y,k,t)$. 

\begin{theorem}\label{main2}
$\\$
$(1)$ As $N \to \infty$, the finite-dimensional distributions of scaled process $\{ \frac{1}{N(\theta)} Z(Nt) ; t \ge 0 \}$ converge weakly to those of the L\'{e}vy process generated by $ - (2\pi)^{-\frac{6}{7-\theta}} C_{\theta,\gamma_{0}} (-\Delta)^{\frac{3}{7 - \theta}}$ if $2 < \theta \le 3$ and by $ - (2\pi)^{-\frac{3}{2}} C_{\theta,\gamma_{0}} (-\Delta)^{\frac{3}{4}}$ if $\theta > 3$, where $C_{\theta,\gamma_{0}}$ is given by $(\ref{diffcoeffi})$ and the ratio of the space-time scaling $N(\theta)$ is defined as
\begin{align*}
N(\theta) := \begin{cases} N^{\frac{7-\theta}{6}} \quad &2 < \theta < 3, \\ (\log{N})^{\frac{1}{2}} N^{\frac{2}{3}} \quad &\theta = 3 , \\ N^{\frac{2}{3}} \quad & \theta > 3. \end{cases}
\end{align*}

\noindent $(2)$ Suppose that $u_{0} \in C_{0}^{\infty}(\R \times \T)$. Define $u_{N}(y,k,t)$ as
\begin{align*}
u_{N}(y,k,t) := \mathbb{E}_{k}[u_{0}(\frac{1}{N(\theta)}(y + Z(Nt)),K(Nt))].
\end{align*}
Then for any $y \in \R$ and $t \ge 0$, we have
\begin{align*}
\lim_{N \to \infty} \int_{\T} dk ~ |u_{N}(N(\theta)y,k,t) - \bar{u}(y,t)|^{2} = 0,
\end{align*}
where $\bar{u}$ is the solution of the following partial differential equation:
\begin{align*}
&\begin{cases}
\partial_{t} \bar{u}(y,t) = - (2\pi)^{-\frac{6}{7-\theta}} C_{\theta,\gamma_{0}} (-\Delta)^{\frac{3}{7 - \theta}} u_{0}(y,t) , \\
\bar{u}(y,0) = \int_{\T} dk ~ u_{0}(y,k) .
\end{cases} & &2 < \theta \le 3, \\
&\begin{cases}
\partial_{t} \bar{u}(y,t) = - (2\pi)^{-\frac{3}{2}} C_{\theta,\gamma_{0}} (-\Delta)^{\frac{3}{4}} u_{0}(y,t) , \\
\bar{u}(y,0) = \int_{\T} dk ~ u_{0}(y,k) .
\end{cases} & &\theta > 3.
\end{align*}

\end{theorem}

We point out that the coefficient of the fractional diffusion obtained by two-step scaling limit is equal to that obtained by direct scaling limit $(\ref{diffcoeffi})$. From Theorem $\ref{main2}$ (2), we see that the time evolution of the macroscopic energy distribution $u(y,t)$ is governed by the fractional diffusion equation. 

\section{proof of Theorem \ref{main0}}\label{sec:proof0}

\subsection{Overview of the proof}

In this subsection we outline the very intuitive, and not rigorous proof of Theorem $\ref{main0}$. As we have mentioned in remark below Theorem $\ref{main0}$, all we have to do is to show the uniqueness of the limit of any convergent subsequence. To prove this, we fix a convergent subsequence and for notational simplicity we also denote this subsequence by $\{ <W_{\epsilon,+},\cdot> \}_{\epsilon}$. 

First we calculate the time evolution of the Wigner distribution. Roughly we obtain
\begin{align*}
\frac{d}{dt} \widetilde{W}_{\epsilon,+}(p,k,t) &= - \frac{\sqrt{-1} \epsilon \omega^{'}(k)}{f_{\theta,s}(\epsilon)} \widetilde{W}_{\epsilon,+}(p,k,t) + \frac{\gamma}{f_{\theta,s}(\epsilon)} [\mathcal{L} \{ \widetilde{W}_{\epsilon,+}(p,\cdot,t)](k) + o_{\epsilon}(1),
\end{align*}
where $\widetilde{W}_{\epsilon,+}$ is the Fourier transform of the Wigner distribution. Since this differential equation includes the scattering term on $\T$, the limit of the Wigner distribution homogenizes in $k \in \T$, that is, we will obtain $W_{\epsilon,+}(y,k,t) \to W(y,t) \quad \epsilon \to 0$ in some sense. Then by taking the Laplace transform of both sides of the above equation, we have
\begin{align*}
\bar{w}_{\epsilon,+}(p,k,\lambda) - \widetilde{W}_{\epsilon,+}(p,k,0) = - \frac{\sqrt{-1} \epsilon \omega^{'}(k)}{f_{\theta,s}(\epsilon)} \bar{w}_{\epsilon,+}(p,k,\lambda) + \frac{\gamma}{f_{\theta,s}(\epsilon)} [\mathcal{L} \{ \bar{w}_{\epsilon,+}(p,\cdot,\lambda)](k) + o_{\epsilon}(1),
\end{align*}
where $\bar{w}_{\epsilon,+}(p,k,\lambda) := \int_{0}^{\infty} e^{-\lambda t} \widetilde{W}_{\epsilon,+}(p,k,t)$. From the above equation and \eqref{initialconvergence}, we obtain
\begin{align*}
&\bigl\{ \int_{\T} dk ~ \frac{2 \gamma R(k)}{f_{\theta,s}(\epsilon)}(1 - \frac{2 \gamma R(k)}{f_{\theta,s}(\epsilon) \lambda + 2 \gamma R(k) + \sqrt{-1} \epsilon \omega^{'}(k)}) \bigr\} \int_{\T} dk ~ \frac{2}{3} R(k) \bar{w}_{\epsilon,+}(p,k,\lambda) \\
&= \bigl( \int_{\T} dk ~ \frac{2\gamma R(k)}{f_{\theta,s}(\epsilon) \lambda + 2 \gamma R(k) + \sqrt{-1} \epsilon \omega^{'}(k)} \bigr) \widetilde{W}_{0}(p) + o_{\epsilon}(1). 
\end{align*}
Since
\begin{align*}
&\int_{\T} dk ~ \frac{2 \gamma R(k)}{f_{\theta,s}(\epsilon)}(1 - \frac{2 \gamma R(k)}{f_{\theta,s}(\epsilon) \lambda + 2 \gamma R(k) + \sqrt{-1} \epsilon \omega^{'}(k)}) \to \begin{cases} \lambda + C_{\theta,\gamma_{0}}|p|^{\frac{6}{7-\theta}} \quad 2 < \theta \le 3, \\ \lambda + C_{\theta,\gamma_{0}}|p|^{\frac{3}{2}} \quad \theta > 3, \end{cases} \\
&\int_{\T} dk ~ \frac{2\gamma R(k)}{f_{\theta,s}(\epsilon) \lambda + 2 \gamma R(k) + \sqrt{-1} \epsilon \omega^{'}(k)} \to 1, \\
&\int_{\T} dk ~ R(k) \bar{w}_{\epsilon,+}(p,k,\lambda) \to w(p,\lambda) = \int_{\R} dy ~ e^{-2 \pi \sqrt{-1} y p} \int_{0}^{\infty} dt ~ e^{-\lambda t} W(y,t),
\end{align*}
as $\epsilon \to 0$, we have
\begin{align*}
\begin{cases} (\lambda + C_{\theta,\gamma_{0}}|p|^{\frac{6}{7-\theta}}) w(p,\lambda) + \widetilde{W}_{0}(p) = 0 \quad 2 < \theta \le 3, \\
(\lambda + C_{\theta,\gamma_{0}}|p|^{\frac{3}{2}}) w(p,\lambda) + \widetilde{W}_{0}(p) = 0 \quad \theta > 3. \end{cases}
\end{align*}
By using the one-to-one correspondence of Laplace-Fourier transform we can show that $W(y,t)$ is the unique solution of the fractional diffusion equation. 

The real situation is more complicated, but the correct proof follows the above intuitive story.

\subsection{Time evolution of the Wigner distribution}

First we introduce the Fourier transform of the Wigner distribution. Define $\widetilde{W}_{\epsilon,+}(p,k,t):\R \times \T \times \R_{\ge 0} \to \mathbb{C} , ~ \iota = +,-$ as
\begin{align}
\widetilde{W}_{\epsilon,+}(p,k,t) &:= \frac{\epsilon}{2} \mathbb{E}_{\epsilon} [\widehat{\psi}_{\epsilon}(k - \frac{\epsilon p}{2},\frac{t}{f_{\theta,s}(\epsilon)})^{*} \widehat{\psi}_{\epsilon}(k + \frac{\epsilon p}{2},\frac{t}{f_{\theta,s}(\epsilon)})] \label{fWigner+} .
\end{align}
To get closed equations governing the dynamics of $\widetilde{W}_{\epsilon,+}$, we also introduce the Fourier transforms of the anti-Wigner distribution $\widehat{Y}_{\epsilon}$ and the $*$-Wigner distribution $\widetilde{W}_{\epsilon,-}$. Define $\widetilde{W}_{\epsilon,-}(p,k,t), \widetilde{Y}_{\epsilon,\iota}(p,k,t):\R \times \T \times \R_{\ge 0} \to \mathbb{C} , ~ \iota = +,-$ as
\begin{align}
\widetilde{W}_{\epsilon,-}(p,k,t) &:= \widetilde{W}_{\epsilon,+}(t)(p,-k,t) \label{fWigner-}, \\ 
\widetilde{Y}_{\epsilon,+}(p,k,t) &:= \frac{\epsilon}{2} \mathbb{E}_{\epsilon} [\widehat{\psi}_{\epsilon}(-k + \frac{\epsilon p}{2},\frac{t}{f_{\theta,s}(\epsilon)}) \widehat{\psi}_{\epsilon}(k + \frac{\epsilon p}{2},\frac{t}{f_{\theta,s}(\epsilon)})] \label{faWigner+}, \\
\widetilde{Y}_{\epsilon,-}(p,k,t) &:= \widetilde{Y}_{\epsilon,+}(t)(-p,k,t)^{*} \label{faWigner-} .
\end{align}
The corresponding $S^{'}(\R)$ elements $<\widetilde{W}_{\epsilon,\iota}(t),\cdot>, <\widetilde{Y}_{\epsilon,\iota}(t),\cdot>$ are defined as
\begin{align}
<\widetilde{W}_{\epsilon,\iota}(t),J> &:= \int_{\R \times \T} dpdk ~ \widetilde{W}_{\epsilon,\iota}(p,k,t) J(p)^{*}, \label{fWigneriota} \\
<\widetilde{Y}_{\epsilon,\iota}(t),J> &:= \int_{\R \times \T} dpdk ~ \widetilde{Y}_{\epsilon,\iota}(p,k,t) J(p)^{*} \label{faWigneriota}.
\end{align}
Note that $<\widetilde{W}_{\epsilon,+},\cdot>(t)$ is indeed the inverse Fourier transformation of $<W_{\epsilon,+}(t),\cdot>$ in $S^{'}(\R)$. We may abbreviate the variables $(p,k,t)$ or some of them for simplicity. From $(\ref{defofpsi})$, we see that the time evolution of $(\widetilde{W}_{\epsilon,+},\widetilde{W}_{\epsilon,-},\widetilde{Y}_{\epsilon,+},\widetilde{Y}_{\epsilon,-})$ is governed by the following differential equations:
\begin{align}\label{startpoint}
\frac{d}{dt} \widetilde{W}_{\epsilon,+} &= - \frac{\sqrt{-1} \epsilon (\delta_{\epsilon} \omega)}{f_{\theta,s}(\epsilon)} \widetilde{W}_{\epsilon,+} + \frac{\gamma}{f_{\theta,s}(\epsilon)} (\mathcal{L}_{\epsilon p} \widetilde{W}_{\epsilon,+}) - \frac{\gamma}{2 f_{\theta,s}(\epsilon)} \sum_{\iota = \pm} (\mathcal{L}^{+}_{\iota \epsilon p} \widetilde{Y}_{\epsilon,-\iota}), \\
\frac{d}{dt} \widetilde{W}_{\epsilon,-} &= \frac{\sqrt{-1} \epsilon (\delta_{\epsilon} \omega)}{f_{\theta,s}(\epsilon)} \widetilde{W}_{\epsilon,-} + \frac{\gamma}{f_{\theta,s}(\epsilon)} (\mathcal{L}_{\epsilon p} \widetilde{W}_{\epsilon,-}) - \frac{\gamma}{2 f_{\theta,s}(\epsilon)} \sum_{\iota = \pm} (\mathcal{L}^{-}_{\iota \epsilon p} \widetilde{Y}_{\epsilon,-\iota}), \\
\frac{d}{dt} \widetilde{Y}_{\epsilon,+} &= - \frac{2\sqrt{-1} \bar{\omega}_{\epsilon}}{f_{\theta,s}(\epsilon)} \widetilde{Y}_{\epsilon,+} + \frac{\gamma}{f_{\theta,s}(\epsilon)} (\mathcal{L}_{\epsilon p} \widetilde{Y}_{\epsilon,+}) - \frac{\gamma}{f_{\theta,s}(\epsilon)}[\mathcal{R}_{\epsilon p} (\widetilde{Y}_{\epsilon,+} - \widetilde{Y}_{\epsilon,-}) ] \notag \\ 
& \quad - \frac{\gamma}{2 f_{\theta,s}(\epsilon)} \sum_{\iota = \pm} (\mathcal{L}^{+}_{\iota \epsilon p} \widetilde{W}_{\epsilon,-\iota}), \\
\frac{d}{dt} \widetilde{Y}_{\epsilon,-} &= \frac{2\sqrt{-1} \bar{\omega}_{\epsilon}}{f_{\theta,s}(\epsilon)} \widetilde{Y}_{\epsilon,-} + \frac{\gamma}{f_{\theta,s}(\epsilon)} (\mathcal{L}_{\epsilon p} \widetilde{Y}_{\epsilon,-}) + \frac{\gamma}{f_{\theta,s}(\epsilon)}[\mathcal{R}_{\epsilon p} (\widetilde{Y}_{\epsilon,+} - \widetilde{Y}_{\epsilon,-}) ] \notag \\ 
& \quad - \frac{\gamma}{2 f_{\theta,s}(\epsilon)} \sum_{\iota = \pm} (\mathcal{L}^{-}_{\iota \epsilon p} \widetilde{W}_{\epsilon,-\iota}), \label{startpointe}
\end{align}
where $\delta_{\epsilon} \omega, \bar{\omega}_{\epsilon}$ are functions on $\R \times \T$ defined as
\begin{align*}
\delta_{\epsilon} \omega(p,k) &:= \frac{1}{\epsilon} ( \omega(k+\frac{\epsilon p}{2}) - \omega(k-\frac{\epsilon p}{2}) ), \\
\bar{\omega}_{\epsilon}(p,k) &:= \frac{1}{2} [ \omega(k+\frac{\epsilon p}{2}) + \omega(k-\frac{\epsilon p}{2}) ].
\end{align*}
In addition, $\mathcal{R}_{p} , \mathcal{L}_{p} , \mathcal{L}^{\pm}_{p}, p \in \R$ are operators on $\mathbb{L}^{2}(\T)$ defined as
\begin{align}
\mathcal{R}_{p}f(k) &:= \int_{\T} dk^{'} ~ R(k,k',p) f(k'), \notag \\
\mathcal{L}_{p}f(k) &:= 2 \mathcal{R}_{p}f(k) - (R(k + \frac{p}{2}) + R(k - \frac{p}{2})) f(k) , \notag \\
\mathcal{L}^{\pm}_{p}f(k) &:= 2 \mathcal{R}_{p}f(k) - 2R(k \pm \frac{p}{2})f(k) , \notag \\
R(k,k',p) &:= \frac{1}{2} \sum_{\iota = \pm 1} r(k+\frac{p}{2},k + \iota k') r(k-\frac{p}{2},k + \iota k'). \label{scatkernel}
\end{align}
and $r(k,k^{'})$ is defined in $(\ref{sqrtscatt})$. Since the above differential equations are linear and the coefficients are bounded, $(\widetilde{W}_{\epsilon,+},\widetilde{W}_{\epsilon,-},\widetilde{Y}_{\epsilon,+},\widetilde{Y}_{\epsilon,-})(p,\cdot,t) \in (\mathbb{L}^{2}(\T))^{4}$ for any $p \in \R, t \ge 0$. 
Now we simplify the righthand side of $(\ref{startpoint})$ in several steps. By expanding $R(k,k^{'},\epsilon p), R(k \pm \frac{\epsilon p}{2})$ with respect to $\epsilon p$ up to the second order and using Lemma $\ref{expressionR}$, we have
\begin{align}\label{startpoint2}
\frac{d}{dt} \widetilde{W}_{\epsilon,\iota} &= - (\operatorname{sgn} \iota) \frac{\sqrt{-1} \epsilon (\delta_{\epsilon} \omega)}{f_{\theta,s}(\epsilon)} \widetilde{W}_{\epsilon,\iota} + \frac{\gamma}{f_{\theta,s}(\epsilon)} [\mathcal{L} \{ \widetilde{W}_{\epsilon,+} - \frac{1}{2}(\widetilde{Y}_{\epsilon,+} + \widetilde{Y}_{\epsilon,-}) \} ] \notag \\ 
& \quad - (\operatorname{sgn} \iota) \frac{\gamma R^{'} \epsilon p}{2 f_{\theta,s}(\epsilon)}(\widetilde{Y}_{\epsilon,+} - \widetilde{Y}_{\epsilon,-}) + \frac{\gamma \epsilon^{2} p^{2}}{f_{\theta,s}(\epsilon)} r^{(1,\iota)}_{\epsilon}, \\
\frac{d}{dt} \widetilde{Y}_{\epsilon,\iota} &= - (\operatorname{sgn} \iota) \frac{2\sqrt{-1} \bar{\omega}_{\epsilon}}{f_{\theta,s}(\epsilon)} \widetilde{Y}_{\epsilon,+} + \frac{\gamma}{f_{\theta,s}(\epsilon)} [\mathcal{L} \{ \widetilde{Y}_{\epsilon,\iota} - \frac{1}{2}(\widetilde{W}_{\epsilon,+} + \widetilde{W}_{\epsilon,-}) \} ] \notag \\ 
& \quad - (\operatorname{sgn} \iota) \bigl\{ \frac{\gamma}{f_{\theta,s}(\epsilon)}[\mathcal{R}_{0} (\widetilde{Y}_{\epsilon,+} - \widetilde{Y}_{\epsilon,-}) ] + \frac{\gamma R^{'} \epsilon p}{2 f_{\theta,s}(\epsilon)}(\widetilde{W}_{\epsilon,+} - \widetilde{W}_{\epsilon,-}) \bigr\} + \frac{\gamma \epsilon^{2} p^{2}}{f_{\theta,s}(\epsilon)} r^{(2,\iota)}_{\epsilon}, \label{startpoint2e}
\end{align}
where $\operatorname{sgn} \iota$ is the sign of $\iota = \pm$ and $r^{(i,\iota)}_{\epsilon}(p,k,t), ~ i=1,2, ~ \iota = \pm$ are remainder terms and these satisfy
\begin{align}\label{remainderbound}
||r^{(i,\iota)}_{\epsilon}(p,\cdot,t)||_{\mathbb{L}^{2}(\T)} \lesssim \sum_{\iota = \pm} ||\widetilde{W}_{\epsilon,\iota}(p,\cdot,t)||_{\mathbb{L}^{2}(\T)} + ||\widetilde{Y}_{\epsilon,\iota}(p,\cdot,t)||_{\mathbb{L}^{2}(\T)}.
\end{align}
for any $0 < \epsilon < 1, p \in \R, t \ge 0$. The operator $\mathcal{L}$ is defined in $(\ref{jumpop})$. Define 
\begin{align}\label{defofU+}
\widetilde{U}_{\epsilon,+}(p,k,t) &:= \frac{1}{2} (\widetilde{Y}_{\epsilon,+} + \widetilde{Y}_{\epsilon,-})(p,k,t) , \\
\widetilde{U}_{\epsilon,-}(p,k,t) &:= \frac{1}{2 \sqrt{-1} } (\widetilde{Y}_{\epsilon,+} - \widetilde{Y}_{\epsilon,-})(p,k,t) \label{defofU-}.
\end{align}
From $(\ref{startpoint2})$, $(\ref{startpoint2e})$, $(\ref{defofU+})$ and $(\ref{defofU-})$, we have
\begin{align}
\frac{d}{dt} \widetilde{W}_{\epsilon,\iota} &= - (\operatorname{sgn} \iota) \frac{\sqrt{-1} \epsilon (\delta_{\epsilon} \omega)}{f_{\theta,s}(\epsilon)} \widetilde{W}_{\epsilon,+} + \frac{\gamma}{f_{\theta,s}(\epsilon)} [\mathcal{L} ( \widetilde{W}_{\epsilon,+} - \widetilde{U}_{\epsilon,+} ) ] \notag \\ & \quad - (\operatorname{sgn} \iota) \frac{\sqrt{-1} \gamma R^{'} \epsilon p}{f_{\theta,s}(\epsilon)}\widetilde{U}_{\epsilon,-} + \frac{\gamma \epsilon^{2} p^{2}}{f_{\theta,s}(\epsilon)} r^{(1,\iota)}_{\epsilon}, \label{fwig1} \\
\frac{d}{dt} \widetilde{U}_{\epsilon,+} &= \frac{2 \bar{\omega}_{\epsilon}}{f_{\theta,s}(\epsilon)} \widetilde{U}_{\epsilon,-} + \frac{\gamma}{f_{\theta,s}(\epsilon)} [\mathcal{L} \{ \widetilde{U}_{\epsilon,+} - \frac{1}{2}(\widetilde{W}_{\epsilon,+} + \widetilde{W}_{\epsilon,-}) \} ] \notag \\ & \quad  + \frac{\gamma \epsilon^{2} p^{2}}{2f_{\theta,s}(\epsilon)} (r^{(2,+)}_{\epsilon} + r^{(2,-)}_{\epsilon}) , \label{fwig3} \\
\frac{d}{dt} \widetilde{U}_{\epsilon,-} &= - \frac{2 \bar{\omega}_{\epsilon}}{f_{\theta,s}(\epsilon)} \widetilde{U}_{\epsilon,+} - \frac{2\gamma R}{f_{\theta,s}(\epsilon)} \widetilde{U}_{\epsilon,-} + \frac{\sqrt{-1} \gamma R^{'} \epsilon p}{2 f_{\theta,s}(\epsilon)}(\widetilde{W}_{\epsilon,+} - \widetilde{W}_{\epsilon,-}) \notag \\ & \quad - \frac{\sqrt{-1} \gamma \epsilon^{2} p^{2}}{2 f_{\theta,s}(\epsilon)} (r^{(2,+)}_{\epsilon} - r^{(2,-)}_{\epsilon}) \label{fwig4}.
\end{align}
By using $(\ref{fwig1}) - (\ref{fwig4})$, we have the following $\mathbb{L}^{2}(\T)$ estimate of the Wigner distribution: 
\begin{proposition}\label{forlaplace}
There exists some positive constant $C_{1}$ such that 
\begin{align*}
\mathbf{E}_{2,\epsilon}(p,T) \le \mathbf{E}_{2,\epsilon}(p,0) \exp \{C_{1} p^{2} \frac{\gamma \epsilon^{2}}{f_{\theta,s}(\epsilon)} T \},  
\end{align*}
holds for any $0 < \epsilon < 1$, $p \in \R$, $T > 0$ where
\begin{align*}
\mathbf{E}_{2,\epsilon}(p,t) := \| \widetilde{W}_{\epsilon,+}(p,\cdot,t) \|_{\mathbb{L}^{2}(\T)}^{2} + \sum_{\iota = \pm } \| \widetilde{U}_{\epsilon,\iota}(p,\cdot,t) \|_{\mathbb{L}^{2}(\T)}^{2}.
\end{align*}
\end{proposition}
\begin{proof}
By multiplying both sides of $(\ref{fwig1}), (\ref{fwig3}), (\ref{fwig4})$ by $\widetilde{W}_{\epsilon,+}^{*}, \widetilde{U}_{\epsilon,\iota}^{*}, \iota = \pm $ and $(\ref{fwig1})^{*}, (\ref{fwig3})^{*}, (\ref{fwig4})^{*}$ by $\widetilde{W}_{\epsilon,+}, \widetilde{U}_{\epsilon,\iota}, \iota = \pm $, adding them sideways, and taking the integral with respect to $k \in \T$ and $t \in [0,T]$, we have
\begin{align*}
&\mathbf{E}_{2,\epsilon}(p,T) + \frac{2 \gamma}{f_{\theta,s}(\epsilon)} \int_{0}^{T} dt ~ \mathcal{E}(\widetilde{W}_{\epsilon,+}(p,\cdot,t) - \widetilde{U}_{\epsilon,+}(p,\cdot,t)) \\ 
& \quad + \frac{4 \gamma}{f_{\theta,s}(\epsilon)} \int_{0}^{T} dt \int_{\T} dk ~ R(k) |\widetilde{U}_{\epsilon,+}(p,k,t)|^{2} \\
& \quad + \frac{2 \gamma \epsilon p}{f_{\theta,s}(\epsilon)} \int_{0}^{T} dt \int_{\T} dk ~ R^{'}(k) \operatorname{Im}(\widetilde{W}_{\epsilon,+}^{*} \widetilde{U}_{\epsilon,+})(p,k,t) \\
&= \mathbf{E}_{2,\epsilon}(p,0) + \frac{\gamma \epsilon^{2} p^{2}}{f_{\theta,s}(\epsilon)} \int_{0}^{T} dt \int_{\T} dk ~ \mathbf{R}_{\epsilon}(p,k,t),
\end{align*}
where 
\begin{align*}
\mathcal{E}(f) &:= \int_{\T} dk ~ f(k)(- \mathcal{L} f)(k) , \quad f \in \mathbb{L}^{2}(\T), \notag \\
&= \int_{\T^{2}} dkdk^{'} ~ R(k,k^{'}) |f(k) - f(k^{'})|^{2} , 
\end{align*}
and $\mathbf{R}_{\epsilon}(p,k,t)$ is a remainder term which satisfies
\begin{align*}
\| \mathbf{R}_{\epsilon}(p,\cdot,t) \|_{\mathbb{L}^{1}(\T)} \lesssim \mathbf{E}_{2,\epsilon}(p,t),
\end{align*}
for any $p \in \R, t \ge 0$. By using Young's inequality, we have
\begin{align*}
&\epsilon p \int_{0}^{T} dt \int_{\T} dk ~ R^{'}(k) \operatorname{Im}(\widetilde{W}_{\epsilon,+}^{*} \widetilde{U}_{\epsilon,+})(p,k,t) \\
& \quad \ge - 2 \int_{0}^{T} dt \int_{\T} dk ~ R(k) |\widetilde{U}_{\epsilon,+}(p,k,t)|^{2} - \frac{\epsilon^{2} p^{2}}{8} \int_{0}^{T} dt \int_{\T} dk ~ \frac{(R^{'}(k))^{2}}{R(k)} |\widetilde{W}_{\epsilon,+}(p,k,t)|^{2}.
\end{align*}
Since $\frac{(R^{'}(k))^{2}}{R(k)}$ is uniformly bounded, we have
\begin{align*}
\mathbf{E}_{2,\epsilon}(p,T) - \mathbf{E}_{2,\epsilon}(p,0) \lesssim \frac{\gamma \epsilon^{2} p^{2}}{f_{\theta,s}(\epsilon)} \int_{0}^{T} dt ~ \mathbf{E}_{2,\epsilon}(p,k,t).
\end{align*}
By using Gronwall's inequality and (\ref{initialbound2}), we conclude the proof of this proposition.
\end{proof}

\subsection{Laplace transform of the Wigner distribution}

Define 
\begin{align*}
\bar{w}_{\epsilon,\iota}(p,k,\lambda) &:= \int_{0}^{\infty} dt ~ e^{- \lambda t} \widetilde{W}_{\epsilon,\iota}(p,k,t) , \\
\bar{u}_{\epsilon,\iota}(p,k,\lambda) &:= \int_{0}^{\infty} dt ~ e^{- \lambda t} \widetilde{U}_{\epsilon,\iota}(p,k,t) , 
\end{align*}
for $\lambda > 0$. Thanks to Proposition $\ref{forlaplace}$, for $p \in \R$ and $\lambda > 0$, we have $\bar{w}_{\epsilon,\iota}(p,\cdot,\lambda) , \bar{u}_{\epsilon,\iota}(p,\cdot,\lambda) \in \mathbb{L}^{2}(\T)$ by choosing sufficiently small $\epsilon > 0$. Taking the Laplace transforms of both sides of $(\ref{fwig1}) - (\ref{fwig4})$, we see that $\bar{w},\bar{u}$ satisfies the follwoing equations: 
\begin{align}
\lambda \bar{w}_{\epsilon,\iota} - \widetilde{W}_{\epsilon,\iota}^{(0)} &= - \operatorname{sgn}(\iota) \{ \frac{\sqrt{-1} \epsilon (\delta_{\epsilon} \omega) }{f_{\theta,s}(\epsilon)} \bar{w}_{\epsilon,\iota} + \frac{\sqrt{-1} \gamma R^{'} \epsilon p}{f_{\theta,s}(\epsilon)} \bar{u}_{\epsilon,-} \} \notag \\ & \quad  + \frac{\gamma}{f_{\theta,s}(\epsilon)} [ \mathcal{L} (\bar{w}_{\epsilon,\iota} - \bar{u}_{\epsilon,+}) ] + \frac{\gamma \epsilon^{2} p^{2}}{f_{\theta,s}(\epsilon)} \mathbf{r}_{\epsilon}^{(1,\iota)}, \label{laplacewigner+} \\
\lambda \bar{u}_{\epsilon,+} - \widetilde{U}_{\epsilon,+}^{(0)} &= \frac{2 \bar{\omega}_{\epsilon}}{f_{\theta,s}(\epsilon)} \bar{u}_{\epsilon,-} + \frac{\gamma}{f_{\theta,s}(\epsilon)} [\mathcal{L} \{ \bar{u}_{\epsilon,+} - \frac{1}{2} ( \bar{w}_{\epsilon,+} + \bar{w}_{\epsilon,-} ) \} ] + \frac{\gamma \epsilon^{2} p^{2}}{f_{\theta,s}(\epsilon)} \mathbf{r}_{\epsilon}^{(2,+)}, \label{laplaceawigner+} \\
\lambda \bar{u}_{\epsilon,-} - \widetilde{U}_{\epsilon,-}^{(0)} &= - \frac{2 \bar{\omega}_{\epsilon}}{f_{\theta,s}(\epsilon)} \bar{u}_{\epsilon,+} - \frac{2 \gamma R}{f_{\theta,s}(\epsilon)} \bar{u}_{\epsilon,-} - \frac{\sqrt{-1} \gamma R^{'} \epsilon p}{2 f_{\theta,s}(\epsilon)} ( \bar{w}_{\epsilon,+} - \bar{w}_{\epsilon,-} ) \notag \\ & \quad + \frac{\gamma \epsilon^{2} p^{2}}{f_{\theta,s}(\epsilon)} \mathbf{r}_{\epsilon}^{(2,-)}, \label{laplaceawigner-}
\end{align} 
where $\widetilde{W}_{\epsilon,\iota}^{(0)} := \widetilde{W}_{\epsilon,\iota}(\cdot,\cdot,0), \widetilde{U}_{\epsilon,+}^{(0)} := \widetilde{U}_{\epsilon,\iota}(\cdot,\cdot,0) , \iota = \pm$. $\mathbf{r}_{\epsilon}^{(i,\iota)}(p,k,\lambda), i = 1,2 ~ \iota = \pm,$ are the Laplace transforms of the remainder terms $r^{(i,\iota)}_{\epsilon}(p,k,t), ~ i=1,2, ~ \iota = \pm$ and they satisfy
\begin{align}\label{esti:Laplaceremainder}
\| \mathbf{r}_{\epsilon}^{i,\iota}(p,\cdot,\lambda) \|_{\mathbb{L}^{2}(\T)}^{2} \lesssim \lambda \| \bar{w}_{\epsilon,+}(p,\cdot,\lambda) \|_{\mathbb{L}^{2}(\T)}^{2} + \lambda \sum_{\iota = \pm} \| \bar{u}_{\epsilon,\iota}(p,\cdot,\lambda) \|_{\mathbb{L}^{2}(\T)}^{2}.
\end{align}
By using the above equations, we can show the following estimate.
\begin{proposition}\label{wignerestimate}
For any positive constant $M > 0$ and a compact interval $I \subset \R_{>0}$, there exists some positive constant $C_{M,I} > 0$ such that
\begin{align}\label{Wigestimate}
&\overline{\lim_{\epsilon \to 0}} \frac{\gamma}{f_{\theta,s}(\epsilon)}  \sup_{|p| \le M, \lambda \in I} [\int_{\T} dk ~ R(k) |\bar{u}_{\epsilon,-}(p,k,\lambda)|^{2} + \mathcal{E}(\bar{w}_{\epsilon,+}(p,\cdot,\lambda)) \notag \\
& \quad \quad + \int_{\T} dk ~ R(k) (\frac{\bar{\omega}_{\epsilon}(p,k)}{\lambda f_{\theta,s}(\epsilon) + 2 \gamma R(k)})^{2} |\bar{u}_{\epsilon,+}(p,k,\lambda)|^{2} ] \le C_{M,I}. 
\end{align}
\end{proposition}
\begin{proof}
First we will show that
\begin{align}\label{wigestimate1}
\overline{\lim_{\epsilon \to 0}} \frac{\gamma}{f_{\theta,s}(\epsilon)} \sup_{|p| \le M, \lambda \in I} [ \int_{\T} dk ~ R(k) |\bar{u}_{\epsilon,-}(p,k,\lambda)|^{2} + \mathcal{E}((\bar{w}_{\epsilon,+}- \bar{u}_{\epsilon,+})(p,\cdot,\lambda)) ] \le C_{M,I}.
\end{align}
By multiplying both sides of $(\ref{laplacewigner+}), (\ref{laplaceawigner+}), (\ref{laplaceawigner-})$ by $\bar{w}_{\epsilon,+}^{*}, \bar{u}_{\epsilon,\iota}^{*}, \iota = \pm $, taking the real parts of them, adding them sideways and taking the integral with respect to $k \in \T$, we obtain
\begin{align*}
&\lambda (\| \bar{w}_{\epsilon,+}(p,\cdot,\lambda) \|_{\mathbb{L}^{2}(\T)}^{2} + \sum_{\iota = \pm} \| \bar{u}_{\epsilon,\iota}(p,\cdot,\lambda) \|_{\mathbb{L}^{2}(\T)}^{2} ) + \frac{2 \gamma}{f_{\theta,s}(\epsilon)} \int_{\T} dk ~ R(k) |\bar{u}_{\epsilon,-}(p,k,\lambda)|^{2} \\
& \quad + \frac{2 \gamma \epsilon p}{f_{\theta,s}(\epsilon)} \int_{\T} dk ~ R^{'}(k) \operatorname{Im}(\bar{w}_{\epsilon,+}^{*}\bar{u}_{\epsilon,-})(p,k,\lambda) + \frac{\gamma}{f_{\theta,s}(\epsilon)} \mathcal{E}(\bar{w}_{\epsilon,+}(p,\cdot,\lambda) - \bar{u}_{\epsilon,+}(p,\cdot,\lambda)) \\
&=  \int_{\T} dk ~ \operatorname{Re} (\widetilde{W}_{\epsilon,+}^{(0)} \bar{w}_{\epsilon,+}^{*} + \sum_{\iota = \pm} \widetilde{U}_{\epsilon,\iota}^{(0)} \bar{u}_{\epsilon,-}^{*})(p,k,\lambda) + \frac{\gamma \epsilon^{2} p^{2}}{f_{\theta,s}(\epsilon)} \int_{\T} dk ~ \mathbf{r}_{\epsilon}(p,k,\lambda),
\end{align*}
where $\mathbf{r}_{\epsilon}(p,k,\lambda)$ is a remainder term which satisfies
\begin{align*}
\| \mathbf{r}_{\epsilon}(p,k,\lambda) \|_{\mathbb{L}^{1}(\T)} \lesssim \| \bar{w}_{\epsilon,+}(p,\cdot,\lambda) \|_{\mathbb{L}^{2}(\T)}^{2} + \sum_{\iota = \pm} \| \bar{u}_{\epsilon,\iota}(p,\cdot,\lambda) \|_{\mathbb{L}^{2}(\T)}^{2} .
\end{align*}
By using Young's inequality, we have
\begin{align*}
&\epsilon p \int_{\T} dk ~ R^{'}(k) \operatorname{Im}(\bar{w}_{\epsilon,+}^{*}\bar{u}_{\epsilon,-})(p,k,\lambda) \\
& \quad \ge - 2 \int_{\T} dk ~ R(k) |\bar{u}_{\epsilon,-}(p,k,\lambda)|^{2} - \frac{\epsilon^{2} p^{2}}{8} \int_{\T} dk ~ \frac{(R^{'}(k))^{2}}{R(k)} |\bar{w}_{\epsilon,+}(p,k,\lambda)|^{2}.
\end{align*}
From the above, $(\ref{initialbound2})$ and Proposition $\ref{forlaplace}$, we get $(\ref{wigestimate1})$. 

Next we will verify that
\begin{align} \label{wigestimate2}
\overline{\lim_{\epsilon \to 0}} \frac{\gamma}{f_{\theta,s}(\epsilon)} \sup_{|p| \le M, \lambda \in I} \int_{\T} dk ~ R(k) (\frac{\bar{\omega}_{\epsilon}(p,k)}{\lambda f_{\theta,s}(\epsilon) + 2 \gamma R(k)})^{2} |\bar{u}_{\epsilon,+}(p,k,\lambda)|^{2} \le C_{M,I}.
\end{align}
From $(\ref{laplaceawigner-})$, we have 
\begin{align*}
&\frac{\bar{\omega}_{\epsilon}(p,k)}{\lambda f_{\theta,s}(\epsilon) + 2 \gamma R(k)} \bar{u}_{\epsilon,+}(p,k,\lambda) \\
& \quad = \bar{u}_{\epsilon,-} + \frac{1}{\lambda f_{\theta,s}(\epsilon) + 2 \gamma R} \{ f_{\theta,s}(\epsilon) \widetilde{U}_{\epsilon,-}^{(0)} - \frac{\sqrt{-1} \gamma R^{'} \epsilon p}{2} ( \bar{w}_{\epsilon,+} - \bar{w}_{\epsilon,-} ) + \gamma \epsilon^{2} \mathbf{r}_{\epsilon}^{(2,-)} \}
\end{align*} 
By using Shwartz inequality, we obtain
\begin{align*}
&\int_{\T} dk ~ R(k) (\frac{\bar{\omega}_{\epsilon}(p,k)}{\lambda f_{\theta,s}(\epsilon) + 2 \gamma R(k)})^{2} |\bar{u}_{\epsilon,+}(p,k,\lambda)|^{2} \notag \\
& \lesssim \int_{\T} dk ~ R(k) \{ (\frac{f_{\theta,s}(\epsilon)}{\lambda f_{\theta,s}(\epsilon) + 2 \gamma R(k)})^{2} |\widetilde{U}_{\epsilon,-}^{(0)}(p,k)|^{2} + (\frac{\gamma \epsilon^{2} p^{2}}{\lambda f_{\theta,s}(\epsilon) + 2 \gamma R(k)})^{2} |\mathbf{r}_{\epsilon}^{(2,-)}|^{2} \} \notag \\
& \quad + \int_{\T} dk ~ R(k)|\bar{u}_{\epsilon,-}(p,k,\lambda)|^{2} + \int_{\T} dk ~ R(k)(\frac{\gamma R^{'}(k) \epsilon p}{\lambda f_{\theta,s}(\epsilon) + 2 \gamma R(k)})^{2} |\bar{w}_{\epsilon,+}(p,k,\lambda)|^{2} .
\end{align*}
Since $(\frac{1}{\lambda f_{\theta,s}(\epsilon) + 2 \gamma R(k)})^{2} \le \frac{1}{8 \lambda f_{\theta,s}(\epsilon) \gamma R(k)}$, we have 
\begin{align*}
&\int_{\T} dk ~ R(k) \{ (\frac{f_{\theta,s}(\epsilon)}{\lambda f_{\theta,s}(\epsilon) + 2 \gamma R(k)})^{2} |\widetilde{U}_{\epsilon,-}^{(0)}(p,k)|^{2} + (\frac{\gamma \epsilon^{2} p^{2}}{\lambda f_{\theta,s}(\epsilon) + 2 \gamma R(k)})^{2} |\mathbf{r}_{\epsilon}^{(2,-)}|^{2} \} \\
& \quad \lesssim \frac{f_{\theta,s}(\epsilon)}{\gamma} + \frac{\gamma \epsilon^{4}}{f_{\theta,s}(\epsilon)} 
\end{align*}
for any $|p| \le M, \lambda \in I$.
By using $(\ref{wigestimate1})$ and the uniform boundness $0 \le \frac{(R^{'}(k))^{2}}{R(k)} \lesssim 1 , k \in \T$, we have
\begin{align*}
&\int_{\T} dk ~ R(k)|\bar{u}_{\epsilon,-}(p,k,\lambda)|^{2} + \int_{\T} dk ~ R(k)(\frac{\gamma R^{'}(k) \epsilon p}{\lambda f_{\theta,s}(\epsilon) + 2 \gamma R(k)})^{2} |\bar{w}_{\epsilon,+}(p,k,\lambda)|^{2} \\
& \quad \lesssim \frac{f_{\theta,s}(\epsilon)}{\gamma} + \frac{\epsilon^{2}}{\gamma} 
\end{align*}
for any $|p| \le M, \lambda \in I$. Therefore we obtain $(\ref{wigestimate2})$.

Finally we will show that
\begin{align}\label{wigestimate3}
\overline{\lim_{\epsilon \to 0}} \frac{\gamma}{f_{\theta,s}(\epsilon)} \sup_{|p| \le M, \lambda \in I} \mathcal{E}(\bar{u}_{\epsilon,+}(p,\cdot,\lambda)) \le C_{M,I}.
\end{align}
If we get $(\ref{wigestimate3})$, then from $(\ref{wigestimate1}), (\ref{wigestimate2})$ and $(\ref{wigestimate3})$ we obtain the estimate $(\ref{Wigestimate})$. Since
\begin{align*}
\mathcal{E}(\bar{u}_{\epsilon,+}(p,\cdot,\lambda)) \lesssim \int_{T} R(k) |\bar{u}_{\epsilon,+}(p,k,\lambda)|^{2},
\end{align*}
for any $p \in \R$, $\lambda > 0$, it is sufficient to show that
\begin{align}\label{wigestimate4}
\overline{\lim_{\epsilon \to 0}} \frac{\gamma}{f_{\theta,s}(\epsilon)} \sup_{|p| \le M, \lambda \in I} \int_{T} R(k) |\bar{u}_{\epsilon,+}(p,k,\lambda)|^{2} \le C_{M,I}.
\end{align}
To prove $(\ref{wigestimate4})$, we divide the domain of the integration into two parts. Since $R(k) \lesssim \frac{f_{\theta,s}(\epsilon)}{\gamma}$ on $\{ |k| \le (\frac{f_{\theta,s}(\epsilon)}{\gamma})^{\frac{1}{2}} \}$, we have
\begin{align*}
\int_{|k| \le (\frac{f_{\theta,s}(\epsilon)}{\gamma})^{\frac{1}{2}}} dk ~ R(k) |\bar{u}_{\epsilon,+}(p,k,\lambda)|^{2} \lesssim \frac{f_{\theta,s}(\epsilon)}{\gamma}.
\end{align*}
On the other hand, since $(f_{\theta,s}(\epsilon))^{-\frac{5 - \theta}{2}} \gamma^{-\frac{\theta - 1}{2}} \lesssim (\frac{\bar{\omega}_{\epsilon}(p,k)}{\lambda f_{\theta,s}(\epsilon) + 2 \gamma R(k)})^{2}$ on $\{ |k| > (\frac{f_{\theta,s}(\epsilon)}{\gamma})^{\frac{1}{2}} \}$, from $(\ref{wigestimate2})$ we have
\begin{align*}
&\int_{|k| > (\frac{f_{\theta,s}(\epsilon)}{\gamma})^{\frac{1}{2}}} dk ~ R(k) |\bar{u}_{\epsilon,+}(p,k,\lambda)|^{2} \\
&\lesssim (f_{\theta,s}(\epsilon))^{\frac{5 - \theta}{2}} \gamma^{\frac{\theta - 1}{2}} \int_{\T} dk ~ R(k) (\frac{\bar{\omega}_{\epsilon}(p,k)}{\lambda f_{\theta,s}(\epsilon) + 2 \gamma R(k)})^{2} |\bar{u}_{\epsilon,+}(p,k,\lambda)|^{2} \\
&\lesssim (f_{\theta,s}(\epsilon))^{\frac{7 - \theta}{2}} \gamma^{- \frac{3-\theta}{2}}.
\end{align*}
Therefore we have $(\ref{wigestimate4})$.
\end{proof}

\subsection{Homogenization of a limit point of the Wigner distribution}

In this subsection, we will show the following homogenization result.
\begin{theorem}
Suppose the same assumption of Theorem $\ref{main0}$. For any constant $M > 0$ and a compact interval $I \subset (0,\infty)$, we have 
\begin{align}
&\lim_{\epsilon \to 0} \sup_{\lambda \in I} \sup_{|p| \le M} \int_{\T} dk ~ |\bar{w}_{\epsilon,+}(p,k,\lambda) - w_{\epsilon,i}(p,\lambda)| = 0, \quad i=1,2, \label{homow} \\
&\lim_{\epsilon \to 0} \sup_{\lambda \in I} \sup_{|p| \le M} \int_{\T} dk ~ |\bar{u}_{\epsilon,\iota}(p,k,\lambda)| = 0, \quad \iota = \pm, \label{homou}
\end{align}
where
\begin{align*}
w_{\epsilon,i}(p,\lambda) := \int_{\T} dk ~ \mathbf{e}_{i}(k) \bar{w}_{\epsilon,+}(p,k,\lambda), \quad i=1,2 
\end{align*}
and $\mathbf{e}_{i}(k), i = 1,2$ are defined in $(\ref{defofe})$.
\end{theorem}

\begin{proof}
First we show that there is no mass concentration at $k = 0$ macroscopically, that is,
\begin{align}\label{nomass}
\lim_{\delta \to 0} \lim_{\epsilon \to 0} \sup_{\lambda \in I} \sup_{p \le M} \int_{|k| < \delta} dk ~ \sum_{\iota} |\bar{w}_{\epsilon,\iota}(p,k,\lambda)| + |\bar{u}_{\epsilon,\iota}(p,k,\lambda)| = 0. 
\end{align}
Actually, by Schwartz's inequality, we have
\begin{align*}
\Bigl( \int_{|k| < \delta} dk ~ \sum_{\iota} |\bar{w}_{\epsilon,\iota}(p,k,\lambda)|+ |\bar{u}_{\epsilon,\iota}(p,k,\lambda)| \Bigr)^{2} \le 2 \delta \int_{\T} dk ~ |\bar{w}_{\epsilon,\iota}(p,k,\lambda)|^{2} + |\bar{u}_{\epsilon,\iota}(p,k,\lambda)|^{2}
\end{align*}
for $0 < \delta < \frac{1}{2}$. From the assumption $(\ref{initialbound2})$ and Proposition $\ref{forlaplace}$, we obtain 
\begin{align*}
\int_{\T} dk ~ |\bar{w}_{\epsilon,\iota}(p,k,\lambda)|^{2} &\le \frac{1}{\lambda} \int_{0}^{\infty} dt ~ \exp \{ - \lambda t \} \int_{\T} dk ~ |\widetilde{W}_{\epsilon,+}(p,k,t)|^{2} \\
&\le \frac{\mathbf{E}_{2,\epsilon}(p,0)}{\lambda} \int_{0}^{\infty} dt ~ \exp \{ - ( \lambda -C_{1} \frac{\epsilon^{2} p^{2}}{f_{\theta,s}(\epsilon)} )t \} \\
&\le \frac{2K_{1}}{\lambda( \lambda -C_{1} \frac{\epsilon^{2} M^{2}}{f_{\theta,s}(\epsilon)} )} , \\
\int_{\T} dk ~ |\bar{u}_{\epsilon,\iota}(p,k,\lambda)|^{2} &\le \frac{2K_{1}}{\lambda( \lambda -C_{1} \frac{\epsilon^{2}M^{2}}{f_{\theta,s}(\epsilon)} )} 
\end{align*}
for sufficiently small $\epsilon > 0$ and the above estimates imply $(\ref{nomass})$. 

Thanks to $(\ref{nomass})$, to verify $(\ref{homow})$ it is sufficient to show
\begin{align}\label{sufficientw}
\lim_{\epsilon \to 0} \sup_{\lambda \in I} \sup_{|p| \le M} \int_{\T} dk ~ \mathbf{e}_{i}(k) |\bar{w}_{\epsilon,+}(p,k,\lambda) - w_{\epsilon,i}(p,\lambda)| = 0, \quad i=1,2.
\end{align}
From the definition of $w_{\epsilon,i}$ and the property $\int_{\T} dk ~ \mathbf{e}_{i}(k) = 1$, we have
\begin{align*}
&\int_{\T} dk ~ \mathbf{e}_{i}(k) |\bar{w}_{\epsilon,+}(p,k,\lambda) - w_{\epsilon,i}(p,\lambda)| \\
& \quad \le \int_{\T^{2}} dkdk' ~ \mathbf{e}_{i}(k) \mathbf{e}_{i}(k') |\bar{w}_{\epsilon,+}(p,k,\lambda) - \bar{w}_{\epsilon,+}(p,k',\lambda)| \\
& \quad \le \Bigl( \int_{\T^{2}} dkdk' ~ R(k,k')|\bar{w}_{\epsilon,+}(p,k,\lambda) - \bar{w}_{\epsilon,+}(p,k',\lambda)|^{2} \Bigr)^{\frac{1}{2}} \Bigl( \int_{\T^{2}} dkdk' ~ \frac{\mathbf{e}_{i}^{2}(k) \mathbf{e}_{i}^{2}(k')}{R(k,k')} \Bigr)^{\frac{1}{2}}.
\end{align*}
Since $\frac{\mathbf{e}_{i}^{2}(k) \mathbf{e}_{i}^{2}(k')}{R(k,k')} , i = 1,2$ are uniformly bounded in $k,k' \in \T$, by using Proposition $\ref{wignerestimate}$ we have $(\ref{sufficientw})$. 

To prove $(\ref{homou})$, it is sufficient to show that 
\begin{align*}
&\lim_{\epsilon \to 0} \sup_{\lambda \in I} \sup_{|p| \le M} \int_{\T} dk ~ R(k) |\bar{u}_{\epsilon,\iota}(p,k,\lambda)| = 0, \quad \iota = \pm,
\end{align*}
but this is a direct consequence of Proposition $\ref{wignerestimate}$.
\end{proof}

\subsection{Characterization of a limit point of the Wigner distribution}\label{subseq:character}

In this subsection we characterize a limit point of convergent subsequence of $(\bar{w}_{\epsilon,+}(p,k,\lambda))_{\epsilon}$. \textcolor{black}{From Remark \ref{re:seqcompact} and \eqref{homou}, for any $J \in \mathbb{S}(\R)$ we have
\begin{align*}
\int_{0}^{\infty} dt ~ e^{-\lambda t} <W(t),J> &= \lim_{\epsilon \to 0} \int_{0}^{\infty} dt ~ e^{-\lambda t} <W_{\epsilon}(t),J> \\
&=  \lim_{\epsilon \to 0} \int_{\R} dp \bigr( \int_{\T} dk ~ \bar{w}_{\epsilon,+}(p,k,\lambda) \bigl) \widetilde{J}(p) \\
&= \lim_{\epsilon \to 0} \int_{\R} dp ~ w_{\epsilon,i}(p,\lambda) \widetilde{J}(p) \\
&= \int_{\R} dp ~ w(p,\lambda) \widetilde{J}(p), 
\end{align*}
where $w(p,\lambda)$ is the Laplace-Fourier transform of $<W(t),\cdot> = \mu(t)(dy)$: 
\begin{align*}
w(p,\lambda) = \int_{0}^{\infty} dt ~ e^{-\lambda t} \widetilde{W}(p,t), \quad \widetilde{W}(p,t) = \int_{\R} \mu(t)(dy) ~ e^{2 \pi \sqrt{-1} p y}.
\end{align*}
In the present subsection we will show the following equation:
\begin{align}\label{laplacefouriergoal}
\int_{\R} dp ~(\lambda + C_{\theta,\gamma_{0}} |p|^{\frac{6}{7 - \theta}}) w(p,\lambda) J(p) = \int_{\R} dp ~ \widetilde{W}_{0}(p) J(p),
\end{align}
for any $J \in \mathbb{S}(\R)$ and $\lambda > 0$. By using one-to-one correspondence of the Laplace-Fourier transform, we see that any limit of convergent subsequence $W(y,t)$ is given by
\begin{align*}
\widetilde{W}(p,t) = e^{-C_{\theta,\gamma_{0}} |p|^{-\frac{6}{7-\theta}} t} \widetilde{W}_{0}(p) , 
\end{align*}
almost every $(p,t)$. We have thus established the theorem.}

Now we will derive (\ref{laplacefouriergoal}). First we rewrite $(\ref{laplacewigner+})$ as follows:
\begin{align}
D_{\epsilon} \bar{w}_{\epsilon,+} = f_{\theta,s}(\epsilon) \widetilde{W}_{\epsilon,+}^{(0)} + \frac{3\gamma}{2} \sum_{i = 1,2} \mathbf{e}_{i} w_{\epsilon,3-i} + q_{\epsilon}, \label{rewritewigner+}
\end{align}
where 
\begin{align*}
D_{\epsilon}(p,k,\lambda) &:= f_{\theta,s}(\epsilon) \lambda + 2 \gamma R(k) + \sqrt{-1} \epsilon \delta_{\epsilon} \omega(p,k), \\
q_{\epsilon}(p,k,\lambda) &:= - \sqrt{-1} \gamma R^{'} \epsilon p \bar{u}_{\epsilon,-}(p,k,\lambda) - \gamma (\mathcal{L} \bar{u}_{\epsilon,+})(p,k,\lambda) + \gamma \epsilon^{2} p^{2} \mathbf{r}_{\epsilon}^{(1,+)}(p,k,\lambda).
\end{align*}
Multiplying both sides of $(\ref{rewritewigner+})$ by $\frac{\gamma \mathbf{e}_{i}}{f_{\theta,s}(\epsilon) D_{\epsilon}}$ and taking the integral with respect to $k \in \T$, we have
\begin{align*}
&\Bigl( \int_{\T} dk ~ \frac{\gamma}{f_{\theta,s}(\epsilon)} (1 -\frac{3 \mathbf{e}_{i}\mathbf{e}_{3-i}}{2 D_{\epsilon}}) \Bigr) w_{\epsilon,i} - \Bigl( \frac{3}{2} \int_{\T} dk ~ \frac{\gamma^{2} \mathbf{e}_{i}^{2}}{D_{\epsilon}} \Bigr) w_{\epsilon,3-i} \notag \\
& \quad = \int_{\T} dk ~ \frac{\gamma \mathbf{e}_{i} \widetilde{W}_{\epsilon,+}^{(0)}}{D_{\epsilon}} + \int_{\T} dk ~ \frac{\gamma \mathbf{e}_{i} q_{\epsilon}}{f_{\theta,s}(\epsilon) D_{\epsilon}}, i=1,2.
\end{align*}
Adding sideways the above equations corresponding to both values of $i$, we have
\begin{align*}
a_{\epsilon} w_{\epsilon,1} - b_{\epsilon}(w_{\epsilon,1} - w_{\epsilon,2}) = \int_{\T} dk ~ \frac{2 \gamma R \widetilde{W}_{\epsilon,+}^{(0)}}{D_{\epsilon}} + \int_{\T} dk ~ \frac{2 \gamma R q_{\epsilon}}{f_{\theta,s}(\epsilon) D_{\epsilon}},
\end{align*}
where
\begin{align}
a_{\epsilon}(p,\lambda) &:= \int_{\T} dk ~ \frac{2 \gamma R}{f_{\theta,s}(\epsilon)} (1 - \frac{2 \gamma R}{D_{\epsilon}}) , \label{superdiffco} \\
b_{\epsilon}(p,\lambda) &:= \frac{3}{2} \int_{\T} dk ~ \frac{\gamma \mathbf{e}_{i}}{f_{\theta,s}(\epsilon)} (1 - \frac{2 \gamma R}{D_{\epsilon}}) \notag .
\end{align}
From \eqref{homow}, \eqref{homou} and the following proposition  we obtain (\ref{laplacefouriergoal}).
\begin{proposition}\label{prop:coefficient}
Suppose the same assumptions of Theorem $\ref{main0}$. For any $J \in \mathbb{S}(\R)$ and $\lambda > 0$, we have
\begin{align}
&\lim_{\epsilon \to 0} \int_{\R} dp ~ | a_{\epsilon}(p,\lambda) - \lambda - C_{\theta,\gamma_{0}} |p|^{\frac{6}{7 - \theta}} | |J(p)| = 0 \quad 2 < \theta \le 3, \label{mainterm} \\
&\lim_{\epsilon \to 0} \int_{\R} dp ~ | a_{\epsilon}(p,\lambda) - \lambda - C_{\theta,\gamma_{0}} |p|^{\frac{3}{2}} | |J(p)| = 0 \quad \theta > 3, \label{mainterm1} \\
&\lim_{\epsilon \to 0} \int_{\R} dp ~ |b_{\epsilon}(p,\lambda)(w_{\epsilon,1} - w_{\epsilon,2})(p,\lambda)| |J(p)| = 0, \label{vanish1} \\
&\lim_{\epsilon \to 0} | \int_{\R \times \T} dpdk ~ \frac{2 \gamma R(k) \widetilde{W}_{\epsilon,+}^{(0)}(p,k)}{D_{\epsilon}} J(p) - \int_{\R} dp ~ \widetilde{W}_{0}(p) J(p)| = 0, \label{initialterm} \\
&\lim_{\epsilon \to 0} \int_{\R} dp |\int_{\T} dk ~ \frac{\gamma R q_{\epsilon}}{f_{\theta,s}(\epsilon) D_{\epsilon}}| |J(p)| = 0, \label{vanish2} \end{align}
where $C_{\theta,\gamma_{0}}$ is a positive constant defined in $(\ref{diffcoeffi})$.
\end{proposition}

\subsubsection{Proof of $(\ref{initialterm})$}

From $(\ref{initialconvergence})$, it is sufficient to show that
\begin{align*}
\lim_{\epsilon \to 0} | \int_{\R \times \T} dpdk ~ (\frac{2 \gamma R}{D_{\epsilon}} - 1) \widetilde{W}_{\epsilon,+}^{(0)} J| = 0.
\end{align*}
By Schwartz's inequality, we have
\begin{align*}
| \int_{\R \times \T} dpdk ~ (\frac{2 \gamma R}{D_{\epsilon}} - 1) \widetilde{W}_{\epsilon,+}^{(0)} J|^{2} \le \Bigl( \int_{\R \times \T} dpdk ~ |\frac{2 \gamma R}{D_{\epsilon}} - 1|^{2} |J| \Bigr) \Bigl( \int_{\R \times \T} dpdk ~ |\widetilde{W}_{\epsilon,+}^{(0)}|^{2} |J| \Bigr).
\end{align*}
From $(\ref{initialbound2})$, $\int_{\R \times \T} dpdk ~ |\widetilde{W}_{\epsilon,+}^{(0)}|^{2} |J|$ is bounded above  by $K_{1} \int_{\R} dp |J|$. On the other hand, $\frac{2 \gamma R}{D_{\epsilon}}$ is uniformly bounded in $(p,k) \in \R \times \T , 0 < \epsilon <1$:
\begin{align*}
|\frac{2 \gamma R}{D_{\epsilon}}| &= |\frac{2 \gamma R D_{\epsilon}^{*}}{|D_{\epsilon}|^{2}}| \\
&= | \frac{2 \gamma R (f_{\theta,s}(\epsilon) \lambda + 2 \gamma R - \sqrt{-1} \epsilon \delta_{\epsilon} \omega)}{(f_{\theta,s}(\epsilon) \lambda + 2 \gamma R)^{2} + \epsilon^{2} (\delta_{\epsilon} \omega)^{2}} | \\
&\le \frac{2 \gamma R (f_{\theta,s}(\epsilon) \lambda + 2 \gamma R )}{(f_{\theta,s}(\epsilon) \lambda + 2 \gamma R)^{2} + \epsilon^{2} (\delta_{\epsilon} \omega)^{2}} + \frac{2 \gamma R \epsilon |\delta_{\epsilon} \omega|}{(f_{\theta,s}(\epsilon) \lambda + 2 \gamma R)^{2} + \epsilon^{2} (\delta_{\epsilon} \omega)^{2}} \\
&\le \frac{2 \gamma R}{f_{\theta,s}(\epsilon) \lambda + 2 \gamma R} + \frac{2 \gamma R \epsilon |\delta_{\epsilon} \omega|}{4 (f_{\theta,s}(\epsilon) \lambda + 2 \gamma R) \epsilon |\delta_{\epsilon} \omega| } \\
&\le \frac{5}{4},
\end{align*}
where we use a fundamental inequality $\frac{a + b}{2} \ge \sqrt{ab}, a,b \ge 0$. In addition, $\lim_{\epsilon \to 0} \sup_{\lambda \in I} \frac{2 \gamma R}{D_{\epsilon}} = 1$ for all $(p,k) \in \R \times \T$. By the dominated convergence theorem, we conclude the proof of $(\ref{initialterm})$.

\subsubsection{Proof of $(\ref{mainterm})$, $(\ref{mainterm1})$}

From $(\ref{superdiffco})$, we have
\begin{align*}
a_{\epsilon} &= \int_{\T} dk ~ \frac{2 \gamma R}{f_{\theta,s}(\epsilon)} \frac{(f_{\theta,s}(\epsilon) \lambda + \sqrt{-1} \epsilon \delta_{\epsilon} \omega) (f_{\theta,s}(\epsilon) \lambda + 2 \gamma R - \sqrt{-1} \epsilon \delta_{\epsilon} \omega)}{|D_{\epsilon}|^{2}} \\
&= \int_{\T} dk ~ \frac{2 \gamma R}{f_{\theta,s}(\epsilon)} \frac{f_{\theta,s}(\epsilon) \lambda (f_{\theta,s}(\epsilon) \lambda + 2 \gamma R) + \epsilon^{2} (\delta_{\epsilon} \omega)^{2}}{|D_{\epsilon}|^{2}} \\
&= \lambda \int_{\T} dk ~ \frac{2 \gamma R(f_{\theta,s}(\epsilon) \lambda + 2 \gamma R)}{|D_{\epsilon}|^{2}} + \int_{\T} dk ~ \frac{2 \gamma R \epsilon^{2} (\delta_{\epsilon} \omega)^{2}}{f_{\theta,s}(\epsilon) |D_{\epsilon}|^{2}},
\end{align*}
where we use the property that $k \to \delta_{\epsilon} \omega(k)$ is a odd function and $k \to R(k)$ is a even function. Since $0 \le \frac{2 \gamma R(f_{\theta,s}(\epsilon) \lambda + 2 \gamma R)}{|D_{\epsilon}|^{2}} \le 1$ and converges to $1$ as $\epsilon \to 0$, we have
\begin{align*}
\lim_{\epsilon \to 0} \bigl|\int_{\T} dk ~ \frac{2 \gamma R(f_{\theta,s}(\epsilon) \lambda + 2 \gamma R)}{|D_{\epsilon}|^{2}} - 1 \bigr| = 0
\end{align*} 
uniformly in $p \in \R$. 

From now on we will show that
\begin{align}
&\lim_{\epsilon \to 0} \int_{\R} dp ~ \bigl| I_{\epsilon}(p,\lambda) - C_{\theta,\gamma_{0}}|p|^{\frac{6}{7 - \theta}} \bigr| J(p) = 0, \label{mainterm2} \\
& \quad I_{\epsilon}(p,\lambda) := \int_{\T} dk ~ \frac{2 \gamma R \epsilon^{2} (\delta_{\epsilon} \omega)^{2}}{f_{\theta,s}(\epsilon) |D_{\epsilon}|^{2}} \notag
\end{align} 
for any $\lambda > 0$. First we change the variable as
\begin{align*}
k &\to \gamma_{0}^{\frac{2}{7 - \theta}} |p|^{- \frac{2}{7 - \theta}} (\frac{f_{\theta,s}(\epsilon)}{\epsilon^{s}})^{-\frac{1}{3}} k, \\
\T &\to \T_{\epsilon} := \{ - \frac{1}{2} \gamma_{0}^{\frac{2}{7 - \theta}} |p|^{- \frac{2}{7 - \theta}} (\frac{f_{\theta,s}(\epsilon)}{\epsilon^{s}})^{-\frac{1}{3}} \le k < \frac{1}{2} \gamma_{0}^{\frac{2}{7 - \theta}} |p|^{- \frac{2}{7 - \theta}} (\frac{f_{\theta,s}(\epsilon)}{\epsilon^{s}})^{-\frac{1}{3}} \}.
\end{align*} 
Define $k_{\epsilon} := \gamma_{0}^{- \frac{2}{7 - \theta}} |p|^{\frac{2}{7 - \theta}} (\frac{f_{\theta,s}(\epsilon)}{\epsilon^{s}})^{\frac{1}{3}} k, ~ k \in \T_{\epsilon}$,
\begin{align*}
R_{\epsilon}(k) &:= \begin{cases} \gamma_{0}^{\frac{4}{7 - \theta}} |p|^{- \frac{4}{7 - \theta}} (\frac{f_{\theta,s}(\epsilon)}{\epsilon^{s}})^{-\frac{2}{3}} R(k_{\epsilon}) \quad &2 < \theta \le 3, \\ \gamma_{0} |p|^{-1}  \epsilon^{-(1-s)} R(k_{\epsilon}) \quad &\theta > 3, \end{cases} \\
(\delta_{\epsilon} \omega)_{\epsilon}(k,p) &:= \begin{cases} \gamma_{0}^{- \frac{3 - \theta}{7 - \theta}} |p|^{- \frac{4}{7 - \theta}} (\frac{f_{\theta,s}(\epsilon)}{\epsilon^{s}})^{\frac{3 - \theta}{6}} \delta_{\epsilon} \omega(k_{\epsilon},p) \quad &2 < \theta < 3, \\ |p|^{-1} \{ - \log (\frac{f_{\theta,s}(\epsilon)}{\epsilon^{s}})^{\frac{1}{3}} \}^{-\frac{1}{2}} \delta_{\epsilon} \omega(k_{\epsilon},p) \quad &\theta = 3, \\ |p|^{-1} \delta_{\epsilon} \omega(k_{\epsilon},p) \quad &\theta > 3. \end{cases}
\end{align*}
Then we have
\begin{align*}
I_{\epsilon}(p,\lambda) = \begin{cases} \gamma_{0}^{-\frac{\theta - 1}{7 - \theta}} |p|^{\frac{6}{7 - \theta}}  \int_{\T_{\epsilon}} dk ~ \frac{2R_{\epsilon} (\delta_{\epsilon} \omega)_{\epsilon}^{2}}{[\gamma_{0}^{- \frac{3 - \theta}{7 - \theta}}(\epsilon^{2s} f_{\theta,s}(\epsilon))^{\frac{1}{3}} \lambda + 2R_{\epsilon}]^{2} + (\delta_{\epsilon} \omega)_{\epsilon}^{2}} \quad &2 < \theta \le 3, \\ \gamma_{0}^{-\frac{1}{2}} |p|^{\frac{3}{2}} \int_{\T_{\epsilon}} dk ~ \frac{2R_{\epsilon} (\delta_{\epsilon} \omega)_{\epsilon}^{2}}{[(\epsilon^{2s} f_{\theta,s}(\epsilon))^{\frac{1}{3}} \lambda + 2R_{\epsilon}]^{2} + (\delta_{\epsilon} \omega)_{\epsilon}^{2}} \quad &\theta > 3. \end{cases}
\end{align*}
From Lemma \ref{expressionR}, we see that $R_{\epsilon}(k)$ converges to $6 \pi^{2} k^{2}$ as $\epsilon \to 0$ for any $k \in \T$. On the other hand, we have
\begin{align}\label{conv:deltaomega}
\lim_{\epsilon \to 0} (\delta_{\epsilon} \omega)_{\epsilon}(p,k) = \begin{cases} \operatorname{sgn}(k) \frac{(\theta - 1)\sqrt{C(\theta)}}{2} |k|^{- \frac{3 - \theta}{2}} \quad &2 < \theta < 3, \\ \operatorname{sgn}(k) \sqrt{C(\theta)} \quad &\theta \ge 3, \end{cases}
\end{align}
for any $(p,k) \neq (0,0)$, see Appendix \ref{app:delome}. Since
\begin{align*}
\frac{R_{\epsilon} (\delta_{\epsilon} \omega)_{\epsilon}^{2}}{[\gamma_{0}^{- \frac{3 - \theta}{7 - \theta}}(\epsilon^{2s} f_{\theta,s}(\epsilon))^{\frac{1}{3}} \lambda + 2R_{\epsilon}]^{2} + (\delta_{\epsilon} \omega)_{\epsilon}^{2}} 1_{\T_{\epsilon}} &\lesssim 1_{ \{ |k| \le 1 \} } R_{\epsilon} + 1_{ \{ \T_{\epsilon} \cap |k| > 1 \} } \frac{(\delta_{\epsilon} \omega)_{\epsilon}^{2}}{R_{\epsilon}} \\
&\lesssim \begin{cases}  1_{ \{ |k| \le 1 \} } |k|^{2} + 1_{ \{ |k| > 1 \} } |k|^{-(5-\theta)} \quad &2 < \theta < 3, \\ 1_{ \{ |k| \le 1 \} } |k|^{2} + 1_{ \{ |k| > 1 \} } \frac{|\log k|}{|k|^{2}} \quad &\theta = 3, \\ 1_{ \{ |k| \le 1 \} } |k|^{2} + 1_{\{ |k| > 1 \} } |k|^{-2} \quad &\theta > 3, \end{cases}
\end{align*}
by the dominated convergence theorem, we have
\begin{align*}
&\lim_{\epsilon \to 0} \int_{\T_{\epsilon}} dk ~ \frac{2 R_{\epsilon} (\delta_{\epsilon} \omega)_{\epsilon}^{2}}{[\gamma_{0}^{- \frac{3 - \theta}{7 - \theta}}(\epsilon^{2s} f_{\theta,s}(\epsilon))^{\frac{1}{3}} \lambda + 2R_{\epsilon}]^{2} + (\delta_{\epsilon} \omega)_{\epsilon}^{2}} \\
& \quad = \int_{\R} dk ~ \frac{12 \pi^{2} (\theta - 1)^{2} C(\theta) |k|^{2}}{576 \pi^{4} |k|^{7 - \theta} + (\theta - 1)^{2} C(\theta) } & &2 < \theta \le 3, \\
&\lim_{\epsilon \to 0} \int_{\T_{\epsilon}} dk ~ \frac{2 R_{\epsilon} (\delta_{\epsilon} \omega)_{\epsilon}^{2}}{[\gamma_{0}^{- \frac{3 - \theta}{7 - \theta}}(\epsilon^{2s} f_{\theta,s}(\epsilon))^{\frac{1}{3}} \lambda + R_{\epsilon}]^{2} + (\delta_{\epsilon} \omega)_{\epsilon}^{2}} \\
& \quad = \int_{\R} dk ~ \frac{12 \pi^{2} C(\theta) |k|^{2}}{144 \pi^{4} |k|^{4} + C(\theta) } & &\theta > 3. 
\end{align*}
By using the calculus of residua, we get
\begin{align*}
\int_{\R} dk ~ \frac{|k|^{2}}{|k|^{4} + 1} \frac{1}{|k|^{\tau}} = \frac{\pi \csc{(\frac{\pi \tau}{4} + \frac{\pi}{4}})}{2} \quad 0 \le \tau < 1 
\end{align*}
and thus we obtain
\begin{align*}
&\int_{\R} dk ~ \frac{12 \pi^{2} (\theta - 1)^{2} C(\theta) |k|^{2}}{576 \pi^{4} |k|^{7 - \theta} + (\theta - 1)^{2} C(\theta) } \\ 
& \quad = \frac{48 \pi^{2}}{7-\theta} (\frac{(\theta - 1)\sqrt{C(\theta)}}{24 \pi^{2}})^{\frac{6}{7-\theta}} \int_{\R} dk ~ \frac{|k|^{2}}{|k|^{4} + 1} \frac{1}{|k|^{\frac{3(3 - \theta)}{7 - \theta}}} \\
& \quad = \frac{24 \pi^{3} \csc{(\frac{3 \pi(3 - \theta) }{4(7 - \theta)} + \frac{\pi}{4}})}{7-\theta} (\frac{(\theta - 1)}{24 \pi^{2}})^{\frac{6}{7-\theta}} C(\theta)^{\frac{3}{7-\theta}} & &2 < \theta \le 3, \\
&\int_{\R} dk ~ \frac{12 \pi^{2} C(\theta) |k|^{2}}{144 \pi^{4} |k|^{4} + C(\theta) } = \frac{\sqrt{3} C(\theta)^{\frac{3}{4}}}{6 \pi} \int_{\R} dk ~ \frac{|k|^{2}}{|k|^{4} + 1} \\
& \quad = \frac{\sqrt{6} C(\theta)^{\frac{3}{4}}}{12} & &\theta > 3.
\end{align*}
From the above arguments and the dominated convergence theorem, we obtain \eqref{mainterm2}.

\subsubsection{Proof of $(\ref{vanish1})$}

From the proof of $(\ref{mainterm})$, we see that $|b_{\epsilon}(p,\lambda) J(p)|$ is bounded by some integrable function on $\R$. By using $(\ref{homow})$, we can verify $(\ref{vanish1})$.

\subsubsection{Proof of $(\ref{vanish2})$}

First we divide $\int_{\T} dk ~ \frac{\gamma R q_{\epsilon}}{f_{\theta,s}(\epsilon) D_{\epsilon}}$ into three parts:
\begin{align*}
&\int_{\T} dk ~ \frac{\gamma R q_{\epsilon}}{f_{\theta,s}(\epsilon) D_{\epsilon}} = \sum_{i = 1}^{3} Q_{\epsilon}^{(i)}, \\
&Q_{\epsilon}^{(1)} := - \sqrt{-1} \int_{\T} dk ~ \frac{ \gamma^{2} R R^{'} \epsilon p \bar{u}_{\epsilon,-}}{f_{\theta,s}(\epsilon) D_{\epsilon}}, \quad Q_{\epsilon}^{(2)} := - \int_{\T} dk ~ \frac{\gamma^{2} R (\mathcal{L} \bar{u}_{\epsilon,+})}{f_{\theta,s}(\epsilon) D_{\epsilon}}, \\
&Q_{\epsilon}^{(3)} := \int_{\T} dk ~ \frac{\gamma^{2} R \epsilon^{2} p^{2} \mathbf{r}_{\epsilon}^{(1)}}{f_{\theta,s}(\epsilon) D_{\epsilon}}.
\end{align*}
From the boundness of $|\frac{\gamma R}{D_{\epsilon}}|$ and \eqref{esti:Laplaceremainder}, we have
\begin{align*}
|Q_{\epsilon}^{(3)}| \lesssim p^{2} \frac{ \epsilon^{2 + s} }{f_{\theta,s}(\epsilon)},
\end{align*}
and thus we obtain
\begin{align*}
\lim_{\epsilon \to 0} \int_{\R} dp ~ | Q_{\epsilon}^{(3)}(p,\lambda) J(p) | = 0.
\end{align*} 

Next we consider the limit of $Q_{\epsilon}^{(1)}(p,\lambda)$. Since $k \to \delta_{\epsilon} \omega(k)$, $k \to R^{'}(k)$ are odd and $k \to R(k)$ is even, we have
\begin{align*}
- \sqrt{-1} \int_{\T} dk ~ \frac{ \gamma^{2} R R^{'} \epsilon p \bar{u}_{\epsilon,-}}{f_{\theta,s}(\epsilon) D_{\epsilon}} &= - \sqrt{-1} \int_{\T} dk ~ \frac{ \gamma^{2} R R^{'} \epsilon p \bar{u}_{\epsilon,-}(f_{\theta,s}(\epsilon) \lambda + 2 \gamma R(k) - \sqrt{-1} \epsilon \delta_{\epsilon} \omega)}{f_{\theta,s}(\epsilon) |D_{\epsilon}|^{2}} \\
&= - \int_{\T} dk ~ \frac{ \gamma^{2}  \epsilon^{2} p R R^{'} \delta_{\epsilon} \omega}{f_{\theta,s}(\epsilon) |D_{\epsilon}|^{2}} \bar{u}_{\epsilon,-}.
\end{align*}
By using Schwartz inequality and Proposition $\ref{wignerestimate}$, we have
\begin{align*}
|Q_{\epsilon}^{(1)}|^{2} &\le \frac{\gamma^{2}  \epsilon^{2} |p|^{2}}{(f_{\theta,s}(\epsilon))^{2}} \Bigl( \int_{\T} dk ~ R |\bar{u}_{\epsilon,-}|^{2} \Bigr) \Bigl( \int_{\T} dk ~ R |\gamma \frac{R^{'} \epsilon \delta_{\epsilon} \omega}{|D_{\epsilon}|^{2}}|^{2} \Bigr) \\
&\le \frac{\gamma^{2}  \epsilon^{2} |p|^{2}}{(f_{\theta,s}(\epsilon))^{2}} \Bigl( \int_{\T} dk ~ R |\bar{u}_{\epsilon,-}|^{2} \Bigr) \Bigl( \int_{\T} dk ~ R |\frac{\gamma  R^{'}}{f_{\theta,s}(\epsilon) \lambda + 2 \gamma R}|^{2} \Bigr) \\
&\lesssim |p|^{2} \frac{\epsilon^{2+s}}{f_{\theta,s}(\epsilon)}.
\end{align*}
Therefore we obtain 
\begin{align*}
\lim_{\epsilon \to 0} \int_{\R} dp ~ | Q_{\epsilon}^{(1)}(p,\lambda) J(p) | = 0.
\end{align*} 

Finally we consider the limit of $Q_{\epsilon}^{(2)}$. Since $\int_{\T} dk ~ (\mathcal{L} f) = 0$ for any $f \in \mathbb{L}^{1}(\T)$, $k \to \delta_{\epsilon} \omega(k)$, $k \to R^{'}(k)$ are odd and $k \to R(k)$, $k \to (\mathcal{L} \bar{u}_{\epsilon,+})(k)$ are even, we have
\begin{align*}
Q_{\epsilon}^{(2)} &= - \int_{\T} dk ~ \frac{\gamma^{2} R (\mathcal{L} \bar{u}_{\epsilon,+})}{f_{\theta,s}(\epsilon) D_{\epsilon}} \\
&= - \int_{\T} dk ~ \frac{\gamma}{2f_{\theta,s}(\epsilon)} (\mathcal{L} \bar{u}_{\epsilon,+}) + \int_{\T} dk ~ \frac{\gamma (f_{\theta,s}(\epsilon) \lambda + \sqrt{-1} \epsilon \delta_{\epsilon} \omega)}{2f_{\theta,s}(\epsilon) D_{\epsilon}} (\mathcal{L} \bar{u}_{\epsilon,+}) \\
&= \int_{\T} dk ~ \frac{\gamma (f_{\theta,s}(\epsilon) \lambda + \sqrt{-1} \epsilon \delta_{\epsilon} \omega)}{2f_{\theta,s}(\epsilon) D_{\epsilon}} (\mathcal{L} \bar{u}_{\epsilon,+}) \\
&= \int_{\T} dk ~ \frac{\gamma (f_{\theta,s}(\epsilon) \lambda + \sqrt{-1} \epsilon \delta_{\epsilon} \omega)(f_{\theta,s}(\epsilon) \lambda + 2 \gamma R - \sqrt{-1} \epsilon \delta_{\epsilon} \omega)}{2f_{\theta,s}(\epsilon) |D_{\epsilon}|^{2}} (\mathcal{L} \bar{u}_{\epsilon,+}) \\
&= \int_{\T} dk ~ \frac{\gamma f_{\theta,s}(\epsilon) \lambda(f_{\theta,s}(\epsilon) \lambda + 2 \gamma R) + \gamma (\epsilon \delta_{\epsilon} \omega)^{2}}{2f_{\theta,s}(\epsilon) |D_{\epsilon}|^{2}} (\mathcal{L} \bar{u}_{\epsilon,+}) \\ &= \sum_{i=1,2,3} Q_{\epsilon}^{(2,i)},
\end{align*}
where 
\begin{align*}
Q_{\epsilon}^{(2,1)} &:= \int_{\T} dk ~ \frac{\gamma \lambda(f_{\theta,s}(\epsilon) \lambda + 2 \gamma R)}{2|D_{\epsilon}|^{2}}(\mathcal{L} \bar{u}_{\epsilon,+}) , \\
Q_{\epsilon}^{(2,2)} &:= \int_{\T} dk ~ \frac{\gamma R (\epsilon \delta_{\epsilon} \omega)^{2}}{2f_{\theta,s}(\epsilon) |D_{\epsilon}|^{2}} \bar{u}_{\epsilon,+}, \quad Q_{\epsilon}^{(2,3)} := \sum_{i=1,2} |u_{\epsilon,i}| \int_{\T} dk ~ \frac{3 \gamma \mathbf{e}_{3-i} (\epsilon \delta_{\epsilon} \omega)^{2}}{8f_{\theta,s}(\epsilon) |D_{\epsilon}|^{2}} 
\end{align*}
and $u_{\epsilon,i}(p,\lambda) := \int_{\T} dk ~ \mathbf{e}_{i}(k) \bar{u}_{\epsilon,+}(p,k,\lambda), i = 1,2$. Since 
\begin{align*}
&|\int_{\T} dk ~ \frac{\gamma \lambda(f_{\theta,s}(\epsilon) \lambda + 2 \gamma R)}{2|D_{\epsilon}|^{2}}(\mathcal{L} \bar{u}_{\epsilon,+})| \\
& \quad \le \sum_{i=1,2} |u_{\epsilon,i}| |\int_{\T} dk ~ \frac{3\gamma \mathbf{e}_{3-i} \lambda(f_{\theta,s}(\epsilon) \lambda + 2 \gamma R)}{8|D_{\epsilon}|^{2}}| + |\int_{\T} dk ~ \frac{\gamma R \lambda(f_{\theta,s}(\epsilon) \lambda + 2 \gamma R)}{2|D_{\epsilon}|^{2}} \bar{u}_{\epsilon,+}| \\
& \quad \le \frac{\lambda}{4} (\sum_{i=1,2} |u_{\epsilon,i}|) + \frac{\lambda}{4} \int_{\T} dk ~ |\bar{u}_{\epsilon,+}|,
\end{align*}
by using $(\ref{homou})$, we have
\begin{align*}
\lim_{\epsilon \to 0} \int_{\R} dp ~ | Q_{\epsilon}^{(2,1)}(p,\lambda) J(p) | = 0.
\end{align*}
Next we estimate $Q_{\epsilon}^{(2,2)}$. Thanks to Schwartz's inequality and Proposition $\ref{wignerestimate}$, we have
\begin{align*}
&\Bigl| \int_{\T} dk ~ \frac{\gamma R (\epsilon \delta_{\epsilon} \omega)^{2}}{2f_{\theta,s}(\epsilon) |D_{\epsilon}|^{2}} \bar{u}_{\epsilon,+} \Bigr|^{2} \\
& \quad \le \frac{\gamma^{2} }{4f_{\theta,s}(\epsilon)^{2}} \Bigl| \int_{\T} dk ~ R |\bar{u}_{\epsilon,+}|^{2} (\frac{\bar{\omega}_{\epsilon}}{\lambda f_{\theta,s}(\epsilon) + \gamma R})^{2} \Bigr| \Bigl| \int_{\T} dk ~ \frac{R (\epsilon \delta_{\epsilon} \omega)^{4} (\lambda f_{\theta,s}(\epsilon) + \gamma R)^{2}}{(\bar{\omega}_{\epsilon})^{2} |D_{\epsilon}|^{4}} \Bigr| \\
& \quad \le \frac{\gamma^{2} }{4f_{\theta,s}(\epsilon)^{2}} \Bigl| \int_{\T} dk ~ R |\bar{u}_{\epsilon,+}|^{2} (\frac{\bar{\omega}_{\epsilon}}{\lambda f_{\theta,s}(\epsilon) + \gamma R})^{2} \Bigr| \Bigl| \int_{\T} dk ~ \frac{R (\epsilon \delta_{\epsilon} \omega)^{2}}{(\bar{\omega}_{\epsilon})^{2}} \Bigr| \\ &\quad \lesssim \frac{\epsilon^{2+s}}{f_{\theta,s}(\epsilon)},
\end{align*}
and therefore we obtain
\begin{align*}
\lim_{\epsilon \to 0} \int_{\R} dp ~ | Q_{\epsilon}^{(2,2)}(p,\lambda) J(p) | = 0.
\end{align*}
Now we consider $Q_{\epsilon}^{(2,3)}$. From Proposition $\ref{wignerestimate}$, we have
\begin{align*}
|u_{\epsilon,i}|^{2} &\le \Bigl| \int_{\T} dk ~ R |\bar{u}_{\epsilon,+}|^{2} (\frac{\bar{\omega}_{\epsilon}}{\lambda f_{\theta,s}(\epsilon) + \gamma R})^{2} \Bigr| \Bigl| \int_{\T} dk ~ (\frac{\lambda f_{\theta,s}(\epsilon) + \gamma R}{\bar{\omega}_{\epsilon}})^{2} \frac{\mathbf{e}_{i}^{2}}{R} \Bigr| \\
&\lesssim \epsilon^{s} f_{\theta,s}(\epsilon), \quad i = 1,2.
\end{align*}
Thus we get
\begin{align*}
\sum_{i=1,2} |u_{\epsilon,i}| \int_{\T} dk ~ \frac{3 \gamma \mathbf{e}_{3-i} (\epsilon \delta_{\epsilon} \omega)^{2}}{8f_{\theta,s}(\epsilon) |D_{\epsilon}|^{2}} &\lesssim ( \epsilon^{s} f_{\theta,s}(\epsilon))^{\frac{1}{2}} \int_{\T} dk ~ \frac{\gamma R (\epsilon \delta_{\epsilon} \omega)^{2}}{f_{\theta,s}(\epsilon) |D_{\epsilon}|^{2}} \\
&\lesssim (\frac{\epsilon^{2+s}}{f_{\theta,s}(\epsilon)})^{\frac{1}{2}} \int_{\T} dk ~ |\delta_{\epsilon} \omega|,
\end{align*}
and 
\begin{align*}
\lim_{\epsilon \to 0} \int_{\R} dp ~ | Q_{\epsilon}^{(2,3)}(p,\lambda) J(p) | = 0.
\end{align*}

\section{proof of Theorem \ref{main1}}

\subsection{Proof of (1),(2)}

\subsubsection{$\star$ - weakly sequentially compactness of $\{ <W_{\epsilon,+}(\cdot), \cdot> \}_{\epsilon}$}

In this subsection we show that $\{ <W_{\epsilon,+}(\cdot), \cdot> \}_{\epsilon}$ is $\star$ - weakly sequentially compact in $C([0,T] ; \mathbb{S}(\R \times \T)^{'})$ for any $T > 0$. Since $\mathbb{S}(\R \times \T)$ is separable, we can take a dence countable subset $\{ J^{(m)} ; m \in \N \} \subset \mathbb{S}(\R \times \T)$. From $(\ref{sequentialcompact})$ and $(\ref{eBoltzmann})$ (appeared in the proof of (3)), we see that $\{ <W_{\epsilon,+}(\cdot), J> \}_{\epsilon}$ is uniformly bounded and equicontinuous for any $J \in \mathbb{S}(\R \times \T)$. Therefore, by using the diagonal argument we can find a subsequence $\{ <W_{\epsilon(n),+}(\cdot), \cdot> \}_{n \in \N}$, $\epsilon(n) \to 0, n \to \infty$ such that $\{ <W_{\epsilon(n),+}(\cdot), J^{(m)}> \}_{n \in \N}$ converges in $C([0,T])$ for all $m \in \N$. Then by using $(\ref{sequentialcompact})$ again, we see that $\{ <W_{\epsilon(n),+}(\cdot), J> \}_{n \in \N}$ converges in $C([0,T])$ for any $J \in \mathbb{S}(\R \times \T)$. 

\subsubsection{Extension to a finite positive measure}\label{subsubsq:extension}

Now we verify that the $\star$ - weak limit $<W(t),\cdot> := \lim_{n \to \infty} <W_{\epsilon(n),+}(\cdot),\cdot>$ can be extended to a finite measure $\mu(t)$ on $\R \times \T$. Note that the following discussion does not depend on the time scaling and the noise scaling for the dynamics.

First we will show that $W(\cdot)$ is multiplicatively positive, that is,
\begin{align*}
<W(t),|J|^{2}> ~ \ge 0 
\end{align*}
for any $t \ge 0$, $J \in \mathbb{S}(\R \times \T)$. Fix a $J \in \mathbb{S}(\R \times \T)$. Since $J$ is smooth,
\begin{align*}
J(\frac{\epsilon}{2}(x+x'),k) - J(\epsilon x,k) &= \frac{\epsilon}{2} \int_{0}^{1} dr ~ (x' - x) \partial_{y} J(\epsilon x + r \frac{\epsilon}{2}(x' - x),k) 
\end{align*}
for any $x,x' \in \Z$. Therefore we have 
\begin{align*}
&\left| \int_{\T} dk e^{2\pi \sqrt{-1} (x'-x) k} \left(J(\frac{\epsilon}{2} (x+x'),k)J(\frac{\epsilon}{2} (x+x'),k)^* - J(\epsilon x,k) J(\frac{\epsilon}{2} (x+x'),k)^{*} \right) \right| \\
&= \left| \frac{\epsilon}{2}(x'-x) \int_{\T} dk e^{2\pi \sqrt{-1} (x'-x) k} J(\frac{\epsilon}{2} (x+x'),k)^{*} \int_{0}^{1} dr \partial_{y} J(\epsilon x + r \frac{\epsilon}{2} (x' - x) , k) \right|.
\end{align*}
By repeating the integration by parts we have
\begin{align*}
& \left|  \int_{\T} dk e^{2\pi \sqrt{-1} (x'-x) k} J(\frac{\epsilon}{2} (x+x'),k)^{*} \int_{0}^{1} dr \partial_{y} J(\epsilon x + r \frac{\epsilon}{2} (x' - x) , k) \right| \\
& = \left | \int_{\T} dk \left(\frac{1}{2 \pi \sqrt{-1} (x' - x)}\right)^{3} e^{2\pi \sqrt{-1} (x'-x) k}  \partial_{k}^{3}[ J(\frac{\epsilon}{2} (x+x'),k)^{*} \int_{0}^{1} dr \partial_{y} J(\epsilon x + r \frac{\epsilon}{2} (x' - x) , k) ] \right| \\
&\le \frac{1}{8\pi^3|x - x'|^{3}}  \int_{\T} dk  \ | \partial_{k}^{3}[ J(\frac{\epsilon}{2} (x+x'),k)^{*} \int_{0}^{1} dr \partial_{y} J(\epsilon x + r \frac{\epsilon}{2} (x' - x) , k) ] |. 
\end{align*}
Hence, we obtain
\begin{align*}
& \left| \int_{\T} dk e^{2\pi \sqrt{-1} (x'-x) k} \left(J(\frac{\epsilon}{2} (x+x'),k)J(\frac{\epsilon}{2} (x+x'),k)^* - J(\epsilon x,k) J(\frac{\epsilon}{2} (x+x'),k)^{*} \right) \right| \\
& \le \frac{\epsilon}{16\pi^3|x - x'|^{2}}  \int_{\T} dk  \ | \partial_{k}^{3}[ J(\frac{\epsilon}{2} (x+x'),k)^{*} \int_{0}^{1} dr \partial_{y} J(\epsilon x + r \frac{\epsilon}{2} (x' - x) , k) ] | \\
& \le \frac{1}{|x - x'|^{2}}O_{J}(\epsilon)
\end{align*}
for all $x \neq x' \in \Z$  where $O_{J}(\epsilon)$ is the remainder term which satisfies $\varlimsup_{\epsilon \to 0} |\frac{O_{J}(\epsilon)}{\epsilon}| \lesssim 1$. In the same way, we can show that
\begin{align*}
\left| \int_{\T} dk e^{2\pi \sqrt{-1} (x'-x) k} \left(J(\epsilon x,k) J(\frac{\epsilon}{2} (x+x'),k)^{*} - J(\epsilon x,k) J(\epsilon x',k)^{*} \right) \right| \le \frac{1}{|x - x'|^{2}}O_{J}(\epsilon).
\end{align*}
On the other hand we have
\begin{align*}
&\frac{\epsilon}{2} \sum_{x,x' \in \Z} <\psi(x',\frac{t}{\epsilon})^{*} \psi(x,\frac{t}{\epsilon})>_{\epsilon} \int_{\T} dk ~ e^{2\pi \sqrt{-1} (x'-x) k} J(\epsilon x,k) J(\epsilon x',k)^{*} \\
&= \frac{\epsilon}{2} \int_{\T} dk ~ <|\sum_{x \in \Z} e^{- 2\pi \sqrt{-1} x k} \psi(x,\frac{t}{\epsilon}) J(\epsilon x,k)|^{2} >_{\epsilon} ~ \ge 0.
\end{align*}
Since
\begin{align*}
& \left | \int_{\T} dk e^{2\pi \sqrt{-1} (x'-x) k} |J(\frac{\epsilon}{2} (x+x'),k)|^{2} - J(\epsilon x,k) J(\epsilon x',k)^{*} \right| \\
&\le \left| \int_{\T} dk e^{2\pi \sqrt{-1} (x'-x) k} \left(J(\frac{\epsilon}{2} (x+x'),k)J(\frac{\epsilon}{2} (x+x'),k)^* - J(\epsilon x,k) J(\frac{\epsilon}{2} (x+x'),k)^{*} \right) \right| \\
& ~ + \left| \int_{\T} dk e^{2\pi \sqrt{-1} (x'-x) k} \left(J(\epsilon x,k) J(\frac{\epsilon}{2} (x+x'),k)^{*} - J(\epsilon x,k) J(\epsilon x',k)^{*} \right) \right|,
\end{align*}
by combining the above calculations we have 
\begin{align*}
<W_{\epsilon,+}(t),|J|^{2}> ~ = \frac{\epsilon}{2} \int_{\T} dk ~ <|\sum_{x \in \Z} e^{- 2\pi \sqrt{-1} x k} \psi_{1}(x,\frac{t}{\epsilon}) J(\epsilon x,k)|^{2} >_{\epsilon} + O_{J}(\epsilon),
\end{align*}
and thus
\begin{align*}
<W(t),|J|^{2}> ~ \ge \liminf_{\epsilon \to 0} <W(t)_{\epsilon,+},|J|^{2}> ~ \ge 0,
\end{align*}
for any $t \ge 0$. Therefore $W(\cdot)$ is multiplicatively positive. 

Next we show that $W(\cdot)$ is positive, that is, 
\begin{align*}
<W(t),J> ~ \ge 0
\end{align*}
for any $t \ge 0$, $J \in \mathbb{S}(\R \times \T) , J \ge 0$. Since $\{ J \in \mathbb{S}(\R \times \T) ; J \in C_{0}^{\infty}(\R \times \T) , J \ge 0 \}$ is a dense subset of $\{ J \in \mathbb{S}(\R \times \T) ; J \ge 0 \} $, it is sufficient to show the positivity on $C_{0}^{\infty}(\R \times \T)$. Fix a positive function $J \in C_{0}^{\infty}(\R \times \T)$. There exists a positive constant $M > 0$ such that the support of $J$ is a subset of $[-M,M] \times \T$. Let $a(y) \in C_{0}^{\infty}(\R), b(k) \in C^{\infty}(\T)$ be functions such that $a(y) = 1 , ~ y \in [-M,M]$. Define $J^{(m)}(y,k) \in C_{0}^{\infty}(\R \times \T) , ~ m \in \N$ as
\begin{align*}
J^{(m)}(y,k) = a(y) \sqrt{J(y,k) + \frac{1}{m} }.
\end{align*}
Then the sequence $\{ |J^{(m)}|^{2} , ~ m \in \N \}$ converges to $b(k)J(y,k)$ in the topology of $C_{0}^{\infty}(\R \times \T)$. Since the embedding of the space $C_{0}^{\infty}(\R \times \T)$ into the space $\mathbb{S}(\R \times \T)$ is continuous, $\{ |J^{(m)}|^{2} , ~ m \in \N \}$ also converges to $J(y,k)$ in the topology of $\mathbb{S}(\R \times \T)$. By the continuity of $W(t)$, we have
\begin{align*}
<W(t),J> = \lim_{m \to \infty} <W(t),|J^{(m)}|^{2}> ~ \ge 0
\end{align*}
for any $t \ge 0$. For  Therefore $W(\cdot)$ is positive. 

By the usual method, for example see Lemma 1 in Chapter 2 of \cite{GV}, we can extend the domain of $W(\cdot)$ to the space $C_{0}(\R \times \T)$. By the Riesz representation theorem there exists a family of positive measures $\{ \mu(t) ; t \ge 0 \}$ such that
\begin{align*}
<W(t),J> ~ = \int_{\R \times \T} \mu(t)(dy,dk) ~ J(y,k) ,
\end{align*}
for all $t \ge 0$ and $J \in C_{0}(\R \times \T)$. 

Finally we show that $\{ \mu(t) ; t \ge 0 \}$ are finite measures. Fix a family of non-negative functions $J^{(l)}(y,k) = J^{(l)}(y) \in \mathbb{S}, ~ l \in \N $ which satisfies $ J^{(l)}(y) \nearrow 1, ~ l \to \infty$ for any $y \in \R$. From \eqref{initialbound}, \eqref{energydisonR} and the monotone convergence theorem, we have
\begin{align*}
\sup_{t \ge 0} \mu(t)(\R \times \T) &= \sup_{t \ge 0} \lim_{l \to \infty} \int_{\R \times \T} \mu(t)(dy,dk) ~ J^{(l)}(y,k) \\
&= \sup_{t \ge 0} \lim_{l \to \infty} \lim_{\epsilon \to 0} <W_{\epsilon,+}(t),J^{(l)}> \\
&= \sup_{t \ge 0} \lim_{l \to \infty} \lim_{\epsilon \to 0} \frac{\epsilon}{2} \sum_{x \in \Z} \mathbb{E}_{\epsilon}[|\psi_{x}(\frac{t}{f_{\theta,s}(\epsilon)})|^{2}] J^{(l)}(\epsilon x) \\
&\le \frac{\sqrt{K_{1}}}{2}.
\end{align*}

\subsection{Proof of (3)}

\subsubsection{Derivation of the Boltzmann equation}

In this section we will show that
\begin{align}\label{eBoltzmann}
\partial_{t} <W_{\epsilon,+}(t),J> &= \frac{1}{2 \pi} <W_{\epsilon,+}(t), \omega^{'}(\partial_{y}J)> + \gamma_{0} <W_{\epsilon,+}(t),\mathcal{L}J> + o_{J}(1)
\end{align}
for $J \in \mathbb{S}(\R \times \T)$ where $o_{J}(1)$ is the remainder term which satisfies
\begin{align*}
\sup_{0< \epsilon < 1} |o_{J}(1)| \lesssim 1, \quad \lim_{\epsilon \to 0} |o_{J}(1)| = 0.
\end{align*}
Notice that $<W_{\epsilon,+}(t), \omega^{'}(\partial_{y}J)>$ is well-defined for any $J \in \mathbb{S}(\R \times \T)$ because from the Schwartz inequality and \eqref{forlaplace} with $p=0$ we have
\begin{align*}
&\bigl| \frac{1}{2 \pi} <W_{\epsilon,+}(t), \omega^{'}(\partial_{y}J)> \bigr|^{2} = \bigl| \int_{\R \times \T} dpdk ~ \sqrt{-1} p \omega^{'}(k) \widetilde{W}_{\epsilon,+}(p,k,t) \widetilde{J}(p,k)^{*} \bigr|^{2} \\ 
&\lesssim \bigl( \int_{\R \times \T} dpdk ~ |p \omega^{'}(k)|^{2} |\widetilde{J}(p,k)| \bigr) \bigl( \int_{\R \times \T} dpdk ~ |\widetilde{W}_{\epsilon,+}(p,k,t)|^{2} |\widetilde{J}(p,k)| \bigr) \\
&\lesssim \bigl( \int_{\R \times \T} dpdk ~ |p \omega^{'}(k)|^{2} |\widetilde{J}(p,k)| \bigr) \bigl( \int_{\R} dp ~ \sup_{k \in \T} |\widetilde{J}(p,k)| \int_{\T} dk ~ |\widetilde{W}_{\epsilon,+}(0,k,t)|^{2} \bigr) \\
&\lesssim \int_{\R \times \T} dpdk ~ |p \omega^{'}(k)|^{2} |\widetilde{J}(p,k)| < \infty. 
\end{align*}
To derive $(\ref{eBoltzmann})$, we need the following lemma.
\begin{lemma}\label{replacelemma}
\begin{align}\label{replace}
&\bigl| \int_{\R \times \T} dpdk ~ \frac{\sqrt{-1}}{\epsilon} (\delta_{\epsilon} \omega)(p,k) \widetilde{W}_{\epsilon,+}(p,k,t) \widetilde{J}(p,k)^{*} \notag \\
& \quad - \int_{\R \times \T} dpdk ~ \sqrt{-1} p \omega^{'}(k) \widetilde{W}_{\epsilon,+}(p,k,t) \widetilde{J}(p,k)^{*} \bigr| = o_{J}(1).
\end{align}
\end{lemma}

\begin{proof}

Thanks to the Schwartz inequality and \eqref{forlaplace} with $p=0$, we have
\begin{align*}
&\bigl| \int_{\R \times \T} dpdk ~ \sqrt{-1} \{ (\delta_{\epsilon} \omega)(p,k) - p \omega^{'}(k) \} \widetilde{W}_{\epsilon,+}(p,k,t) \widetilde{J}(p,k)^{*} \bigr|^{2} \\
&\lesssim \bigl( \int_{\R \times \T} dpdk ~ |(\delta_{\epsilon} \omega)(p,k) - p \omega^{'}(k)|^{2} |\widetilde{J}(p,k)| \bigr) \bigl( \int_{\R \times \T} dpdk ~ |\widetilde{W}_{\epsilon,+}(p,k,t)|^{2} |\widetilde{J}(p,k)| \bigr) \\
&\lesssim \bigl( \int_{\R \times \T} dpdk ~ |(\delta_{\epsilon} \omega)(p,k) - p \omega^{'}(k)|^{2} |\widetilde{J}(p,k)| \bigr) \bigl( \int_{\R} dp ~ \sup_{k \in \T} |\widetilde{J}(p,k)| \int_{\T} dk ~ |\widetilde{W}_{\epsilon,+}(0,k,t)|^{2} \bigr) \\
&\lesssim \int_{\R \times \T} dpdk ~ |(\delta_{\epsilon} \omega)(p,k) - p \omega^{'}(k)|^{2} |\widetilde{J}(p,k)|.
\end{align*}
Since  
\begin{align*}
|(\delta_{\epsilon} \omega)(p,k) - p \omega^{'}(k)|^{2} \lesssim |(\delta_{\epsilon} \omega)(p,k)|^{2} + p^{2} |\omega^{'}(k)|^{2} \lesssim \begin{cases} |p|^{2} |k|^{-(3 - \theta)} \quad &2 < \theta < 3, \\ |p|^{2} |\log |k||^{2} \quad &\theta = 3, \\ |p|^{2} \quad &\theta > 3, \end{cases}
\end{align*}
and $(\delta_{\epsilon} \omega)(p,k)$ converges to $p \omega^{'}(k)$, by the dominated convergence theorem we have
\begin{align*}
\int_{\R \times \T} dpdk ~ |(\delta_{\epsilon} \omega)(p,k) - p \omega^{'}(k)|^{2} |\widetilde{J}(p,k)| = o_{J}(1).
\end{align*}
\end{proof}

From Lemma $\ref{replacelemma}$, $(\ref{startpoint2})$ and $(\ref{startpoint2e})$ we have
\begin{align}
\partial_{t} <W_{\epsilon,+},J> &= \frac{1}{2 \pi} <W_{\epsilon,+},\omega^{'} (\partial_{y}J)> + \gamma_{0} <W_{\epsilon,+},\mathcal{L}J> \notag \\
& \quad  - \frac{\gamma_{0}}{2}<Y_{\epsilon,+} + Y_{\epsilon,-},\mathcal{L}J> + o_{J}(1), \label{wigdisevo} \\
\partial_{t} <Y_{\epsilon,+},J> &= -\frac{2 \sqrt{-1}}{\epsilon} <Y_{\epsilon,+},\bar{\omega}_{\epsilon}J> + \gamma_{0} <Y_{\epsilon,+},\mathcal{L}J> \notag \\
& \quad  - \frac{\gamma_{0}}{2}<W_{\epsilon,+} + W_{\epsilon,-},\mathcal{L}J> - \gamma_{0} <Y_{\epsilon,+} - Y_{\epsilon,-},\mathcal{R}_{0}J> + o_{J}(1) \label{awigdisevo}.
\end{align}
From $(\ref{awigdisevo})$, for any $T > 0$ we obtain
\begin{align*}
\lim_{\epsilon \to 0} |\int_{0}^{T} dt ~ <Y_{\epsilon,\iota},\omega J>| = 0 , \quad \iota = +,-.
\end{align*}
If $J \in \mathbb{S}(\R \times \T)$, then $\frac{\mathcal{L}J}{\omega}, \frac{\mathcal{R}_{0}J}{\omega} \in \mathbb{S}(\R \times \T)$. Therefore for any $J \in \mathbb{S}(\R \times \T)$ we have
\begin{align}\label{disappaWigner}
\lim_{\epsilon \to 0} \bigl| \int_{0}^{T} dt ~ <Y_{\epsilon,\iota},\mathcal{L}J> \bigr| + \bigl| \int_{0}^{T} dt ~ <Y_{\epsilon,\iota},\mathcal{R}_{0}J> \bigr| = 0 , \quad \iota = +,-.
\end{align} 
From $(\ref{wigdisevo})$ and $(\ref{disappaWigner})$, we can derive $(\ref{eBoltzmann})$.

Finally we show that a limit of a convergent subsequence $\mu(t)$ satisfies $\mu(t)(dy,\{ 0 \}) = 0$. From $(\ref{forlaplace})$ with $p = 0$ and Schwartz inequality, we have
\begin{align*}
|<W_{\epsilon,+}(t), J(y)f^{\lambda,0,r}(k)>| &\lesssim \| J \| \bigl( \int_{\T} dk ~ |f^{\lambda,0,r}(k)|^{2} \bigr)^{\frac{1}{2}} \lesssim \| J \| r,
\end{align*}
for any $J \in \mathbb{S}(\R)$, $0 < r < \frac{1}{2}$. By taking the limit $\epsilon \to 0$ and then $\lambda \to 0$, we have
\begin{align*}
\bigl| \int_{\R} \mu(t)(dy,[-r,r]) J(y) \bigr| \lesssim \| J \| r \to 0, \quad r \to 0. 
\end{align*}
Therefore we obtain $\mu(t)(dy,\{ 0 \}) = 0$.

\subsubsection{Uniqueness of the solution of $(\ref{boltzmann})$}

In this subsection we prove the uniqueness of the initial value problem of $(\ref{boltzmann})$. First we reduce the problem to the space-homogeneous case:

Suppose that a family of finite positive measures $\{ \mu(t) ; t \ge 0 \}$ is a solution of the Boltzmann equation $(\ref{boltzmann})$. Then 
\begin{align*}
\tilde{\mu}(t)(dy,dk) := \begin{cases} \mu(t)(dy + \frac{1}{2\pi} \omega'(k)t,dk) \quad &\text{on} \quad \R \times \T_{0}, \\ 0 \quad &\text{on} \quad \R \times \{ 0 \}, \end{cases}
\end{align*}
is a solution of the following space-homogeneous Boltzmann equation
\begin{align*}
\partial_{t} \int_{\R \times \T} d\tilde{\mu}(t) J &= \int_{\R \times \T} d\tilde{\mu}(t) \mathcal{L}J \\
\int_{\R \times \T} d\tilde{\mu}(0) J &= \int_{\R \times \T} d\tilde{\mu}_{0} J \quad \quad \tilde{\mu}(t)(dy,\{ 0 \}) = 0, 
\end{align*}
where 
\begin{align*}
\int d\tilde{\mu}(t) J &= \int \mu(t)(dy + \frac{1}{2\pi} \omega'(k)t,dk) J(y,k) \\ 
&:= \int \mu(t)(dy,dk) J(y - \frac{1}{2\pi} \omega'(k)t,k). 
\end{align*}
Conversely, if $\tilde{\mu}(t)$ is a solution of the space-homogeneous Boltzmann equation, then 
\begin{align*}
\mu(t)(dy,dk) := \begin{cases} \tilde{\mu}(t)(dy - \frac{1}{2\pi} \omega'(k)t,dk) \quad &\text{on} \quad \R \times \T_{0} \\ 0 \quad &\text{on} \quad \R \times \{ 0 \}, \end{cases}
\end{align*}
is a solution of \eqref{boltzmann}. Therefore it is sufficient to show the uniqueness of the solution for the space-homogeneous Boltzmann equation. 

Suppose that $J(y,k) = f^{\lambda,y^{*},r}(y)G(k)$, where 
\begin{align*}
f^{\lambda,y^{*},r}(y) &=\exp \left( - \frac{\lambda}{r^{2} - |y -y^{*}|^{2}} \right) 1_{B(y^{*},r)}(y), \\
B(y^{*},r) &= \{ y \in \R \ ; \ |y - y^{*}| < r \},  
\end{align*}
$y^{*} \in \R$ , $r > 0$ and $G(\cdot) \in C_{0}^{\infty}(\T)$. Note that $f^{\lambda,y^{*},r} \in C^{\infty}_{0}(\R)$, $\|f^{\lambda,y^{*},r} \|_{\infty} \le 1$ and 
\begin{align*}
\lim_{\lambda \to 0} f^{\lambda,y^{*},r}(y) = 1_{B(y^{*},r)}(y).
\end{align*}
Let $\mu(t), \nu(t)$ be solutions of the space-homogeneous Boltzmann equation with a same initial condition. Then 
\begin{align*}
&\bigl|\int d\mu(t) J - \int d\nu(t) J \bigr| \\
& = \bigl|\int d\mu(t) f^{\lambda,y^{*},r}(y) G(k) - \int d\nu(t) f^{\lambda,y^{*},r}(y) G(k) \bigr|\\
&\le \int_{0}^{t} ds \bigl| \int d(\mu(s) - \nu(s)) \cdot (f^{\lambda,y^{*},r}(y) (\mathcal{L}G)(k)) \bigr|.
\end{align*}
By taking the limit $\lambda \to 0$, we have
\begin{align*}
&\bigl| \int_{\T} G(k) (\mu(t)(B(y^{*},r),dk) - \nu(t)(B(y^{*},r),dk)) \bigr| \\
&\le \int_{0}^{t} ds \bigl| \int_{\T} (\mu(s)(B(y^{*},r),dk) - \nu(s)(B(y^{*},r),dk)) (\mathcal{L}G) \bigr| \\
&\lesssim \int_{0}^{t} ds \| \mu(s)(B(y^{*},r),dk) - \nu(s)(B(y^{*},r),dk) \| \\
\end{align*}
where $\| \cdot \|$ is the total variation for a bounded signed measure on $\T$. Hence,
\begin{align*}
& \| \mu(t)(B(y^{*},r),dk) - \nu(t)(B(y^{*},r),dk) \| \\
& \quad \quad \lesssim \int_{0}^{t} ds \| \mu(s)(B(y^{*},r),dk) - \nu(s)(B(y^{*},r),dk) \|.
\end{align*}
Therefore $\mu(t)(B(y^{*},r),dk) = \nu(t)(B(y^{*},r),dk)$ on $\T_{0}$ for any ball $B(y^{*},r) \subset \R$, which concludes $\mu(t)=\nu(t)$ for any $t \ge 0$.

\section{proof of theorem \ref{main2}}

\subsection{proof of (1)}

It is sufficient to show that in our case, \cite[Condition 2.1, 2.2, 2.3, and $(2.12)$]{KJO} are satisfied. Since our scattering operator is same with that of \cite{KJO}, \cite[Condition 2.2 and 2.3]{KJO} are satisfied. In addition, $\int_{\T}  P(\cdot,dk^{'}) \frac{\omega^{'}(k^{'})}{R(k^{'})} \in \mathbb{L}^{2}(\pi)$, \cite[$(2.12)$]{KJO} is also satisfied. 

From now on we will verify that \cite[Condition 2.1]{KJO} is satisfied. Actually, instead of \cite[Condition 2.1]{KJO}, it is sufficient to show that
\begin{align*}
\lim_{N \to \infty} N \pi(\frac{\omega^{'}(k)}{2 \gamma_{0} R(k)} > N(\theta) \lambda) = \begin{cases} C_{*}(\theta) \gamma_{0}^{-\frac{6}{7-\theta}} \lambda^{- \frac{6}{7 - \theta}} \quad &2 < \theta < 3, \\ C_{*}(\theta) \gamma_{0}^{-\frac{3}{2}} \lambda^{- \frac{3}{2}} \quad &\theta \ge 3, \end{cases}
\end{align*}
for $\lambda > 0$ where $C_{*}(\theta)$ is some positive constant. About the sufficiency, see the proof of \cite[Lemma 5.5]{KJO}. By using a change of variable $k \in \T \to N^{\frac{1}{3}} k \in N^{\frac{1}{3}} \T$, we have
\begin{align*}
N \pi(\frac{\omega^{'}(k)}{2 \gamma_{0} R(k)} > N(\theta) \lambda) &= \frac{2N}{3} \int_{\{ \frac{\omega^{'}(k)}{2 \gamma_{0} R(k)} > N(\theta) \lambda \}} dk ~ R(k) \\
&= \frac{2}{3} \int_{\{ \frac{\omega^{'}(N^{- \frac{1}{3}} k)}{2 \gamma_{0} R(N^{-\frac{1}{3}} k)} > N(\theta) \lambda \}} dk ~ N^{\frac{2}{3}} R(N^{-\frac{1}{3}} k) \\
&\to \begin{cases} 4\pi^{2} \int_{\{ \frac{(\theta - 1)\sqrt{C(\theta)}}{24 \gamma_{0} \pi^{2}} |k|^{- \frac{7-\theta}{2}} > \lambda, k > 0 \}} dk ~ |k|^{2} \quad &2 < \theta \le 3, \\ 4 \pi^{2} \int_{\{ \frac{\sqrt{C(\theta)}}{12 \gamma_{0} \pi^{2}} |k|^{- 2 } > \lambda, k > 0 \}} dk ~ |k|^{2} \quad &\theta > 3,\end{cases} \\
&= \begin{cases} C_{*}(\theta) \gamma_{0}^{-\frac{6}{7-\theta}} |\lambda|^{-\frac{6}{7-\theta}} \quad &2 < \theta \le 3, \\ C_{*}(\theta) \gamma_{0}^{-\frac{3}{2} } |\lambda|^{-\frac{3}{2} } \quad &\theta > 3, \end{cases} 
\end{align*}
where $C_{*}(\theta) := \begin{cases} \frac{4 \pi^{2}}{3} (\frac{(\theta - 1)\sqrt{C(\theta)}}{24 \pi^{2}})^{\frac{6}{7-\theta}} \quad &2 < \theta \le 3, \\ \frac{4 \pi^{2}}{3} (\frac{\sqrt{C(\theta)}}{12 \pi^{2}})^{\frac{3}{2}} \quad &\theta > 3. \end{cases}$ 
Therefore, by \cite[Theorem 2.8, Theorem 6.1]{KJO}, the finite-dimensional distributions of scaled process $\{ \frac{1}{N(\theta)} Z(Nt) ; t \ge 0 \}$ converge weakly to those of a L\'{e}vy process whose characteristic function at time 1, denoted by $\phi(y)$, is given by
\begin{align*}
\phi(y) := \begin{cases} \exp \bigl( \frac{9}{7 - \theta} \int_{\R} d\lambda ~ ( e^{\sqrt{-1}\lambda y} - 1 - \sqrt{-1} \lambda y ) C_{*}(\theta) \gamma_{0}^{-\frac{\theta - 1}{7-\theta}} \Gamma(\frac{6}{7-\theta} + 1) |\lambda|^{-\frac{6}{7-\theta} - 1} \bigr) &2 < \theta \le 3, \\ \exp \bigl( \frac{9}{4} \int_{\R} d\lambda ~ ( e^{\sqrt{-1}\lambda y} - 1 - \sqrt{-1} \lambda y ) C_{*}(\theta) \gamma_{0}^{-\frac{1}{2} } \Gamma(\frac{3}{2} + 1) |\lambda|^{-\frac{3}{2} - 1} \bigr) &\theta > 3 \end{cases}
\end{align*}
where $\Gamma(a), a > -1$ is the gamma function. In addition, the generator of the L\'{e}vy process is $- \frac{c_{\theta,\gamma_{0}}}{(2\pi)^{\frac{6}{7-\theta}}} (-\Delta)^{\frac{3}{7-\theta}}$ if $2 < \theta \le 3$, and $- \frac{c_{\theta,\gamma_{0}}}{(2\pi)^{\frac{3}{2}}} (-\Delta)^{\frac{3}{4}}$ if $\theta > 3$, where $c_{\theta,\gamma_{0}}$ is given by
\begin{align*}
c_{\theta,\gamma_{0}} &:= \begin{cases} \frac{24 \pi^{2}}{7-\theta} \gamma_{0}^{- \frac{\theta - 1}{7 - \theta}} (\frac{(\theta - 1)\sqrt{C(\theta)}}{24 \pi^{2}})^{\frac{6}{7-\theta}} \Gamma(\frac{6}{7-\theta} + 1) \int_{\R} dy ~ (1-\cos y) |y|^{-\frac{6}{7-\theta} - 1} \quad &2 < \theta \le 3, \\
 \frac{\sqrt{3}}{12 \pi} \gamma_{0}^{- \frac{1}{2}} C(\theta)^{\frac{3}{4}} \Gamma(\frac{3}{2} + 1) \int_{\R} dy ~ (1-\cos y) |y|^{-\frac{3}{2} - 1} \quad &\theta > 3. \end{cases} 
\end{align*}

Finally, we show $c_{\theta,\gamma_{0}} = C_{\theta,\gamma_{0}}$, where $C_{\theta,\gamma_{0}}$ is given by (\ref{diffcoeffi}). In Appendix \ref{ap:toGamma} we show
\begin{align}
\begin{cases} \int_{\R} dy ~ (1-\cos y) |y|^{-\frac{6}{7-\theta} - 1} = \frac{7 - \theta}{3} \cos{\frac{3 \pi}{7-\theta}} \Gamma(1-\frac{6}{7 - \theta}) &2 < \theta \le 3, \\
\int_{\R} dy ~ (1-\cos y) |y|^{-\frac{3}{2} - 1} = \frac{4}{3} \cos{\frac{3 \pi}{4}} \Gamma(1-\frac{3}{2}) &\theta > 3. \end{cases} \label{toGamma}
\end{align}
From $(\ref{toGamma})$ and Euler's reflection formula $\Gamma(1+y) \Gamma(1-y) = \pi y \csc{\pi y}, ~ y \in \R$, we have
\begin{align*}
&\Gamma(\frac{6}{7-\theta} + 1) \int_{\R} dy ~ (1-\cos y) |y|^{-\frac{6}{7-\theta} - 1} = \pi \csc{\frac{3 \pi}{7 - \theta}} & 2 < \theta \le 3,\\
&\Gamma(\frac{3}{2} + 1) \int_{\R} dy ~ (1-\cos y) |y|^{-\frac{3}{2} - 1} = \sqrt{2} \pi & \theta > 3.
\end{align*}
Hence we obtain $c_{\theta,\gamma_{0}} = C_{\theta,\gamma_{0}}$.

\subsection{proof of $(2)$}

From Theorem $\ref{main2} ~ (1)$, it suffices to show that 
\begin{align}\label{homogenizeonT}
\lim_{N \to \infty} \int_{\T} dk ~ \bigl| u_N(N(\theta) y,k,Nt) - \mathbb{E}_{k} [\int_{\T} dk^{'} u_{0}(y + \frac{1}{N(\theta)} Z(Nt),k^{'})] \bigr|^{2} = 0. 
\end{align}
By using the Fourier transform, we obtain another representation of $u_{N}(y,k,t)$. 
\begin{align*}
u_N(N(\theta) y,k,Nt) &= \mathbb{E}_{k}[u_{0}(y + \frac{1}{N(\theta)} Z(Nt),K(Nt))] \\
&= \sum_{x \in \Z} \int_{\R} dp ~ \widetilde{u_{0}} (p,x) \mathbb{E}_{k}[e^{2 \pi  \sqrt{-1} p Z_{N}(Nt)} e^{2 \pi \sqrt{-1} x K(Nt)} ] ,
\end{align*}
where $Z_{N}(Nt) := y + \frac{1}{N(\theta)} Z(Nt)$. Hence we have
\begin{align*}
&\bigl| u_N(N(\theta) y,k,Nt) - \mathbb{E}_{k} [\int_{\T} dk^{'} u_{0}(y + \frac{1}{N(\theta)} Z(Nt),k^{'})] \bigr|^{2} \notag \\
&\lesssim \sum_{x \in \Z} \int_{\R} dp ~ \bigl| \widetilde{u_{0}} (p,x) \bigr| \bigl| \mathbb{E}_{k}[e^{\sqrt{-1} p Z_{N}(Nt)} (e^{2 \pi \sqrt{-1} x K(Nt)} - \int_{\T} dk^{'} e^{2 \pi \sqrt{-1} x k^{'}} )] \bigr|^{2}.
\end{align*}
Thanks to the assumption $u_{0} \in C^{\infty}_{0}(\R \times \T)$ and the dominated convergence theorem, if we get
\begin{align}\label{goalofmain2}
\lim_{N \to \infty} \int_{\T} dk ~ \bigl| \mathbb{E}_{k}[e^{2 \pi \sqrt{-1} p Z_{N}(Nt)} (e^{2 \pi \sqrt{-1} x K(Nt)} - \int_{\T} dk^{'} e^{2 \pi \sqrt{-1} x k^{'}} )] \bigr|^{2} = 0
\end{align}
for any $p \in \R , x \in \Z$, then we obtain $(\ref{homogenizeonT})$. 

From now on we will show $(\ref{goalofmain2})$. We use a trick to consider the convergence of $Z_{N}(Nt)$ and $K(Nt)$ separately. Let $\{ m_{N} \}$ be a family of increasing positive numbers which satisfies
\begin{align*}
\lim_{N \to \infty} m_{N} = \infty, \quad \lim_{N \to \infty} m_{N} N(\theta)^{-1} = 0.
\end{align*}
Then we have
\begin{align*}
&\bigl| \mathbb{E}_{k}[e^{2 \pi \sqrt{-1} p Z_{N}(Nt)} e^{2 \pi \sqrt{-1} x K(Nt)} - e^{2 \pi \sqrt{-1} p Z_{N}(Nt - m_{N}t)} e^{2 \pi \sqrt{-1} x K(Nt)} ] \bigr|^{2} \\
&\le \mathbb{E}_{k}[\bigl| e^{2 \pi \sqrt{-1} p ( Z_{N}(Nt) - Z_{N}(Nt - m_{N}t) )} - 1 \bigr|^{2}] \\
&\le p^{2} \mathbb{E}_{k}[\bigl| Z_{N}(Nt) - Z_{N}(Nt - m_{N}t) \bigr|^{2}],
\end{align*}
and
\begin{align*}
&\int_{\T} dk ~ p^{2} \mathbb{E}_{k}[\bigl| Z_{N}(Nt) - Z_{N}(Nt - m_{N}t) \bigr|^{2}] \\ 
&\le p^{2} \frac{ m_{N}t}{N(\theta)^{2}} \int_{Nt - m_{N}t}^{Nt} ds \int_{\T} dk ~  \mathbb{E}_{k}[|\omega^{'}(K(s))|^{2}] \\
&= p^{2} \frac{m_{N}^{2}t^{2}}{N(\theta)^{2}} \int_{\T} dk |\omega^{'}(k)|^{2} \to 0, \quad N \to \infty.
\end{align*}
Here we use the fact that the uniform probability measure on $\T$ is the reversible probability measure of $\{ K(t) ; t \ge 0 \}$. Hence we can replace $\mathbb{E}_{k}[e^{2 \pi \sqrt{-1} p Z_{N}(Nt)} e^{2 \pi \sqrt{-1} x K(Nt)}]$ by $\mathbb{E}_{k}[e^{2 \pi \sqrt{-1} p Z_{N}(Nt - m_{N}t)} e^{2 \pi \sqrt{-1} x K(Nt)}]$. By the same argument, we can also replace $\mathbb{E}_{k}[e^{2 \pi \sqrt{-1} p Z_{N}(Nt)} \int_{\T} dk^{'} e^{2 \pi \sqrt{-1} x k^{'}} ]$ by $\mathbb{E}_{k}[e^{2 \pi \sqrt{-1} p Z_{N}(Nt - m_{N}t)} \int_{\T} dk^{'} e^{2 \pi \sqrt{-1} x k^{'}} ]$. In addition, by using the Markov property we have
\begin{align*}
&\bigl| \mathbb{E}_{k}[e^{\sqrt{-1} p Z_{N}(Nt - m_{N}t)} (e^{2 \pi \sqrt{-1} x K(Nt)} - \int_{\T} dk^{'} e^{2 \pi \sqrt{-1} x k^{'}})] \bigr|^{2} \\
&= \bigl| \mathbb{E}_{k}[e^{2 \pi \sqrt{-1} p Z_{N}(Nt - m_{N}t)} \mathbb{E}_{K(Nt - m_{N}t)}[e^{2 \pi \sqrt{-1} x K(m_{N}t)} - \int_{\T} dk^{'} e^{2 \pi \sqrt{-1} x k^{'}} ]] \bigr|^{2} \\
&\le \mathbb{E}_{k}[ \bigl| \mathbb{E}_{K(Nt - m_{N}t)}[e^{2 \pi \sqrt{-1} x K(m_{N}t)} - \int_{\T} dk^{'} e^{2 \pi \sqrt{-1} x k^{'}} ] \bigr|^{2} ].
\end{align*}
Let $h(k) := e^{2 \pi \sqrt{-1} x k} - \int_{\T} dk^{'} e^{2 \pi \sqrt{-1} x k^{'}}$ and denote by $\{ P^{t} ; t \ge 0 \}$ the semigroup generated by $2 \gamma_{0} \mathcal{L}$. Since $0$ is simple eigenvalue of $2 \gamma_{0} \mathcal{L}$ and the uniform probability measure on $\T$ is the reversible probability measure, we have $\lim_{t \to \infty} \| P^{t} f \|^{2}_{\mathbb{L}^{2}(\T)} \to 0$ for any $f \in \mathbb{L}^{2}(\T), \int_{\T} dk ~ f(k) = 0$ by the ergodic theorem. Hence we obtain
\begin{align*}
&\int_{\T} dk ~ \mathbb{E}_{k}[ \bigl| \mathbb{E}_{K(Nt - m_{N}t)}[e^{2 \pi \sqrt{-1} x K(m_{N}t)} - \int_{\T} dk^{'} e^{2 \pi \sqrt{-1} x k^{'}} ] \bigr|^{2} ] \\
&= \| P^{m_{N}t} h \|^{2}_{\mathbb{L}^{2}(\T)} \to 0, \quad N \to \infty.
\end{align*}

Summarizing the above, we have
\begin{align*}
&\int_{\T} dk ~ \bigl| \mathbb{E}_{k}[e^{2 \pi \sqrt{-1} p Z_{N}(Nt)} (e^{2 \pi \sqrt{-1} x K(Nt)} - \int_{\T} dk^{'} e^{2 \pi \sqrt{-1} x k^{'}} )] \bigr|^{2} \\
&\lesssim \int_{\T} dk ~ | \mathbb{E}_{k}[e^{2 \pi \sqrt{-1} p Z_{N}(Nt)} e^{2 \pi \sqrt{-1} x K(Nt)} - e^{2 \pi \sqrt{-1} p Z_{N}(Nt - m_{N}t)} e^{2 \pi \sqrt{-1} x K(Nt)} ] \bigr|^{2} \\
& \quad + \int_{\T} dk ~ \bigl| \mathbb{E}_{k}[e^{2 \pi \sqrt{-1} p Z_{N}(Nt - m_{N}t)} (e^{2 \pi \sqrt{-1} x K(Nt)} - \int_{\T} dk^{'} e^{2 \pi \sqrt{-1} x k^{'}})] \bigr|^{2} \\
& \quad + \int_{\T} dk ~ | \mathbb{E}_{k}[e^{2 \pi \sqrt{-1} p Z_{N}(Nt - m_{N}t)} \int_{\T} dk^{'} e^{2 \pi \sqrt{-1} x k^{'}} - e^{2 \pi \sqrt{-1} p Z_{N}(Nt)} \int_{\T} dk^{'} e^{2 \pi \sqrt{-1} x k^{'}} ] \bigr|^{2} \\
& \to 0, \quad N \to \infty.
\end{align*}
We have thus established the theorem.

\section*{Acknowledgement}

We are grateful to Keiji Saito for insightful discussions. HS was supported by JSPS KAKENHI Grant Number JP19J11268 and the Program for Leading Graduate Schools, MEXT, Japan.  

\appendix

\section{Existence and Uniqueness of the solution of ($\ref{defofpsi}$)}\label{app:exuni}

Let $(E, \mathcal{F}, \mathbb{P})$ be a probability space and $B$ be a cylindrical Wiener process in $\mathbb{L}^{2}(\T)$ defined on $(E, \mathcal{F}, \mathbb{P})$. For any $T > 0$, we introduce a Banach space $\mathcal{H}_{T}$ defined as
\begin{align*}
\mathcal{H}_{T} &:= \{ f : \T \times [0,T] \times \Omega \to \mathbb{C} ; \| f \|_{\mathcal{H}} := \bigr( \sup_{0 \le t \le T} \mathbb{E}[\|f(t)\|_{\mathbb{L}^{2}(\T)}^{2}] \bigl)^{\frac{1}{2}} < \infty \}.
\end{align*}
Fix a $\widehat{\psi}_{0} \in \mathbb{L}^{2}(\T)$. We define a mapping $I_{T} : \mathcal{H}_{T} \to \mathcal{H}_{T}$ as 
\begin{align*}
I(f)(t) &:= \widehat{\psi}_{0}(k) - \int_{0}^{t} ds ~ \sqrt{-1} \omega(k) f(k,s) + \gamma R(k) \{ f(k,s) - f^{*}(-k,s) \} ] \\
& ~ + \sqrt{-1} \sqrt{\gamma} \int_{0}^{t} \int_{\T} B(dk^{'},ds) ~ r(k,k') [f(k-k',s) - f{*}(k'-k,s)] 
\end{align*} 
If $T > 0$ is sufficiently small, then $I$ is contractive and there exists the unique fixed point $\widehat{\psi} \in \mathcal{H}_{T}$. The case of general $T > 0$ can be handled in the usual way.

\section{Auxiliary Results}

\begin{lemma}\label{expressionR}
\begin{align}
R(k,k',p) &= 8(\sin^{2}{\pi k} - \sin^{2}{\frac{\pi p}{2}})(\sin^{2}{\pi k'} - \sin^{2}{\frac{\pi p}{2}})  (\sin^{2}{\pi(k+k')} + \sin^{2}{\pi(k-k')} - 2\sin^{2}{\pi p}) \notag \\
&=: R(k,k') + \sin^{2}{\frac{\pi p}{2}} R_{1}(k,k') + \sin^{4}{\frac{\pi p}{2}} R_{2}(k,k') \notag \\
& \quad + \sin^{6}{\frac{\pi p}{2}} R_{3}(k,k') + \sin^{8}{\frac{\pi p}{2}} R_{4}(k,k') ,\\
R(k,k') &= \frac{3}{4} \sum_{i = 1,2} \mathbf{e}_{i}(k)\mathbf{e}_{3 - i}(k'), 
\end{align}
where $R(k,k^{'},p)$ is defined in $(\ref{scatkernel})$ and $\mathbf{e}_{i}, i=1,2$ are defined as
\begin{align}\label{defofe}
\mathbf{e}_{1}(k) := \frac{8}{3} \sin^{4}(\pi k) , \quad \mathbf{e}_{2}(k) := 2 \sin^{2}(2\pi k).
\end{align}
\end{lemma}

\begin{remark}
Notice that $\int_{\T} dk^{'} ~ R(k,k^{'}) = R(k)$, where $R(k)$ is defined in $(\ref{defofmeanscat})$.
Functions $\mathbf{e}_{i} , i=1,2$ are normalized to satisfy $\int_{\T} dk ~ \mathbf{e}_{i}(k) = 1$. 
\end{remark}

\begin{proof}

\begin{align*}
&\frac{1}{4} r(k-\frac{p}{2},k-k^{'})r(k+\frac{p}{2},k-k^{'}) \notag \\
&= \sin^{2}{\pi(k+\frac{p}{2})} \sin^{2}{\pi(k-\frac{p}{2})} \sin{2\pi(k^{'}+\frac{p}{2})} \sin{2\pi(k^{'}-\frac{p}{2})} \\
& \quad + \sin^{2}{\pi(k^{'}+\frac{p}{2})} \sin^{2}{\pi(k^{'}-\frac{p}{2})} \sin{2\pi(k +\frac{p}{2})} \sin{2\pi(k-\frac{p}{2})} \\
& \quad + \sin^{2}{\pi(k+\frac{p}{2})} \sin^{2}{\pi(k^{'}-\frac{p}{2})} \sin{2\pi(k^{'} +\frac{p}{2})} \sin{2\pi(k-\frac{p}{2})} \\
& \quad + \sin^{2}{\pi(k^{'}+\frac{p}{2})} \sin^{2}{\pi(k-\frac{p}{2})} \sin{2\pi(k +\frac{p}{2})} \sin{2\pi(k^{'}-\frac{p}{2})} \\
&= (\sin^{2}{\pi k} - \sin^{2}{\frac{\pi p}{2}})(\sin^{2}{\pi k'} - \sin^{2}{\frac{\pi p}{2}}) \\
& \quad \times \begin{cases}  4\{ \sin^{2}{\pi k} - \sin^{2}{\frac{\pi p}{2}} - (\sin^{2}{\pi k^{'}} + \sin^{2}{\frac{\pi p}{2}})(\sin^{2}{\pi k} - \sin^{2}{\frac{\pi p}{2}}) \} \\
\quad + 4 \{ \sin^{2}{\pi k^{'}} - \sin^{2}{\frac{\pi p}{2}} - (\sin^{2}{\pi k} + \sin^{2}{\frac{\pi p}{2}})(\sin^{2}{\pi k^{'}} - \sin^{2}{\frac{\pi p}{2}}) \} \\
\quad + (\sin{2\pi k} + \sin{\pi p })(\sin{2\pi k^{'}} - \sin{\pi p}) \\
\quad + (\sin{2\pi k^{'}} + \sin{\pi p })(\sin{2\pi k} - \sin{\pi p}) \end{cases} \\
&= 4 (\sin^{2}{\pi k} - \sin^{2}{\frac{\pi p}{2}})(\sin^{2}{\pi k'} - \sin^{2}{\frac{\pi p}{2}})(\sin^{2}{\pi(k + k^{'})} - \sin^{2}{\pi p}).
\end{align*}
Therefore we have
\begin{align*}
R(k,k',p) &= \frac{1}{2} \sum_{\iota = \pm } r(k-\frac{p}{2},k + \iota k^{'})r(k+\frac{p}{2},k + \iota k^{'}) \\
&= 8(\sin^{2}{\pi k} - \sin^{2}{\frac{\pi p}{2}})(\sin^{2}{\pi k'} - \sin^{2}{\frac{\pi p}{2}}) \notag \\
& \quad \times (\sin^{2}{\pi(k+k')} + \sin^{2}{\pi(k-k')} - 2\sin^{2}{\pi p}). 
\end{align*}

\end{proof}

\begin{lemma}\label{asympofa}
\begin{align}
&\widehat{\a}(k) = \begin{cases} (4 \pi^{\theta - 1} \int_{0}^{\infty} dy ~ \frac{\sin^{2} y}{|y|^{\theta - 1}}) |k|^{\theta - 1} + O(k^{2}), \quad 2 < \theta < 3, \\
4 \pi^{2} k^{2} \bigl| \log |k| \bigr| + O(k^{2}), \quad \theta = 3, \\
4 \pi^{2} (\sum_{x \ge 1} |x|^{2 - \theta}) k^{2} + o(k^{2}), \quad \theta > 3. \end{cases} \label{asympofa1} \end{align}
\end{lemma}

\begin{proof}
First we show (\ref{asympofa1}). When $\theta > 3$ (\ref{asympofa1}) is obvious, so we only consider the case $2 < \theta \le 3$. 
\begin{align*}
\widehat{\a}(k) &= \sum_{x \neq 0} e^{-2\pi k x} |x|^{-\theta} = 4 \sum_{x \in \N} \frac{\sin^{2}{\pi k x}}{|x|^{\theta}} \\
&= 4 \int_{1}^{\infty} dy ~ \frac{\sin^{2} \pi k y}{|y|^{\theta}} + 4 \sum_{x \in \N} \bigl( \int_{x}^{x+1} dy ~  \frac{\sin^{2}{\pi k x}}{|x|^{\theta}} - \frac{\sin^{2} \pi k y}{|y|^{\theta}} \bigr) \\
&= 4 \pi^{\theta - 1} |k|^{\theta - 1} \bigl[ \int_{|k|}^{\infty} dy ~ \frac{\sin^{2} y}{|y|^{\theta}} - \sum_{x \in \N} \int_{\pi |k| x}^{\pi |k| (x+1)} dy \int_{\pi |k| x}^{y} dy^{'} ~ \partial_{y^{'}}(\frac{\sin^{2} y^{'}}{|y^{'}|^{\theta}}) \bigr] .
\end{align*}
Since 
\begin{align*}
&\bigl| (\frac{\sin^{2} y^{'}}{|y^{'}|^{\theta}})^{'} \bigr| \le \frac{2+\theta}{|\pi k x|^{\theta -1}}, \quad \pi |k| x \le y^{'} \le y \\ 
&|y - \pi |k| x| \le \pi |k| , \quad \pi |k| x \le y \le \pi |k| (x+1),
\end{align*}
we have
\begin{align*}
&\int_{\pi k x}^{\pi k (x+1)} dy \int_{\pi k x}^{y} dy^{'} ~ \partial_{y^{'}} (\frac{\sin^{2} y^{'}}{|y^{'}|^{\theta}}) \lesssim \int_{\pi k x}^{\pi k (x+1)} dy ~ \frac{y - \pi k x}{|\pi k x|^{\theta -1}} \lesssim \frac{|k|^{3 - \theta}}{|x|^{\theta - 1}}.
\end{align*}
Therefore we obtain
\begin{align*}
4 \pi^{\theta - 1} |k|^{\theta - 1} \sum_{x \in \N} \int_{\pi |k| x}^{\pi |k| (x+1)} dy \int_{\pi |k| x}^{y} dy^{'} ~ \partial_{y^{'}}(\frac{\sin^{2} y^{'}}{|y^{'}|^{\theta}}) = O(k^{2}). 
\end{align*}
On the other hand, we get
\begin{align*}
&\lim_{k \to 0} \int_{|k|}^{\infty} dy ~ \frac{\sin^{2} y}{|y|^{\theta}} = \int_{0}^{\infty} ~ \frac{\sin^{2} y}{|y|^{\theta}} \quad 2 < \theta < 3, \\ 
&\lim_{k \to 0} \frac{1}{\bigl| \log |k| \bigr|} \int_{|k|}^{\infty} dy ~ \frac{\sin^{2} y}{|y|^{\theta}} = 1 \quad \theta = 3.
\end{align*}
Hence we have (\ref{asympofa1}) when $2 < \theta \le 3$.

\end{proof}

\section{Proof of Proposition \ref{energywave}}\label{app:energywave}

For notational simplicity, we omit the variable $t \ge 0$. From $\eqref{defofpotential}$, we have
\begin{align*}
\phi_{x} &:= \mathbb{E}_{\epsilon}[ \frac{1}{2} |\psi_{x}|^{2} - e_{x}] \\
&= \frac{1}{16} \int_{\T^{2}} dkdk^{'} e^{2 \pi \sqrt{-1} (k+k^{'}) x} F_{1}(k,k^{'}) \mathbb{E}_{\epsilon}[ (\widehat{\psi}(k) + \widehat{\psi}(-k)^{*}) (\widehat{\psi}(k^{'}) + \widehat{\psi}(-k^{'})^{*}) ], 
\end{align*}
where $F_{1}(k,k^{'}) := F(k,k^{'}) + 2$. Note that $F_{1}(k,-k) = 0, k \in \T$. Then by using $(\ref{initialbound2})$ and $\eqref{forlaplace}$ with $p = 0$, we obtain
\begin{align*}
&\bigl| <W_{\epsilon,+}(t),J> - \epsilon \sum_{x \in \Z} e_{x}(\frac{t}{f_{\theta,s}(\epsilon)}) J(\epsilon x) \bigr| \\
&= \bigl| \epsilon \sum_{x \in \Z} \phi_{x} J(\epsilon x) \bigr| = \bigl| \epsilon \sum_{x \in \Z} \phi_{x} \int_{\R} dp ~ e^{2 \pi \sqrt{-1} p \epsilon x} \widetilde{J}(p) \bigr| \\
&= \bigl| \frac{\epsilon}{16} \int_{\R \times \T} dpdk ~ F_{1}(k,-k - \epsilon p) \mathbb{E}_{\epsilon}[ (\widehat{\psi}(k) + \widehat{\psi}(-k)^{*}) (\widehat{\psi}(-k - \epsilon p) + \widehat{\psi}(k + \epsilon p)^{*})] \widetilde{J}(p) \bigr| \\
&\lesssim \int_{\R \times \T} dpdk ~ |F_{1}(k,-k - \epsilon p)| ~ \epsilon \mathbb{E}_{\epsilon}[|\widehat{\psi}(k)|^{2}] ~ |\widetilde{J}(p)| \\
&\lesssim \int_{\R} dp ~ \bigl( \int_{\T} dk ~ |F_{1}(k,-k - \epsilon p)|^{2} \bigr)^{\frac{1}{2}} \bigl( \int_{\T} dk ~ \epsilon^{2} \mathbb{E}_{\epsilon}[|\widehat{\psi}(k)|^{2}]^{2} \bigr)^{\frac{1}{2}} |\widetilde{J}(p)| \\
&\lesssim \int_{\R} dp ~ \bigl( \int_{\T} dk ~ |F_{1}(k,-k - \epsilon p)|^{2} \bigr)^{\frac{1}{2}} |\widetilde{J}(p)|. 
\end{align*}
Hence if we show that $\bigl( \int_{\T} dk ~ |F_{1}(k,-k - \epsilon p)|^{2} \bigr)^{\frac{1}{2}}$ is bounded above by some positive constant uniformly in $0 < \epsilon << 1, p \in \R$, then by using the dominated convergence theorem we have 
\begin{align*}
\lim_{\epsilon \to 0} \int_{\R} dp ~ \bigl( \int_{\T} dk ~ |F_{1}(k,-k - \epsilon p)|^{2} \bigr)^{\frac{1}{2}} |\widetilde{J}(p)| = 0.
\end{align*}
Now we will estimate $| \a(k+k^{'}) - \a(k) - \a(k^{'})|, ~ k^{'} = - k - \epsilon p$. Since $| \sin^{2}{(y_{1} + y_{2})} - \sin^{2}{y_{2}} - \sin^{2}{y_{2}} | \lesssim |\sin{y_{1}} \sin{y_{2}}|, ~ y_{1},y_{2} \in \R$, we get
\begin{align*}
&| \a(k+k^{'}) - \a(k) - \a(k^{'})| \lesssim \sum_{x \in \N} \frac{|\sin{\pi k x} \sin{\pi k^{'} x}|}{|x|^{\theta}} \\
&= \int_{1}^{\infty} dy ~ \frac{|\sin{\pi k y} \sin{\pi k^{'} y}|}{|y|^{\theta}} + \sum_{x \in \N} \bigl( \int_{x}^{x+1} dy ~  \frac{|\sin{\pi k x} \sin{\pi k^{'} x}|}{|x|^{\theta}} - \frac{|\sin{\pi k y} \sin{\pi k^{'} y}|}{|y|^{\theta}} \bigr) \\
&\lesssim |k|^{\frac{\theta - 1}{2}}|k^{'}|^{\frac{\theta - 1}{2}} \int_{\sqrt{|k k^{'}|}}^{\infty} dy \frac{|\sin{\sqrt{\frac{k}{k^{'}}}y} \sin{\sqrt{\frac{k^{'}}{k}}y}|}{|y|^{\theta}} + \sum_{x \in \N} \int_{x}^{x+1} dy \bigl| \int_{x}^{y} dy^{'} ~ \partial_{y^{'}} ( \frac{\sin{k y^{'}} \sin{k^{'}y^{'}}}{|y^{'}|^{\theta}} ) \bigr|.
\end{align*}
If $\theta > 3$, then we have 
\begin{align*}
| \a(k+k^{'}) - \a(k) - \a(k^{'})| \lesssim \sum_{x \in \N} \frac{|\sin{\pi k x} \sin{\pi k^{'} x}|}{|x|^{\theta}} \lesssim |kk^{'}| \lesssim \omega(k)\omega(k^{'}),
\end{align*}
and we see that $F(k,k^{'})$ is uniformly bounded by some positive constant. When $2 < \theta < 3$, we obtain
\begin{align*}
&|k|^{\frac{\theta - 1}{2}}|k^{'}|^{\frac{\theta - 1}{2}} \int_{\sqrt{|k k^{'}|}}^{\infty} dy \frac{|\sin{\sqrt{\frac{k}{k^{'}}}y} \sin{\sqrt{\frac{k^{'}}{k}}y}|}{|y|^{\theta}} \lesssim |k|^{\frac{\theta - 1}{2}}|k^{'}|^{\frac{\theta - 1}{2}} \lesssim \omega(k)\omega(k^{'}), \\
&\int_{x}^{x+1} dy \bigl| \int_{x}^{y} dy^{'} ~ \partial_{y^{'}} ( \frac{\sin{k y^{'}} \sin{k^{'}y^{'}}}{|y^{'}|^{\theta}} ) \bigr| \\ 
&= \int_{x}^{x+1} dy \bigl| \int_{x}^{y} dy^{'}  \frac{k \cos{k y^{'}} \sin{k^{'}y^{'}}}{|y^{'}|^{\theta}} + \frac{k^{'} \sin{k y^{'}} \cos{k^{'}y^{'}}}{|y^{'}|^{\theta}} - \frac{\theta \sin{k y^{'}} \sin{k^{'}y^{'}}}{|y^{'}|^{\theta + 1}} \bigr|  \\
&\lesssim \frac{|kk^{'}|}{|x|^{\theta - 1}} \lesssim  \frac{ \omega(k)\omega(k^{'})}{|x|^{\theta-1}}
\end{align*}
and thus we have $| \a(k+k^{'}) - \a(k) - \a(k^{'})| \lesssim \omega(k)\omega(k^{'})$. Finally we consider the case $\theta = 3$. \textcolor{black}{Since
\begin{align*}
&|kk^{'}| \int_{\sqrt{|k k^{'}|}}^{\infty} dy \frac{|\sin{\sqrt{\frac{k}{k^{'}}}y} \sin{\sqrt{\frac{k^{'}}{k}}y}|}{|y|^{3}} \lesssim \bigl| kk^{'} \log|k k^{'}|  \bigr| + |kk^{'}|, \\
&\int_{x}^{x+1} dy \bigl| \int_{x}^{y} dy^{'} ~ \partial_{y^{'}} ( \frac{\sin{k y^{'}} \sin{k^{'}y^{'}}}{|y^{'}|^{\theta}} ) \bigr|  \lesssim \frac{ |kk^{'}|}{|x|^{2}},
\end{align*}
we obtain
\begin{align*}
&|F_{1}(k,k^{'})|^{2} \lesssim \bigl| \frac{\sqrt{|\log |k||}}{\sqrt{|\log |k^{'}||}} + \frac{\sqrt{|\log |k^{'}||}}{\sqrt{|\log |k||}} + \frac{1}{\sqrt{|\log |k| \log |k^{'}||}} + 1 \bigr|^{2} \\
&\quad \lesssim \frac{|\log |k||}{|\log |k^{'}||} + \frac{|\log |k^{'}||}{|\log |k||} + \frac{1}{|\log |k| \log |k^{'}||} + 1, \\
&\int_{\T} dk ~ \frac{|\log |k||}{|\log |k^{'}||} = \int_{\T} dk ~ \frac{|\log |k^{'}||}{|\log |k||} \lesssim \int_{\T} dk ~ |\log |k^{'}|| = \int_{\T} dk ~ |\log |k|| < \infty, \\
&\int_{\T} dk ~ \frac{1}{|\log |k| \log |k^{'}||} \lesssim \int_{\T} dk ~ \frac{1}{|\log |k^{'}||} = \int_{\T} dk ~ \frac{1}{|\log |k||} < \infty.
\end{align*}
Thus the proposition is proved.}

\section{Proof of \eqref{conv:deltaomega}}\label{app:delome}

From the definition of $(\delta_{\epsilon} \omega)_{\epsilon}$, we have
\begin{align*}
(\delta_{\epsilon} \omega)_{\epsilon}(k,p) = \begin{cases} \frac{4\pi \gamma_{0}^{- \frac{3 - \theta}{7 - \theta}} |p|^{\frac{3 - \theta}{7 - \theta}} (\frac{f_{\theta,s}(\epsilon)}{\epsilon^{s}})^{\frac{3 - \theta}{6}}}{\omega(k_{\epsilon}+\frac{\epsilon p}{2}) + \omega(k_{\epsilon}-\frac{\epsilon p}{2})} \sum_{x \ge 1} \frac{1}{|x|^{\theta - 1}} \frac{\sin \pi \epsilon p x}{\pi \epsilon p x} \sin 2 \pi k_{\epsilon} x \quad &2 < \theta < 3,\\
\frac{4\pi \{ - \log (\frac{f_{\theta,s}(\epsilon)}{\epsilon^{s}})^{\frac{1}{3}} \}^{-\frac{1}{2}}}{\omega(k_{\epsilon}+\frac{\epsilon p}{2}) + \omega(k_{\epsilon}-\frac{\epsilon p}{2})} \sum_{x \ge 1} \frac{1}{|x|^{2}} \frac{\sin \pi \epsilon p x}{\pi \epsilon p x} \sin 2 \pi k_{\epsilon} x \quad &\theta = 3, \\ \frac{4\pi}{\omega(k_{\epsilon}+\frac{\epsilon p}{2}) + \omega(k_{\epsilon}-\frac{\epsilon p}{2})} \sum_{x \ge 1} \frac{1}{|x|^{\theta - 1}} \frac{\sin \pi \epsilon p x}{\pi \epsilon p x} \sin 2 \pi k_{\epsilon} x \quad &\theta > 3. \end{cases}
\end{align*}

If $2 < \theta < 3$, then we have
\begin{align}\label{delome:23}
(\delta_{\epsilon} \omega)_{\epsilon}(k,p) &= \frac{1}{2} |k|^{- \frac{3 - \theta}{2}} \frac{2 |k_{\epsilon}|^{\frac{\theta - 1}{2}}}{\omega(k_{\epsilon}+\frac{\epsilon p}{2}) + \omega(k_{\epsilon}-\frac{\epsilon p}{2})} \frac{4 \pi}{|k_{\epsilon}|^{\theta - 2}} \sum_{x \ge 1} \frac{1}{|x|^{\theta - 1}} \sin 2 \pi k_{\epsilon} x \notag \\
& \quad + \frac{1}{2} \gamma_{0}^{\frac{2(\theta - 2)}{7 - \theta}} |p|^{-\frac{2(\theta - 2)}{7 - \theta}} \operatorname{sgn}(k) |k|^{- \frac{\theta - 1}{2}} \frac{2 |k_{\epsilon}|^{\frac{\theta - 1}{2}}}{\omega(k_{\epsilon}+\frac{\epsilon p}{2}) + \omega(k_{\epsilon}-\frac{\epsilon p}{2})} \notag \\
& \quad \quad \times \frac{4 \pi}{(\frac{f_{\theta,s}(\epsilon)}{\epsilon^{s}})^{\frac{\theta - 2}{3}}} \sum_{x \ge 1} \frac{1}{|x|^{\theta - 1}} (\frac{\sin \pi \epsilon p x}{\pi \epsilon p x} - 1) \sin 2 \pi k_{\epsilon} x.
\end{align}
From \eqref{asympofa}, the first term of \eqref{delome:23} converges to $\operatorname{sgn}(k) \frac{(\theta - 1)\sqrt{C(\theta)}}{2} |k|^{- \frac{3 - \theta}{2}}$ for any $(p,k) \neq (0,0)$. Now we will estimate the second term of \eqref{delome:23}. Fix a number $0 < a < \theta - 2$. Since there exists some constant $C_{a} > 0$ such that $\bigl| \frac{\sin x}{x} - 1 \bigr| \le C_{a} |x|^{a}$ for any $x \in \R$ and $\lim_{\epsilon \to 0} \epsilon (\frac{f_{\theta,s}(\epsilon)}{\epsilon^{s}})^{- \frac{1}{3}} = 0$, we obtain
\begin{align*}
&\bigl| (\frac{f_{\theta,s}(\epsilon)}{\epsilon^{s}})^{- \frac{\theta - 2}{3}} \sum_{x \ge 1} \frac{1}{|x|^{\theta - 1}} (\frac{\sin \pi \epsilon p x}{\pi \epsilon p x} - 1) \sin 2 \pi k_{\epsilon} x \bigr| \\ 
&\le C_{a} (\frac{f_{\theta,s}(\epsilon)}{\epsilon^{s}})^{- \frac{\theta - 2}{3}} \sum_{x \ge 1} \frac{1}{|x|^{\theta - 1 - a}} | \sin 2 \pi k_{\epsilon} x | \\
&\le C_{a,\theta} (\frac{f_{\theta,s}(\epsilon)}{\epsilon^{s}})^{- \frac{\theta - 2}{3}} \epsilon^{a} |k_{\epsilon}|^{\theta - 2 - a} \\
&\le C_{a,\theta,p,k} \bigl[ \epsilon (\frac{f_{\theta,s}(\epsilon)}{\epsilon^{s}})^{- \frac{1}{3}} \bigr]^{a} \to 0 \quad \epsilon \to 0,
\end{align*}
where $C_{a,\theta}, C_{a,\theta,p,k}$ are some constants which depend on the variables in their subscripts. 

If $\theta \ge 3$, then we can use almost the same discussion and thus we obtain \eqref{conv:deltaomega}.

\section{Proof of (\ref{toGamma})}\label{ap:toGamma}

First we observe
\begin{align*}
\begin{cases} \int_{\R} dy ~ (1-\cos y) |y|^{-\frac{6}{7-\theta} - 1} = \frac{7 - \theta}{3} \int_{0}^{\infty} dy ~ (\sin{y}) |y|^{-\frac{6}{7 - \theta}} &2 < \theta \le 3, \\
\int_{\R} dy ~ (1-\cos y) |y|^{-\frac{3}{2} - 1} = \frac{4}{3} \int_{0}^{\infty} dy ~ (\sin{y}) |y|^{-\frac{3}{2}}  &\theta > 3. \end{cases} 
\end{align*}
Hence it is sufficient to show that for any $1 < a < 2$,
\begin{align*}
\int_{0}^{\infty} dy ~ (\sin{y}) |y|^{-a} = \cos{\frac{a \pi}{2}} \Gamma(1-a).
\end{align*}
\textcolor{black}{For any positive constant $0 < b << 1$, we have
\begin{align*}
&\int_{b}^{\infty} dy ~ (\sin{y}) |y|^{-a} = |b|^{-a + 1} \int_{0}^{\infty} dy ~ \bigl( \sin{b(y+1)} \bigr) |y+1|^{-a} \\
&= \frac{ \sqrt{-1}}{2} |b|^{-a + 1} \int_{0}^{\infty} dy ~ e^{- \sqrt{-1}b(y+1)} - e^{\sqrt{-1}b(y+1)} |y+1|^{-a} \\
&= \frac{ \sqrt{-1}}{2} |b|^{-a + 1} \bigl( e^{-\sqrt{-1} b} \Psi(1,2-a;\sqrt{-1}b) - e^{\sqrt{-1} b} \Psi(1,2-a;-\sqrt{-1}b) \bigr) \\
&= \frac{\sqrt{-1}}{2} |b|^{-a + 1} \bigl( \Psi(1,2-a;\sqrt{-1}b) - \Psi(1,2-a;-\sqrt{-1}b) \bigr) + O(|b|^{2-a}), 
\end{align*} 
where $\Psi(~\cdot~,~\cdot~;~\cdot~)$ is the confluent hypergeometric function of the second kind. For the definition and the property of the confluent hypergeometric function, see \cite[Section 9]{L}. From the relationship between the confluent hypergeometric function $\Phi(~\cdot~,~\cdot~;~\cdot~)$ and $\Psi(~\cdot~,~\cdot~;~\cdot~)$, we have
\begin{align*}
\Psi(1,2-a;\pm \sqrt{-1}b) &= \frac{\Gamma(1+a)}{\Gamma(a)} \Phi(1,2-a;\pm \sqrt{-1}b) + \Gamma(1-a) (\pm \sqrt{-1} b)^{a-1} \Phi(a,a;\pm \sqrt{-1} b) \\
&= \frac{\Gamma(1+a)}{\Gamma(a)} \Phi(1,2-a;\pm \sqrt{-1}b) + \Gamma(1-a) (\pm \sqrt{-1} b)^{a-1} e^{\pm \sqrt{-1} b}.
\end{align*}
Since $|\Phi(1,2-a; \sqrt{-1}b) - \Phi(1,2-a; -\sqrt{-1}b)| = O(|b|)$, by taking the limit $b \to 0$ we have
\begin{align*}
\int_{0}^{\infty} dy ~ (\sin{y}) |y|^{-a} &= \frac{1}{2} \bigl( (\sqrt{-1})^{a} + (-\sqrt{-1})^{a} \bigr) \Gamma(1-a) \\
&= \cos{\frac{a \pi}{2}} \Gamma(1-a).
\end{align*}}


\begin{thebibliography}{21}
\bibitem{B} {\sc D. \ Bagchi } : {\em Energy transport in the presence of long-range interactions }. 
Phys. Rev. E \textbf{96}, 042121 (2017)
\bibitem{B2} {\sc D. \ Bagchi } : {\em Thermal transport in the Fermi-Pasta-Ulam model with long-range interactions}.  
Phys. Rev. E \textbf{95}, 032102 (2017)
\bibitem{BBO} {\sc G. \ Basile, C. \ Bernardin, S. \ Olla} : {\em Thermal Conductivity for a Momentum Conservative Model}.
Commun. \ Math. \ Phys. \textbf{287}, 67-98 (2009)
\bibitem{BOS} {\sc G. \ Basile, S. \ Olla, H. \ Spohn} : {\em Energy transport in stochastically perturbed lattice dynamics}.
Arch. \ Ration. \ Mech. \textbf{195}, 171--203 (2009) 
\bibitem{BO} {\sc C. \ Bernardin, S. \ Olla } : {\em Fourier's Law for a microscopic model of heat conduction }. 
J. Stat. Phys. \textbf{121}, 271--289 (2005)
\bibitem{D} {\sc A. \ Dhar} : {\em Heat transport in low-dimensional systems}.
Adv. \ Phys. \textbf{57}(5), 457--537 (2008)
\bibitem{GV} {\sc I. \ M. \ Gelfand, N. \ Ya. \ Vilenkin} : {\em Generalized Functions volume 4}.
Academic Press, New York (1964)
\bibitem{ICLLC} {\sc S. \ Iubini, P. \ Di Cintio, S. \ Lepri, R. \ Livi, and L. \ Casetti } : {\em Heat transport in oscillator chains with long-range interactions coupled to thermal reservoirs}. Phys. Rev. E \textbf{97}, 032102 (2018)
\bibitem{KJO} {\sc M. \ Jara, T. \ Komorowski,  S. \ Olla} : {\em A limit theorem for an additive functionals of Markov chains}.
Ann. \ Appl. \ Probab. \textbf{19}, 2270--2230 (2009)
\bibitem{JKO} {\sc M. \ Jara, T. \ Komorowski, S. \ Olla} : {\em Superdiffusion of Energy in a Chain of Harmonic Oscillators with Noise}. Commun. \ Math. \ Phys. \textbf{339}, 407--453 (2015)
\bibitem{KO2} {\sc T. \ Komorowski, S. \  Olla}, {\em Ballistic and superdiffusive scales in macroscopic evolution of a chain of oscillators}, Nonlinearity \textbf{29}, 962--999 (2016),
\bibitem{KO} {\sc T. \ Komorowski, S. \  Olla}, {\em Diffusive Propagation of Energy in a Non-acoustic Chain}, Arch. Rational Mech. Anal., \textbf{223} 95--139 (2017)
\bibitem{KOS} {\sc T. \ Komorowski, S. \  Olla, M. \ Simon} : {\em Macroscopic evolution of mechanical and thermal energy in a harmonic chain with random flip of velocities }. 
Kinetic and Related Models \textbf{11}(3) 615--645 (2018)
\textcolor{black}{\bibitem{L} {\sc N. \ N. \ Lebedev} : {\em Special functions and their applications}. 
Prentice-Hall, inc. (1965)}
\bibitem{Ls} {\em Thermal Transport in Low Dimensions : From Statistical Physics to Nanoscale Heat Transfer}.
edited by S. Lepri (Springer, New York, 2016)
\bibitem{LLPb} {\sc S. \ Lepri, R. \ Livi, A. \ Politi} : {\em Heat conduction in chains of nonlinear oscillators}. 
Phys. Rev. Lett. \textbf{78}, 1896--1899 (1997)
\bibitem{LLP} {\sc S. \ Lepri, R. \ Livi, A. \ Politi} : {\em Thermal conduction in classical low-dimensional lattices}.
Phys. Rep. \textbf{377}(1) 1--80 (2003)
\bibitem{MA} {\sc A. \ Mellet, S. Merino-Aceituno} : {\em Anomalous Energy Transport in FPU-$\beta$ Chain}. J. Stat. Phys. \textbf{160}(3), 583--621 (2015)
\bibitem{S} {\sc H. \ Spohn} : {\em Nonlinear fluctuating hydrodynamics for anharmonic chains}. J. Stat. Phys. \textbf{154}(5), 1191-1227 (2014)
\textcolor{black}{\bibitem{HS} {\sc H. \ Suda} : {\em Superballistic and superdiffusive scaling limits of stochastic harmonic chains with long-range interactions}. in preparation.
\bibitem{TS} {\sc S. \ Tamaki, K. \ Saito} : {\em Energy current correlation in solvable long-range interacting systems}. arXiv:1906.08457}
\end{thebibliography}
\end{document}